\documentclass[reqno, 12pt]{amsart}

\usepackage{amsmath}
\usepackage{amsfonts}
\usepackage{amssymb}
\usepackage{amsthm}
\usepackage{bbm}
\usepackage{graphicx, color}
\usepackage{tikz}
\usetikzlibrary{cd, intersections, calc, decorations.pathmorphing, arrows, decorations.pathreplacing}

\usepackage{mathtools}

\usepackage{enumerate}

\definecolor{mygreen}{rgb}{0,0.7,0.3}
\definecolor{myblue}{rgb}{0,0.50,1.20}
\definecolor{myorange}{rgb}{1,0.38,0}

\usepackage{subfiles}
\usepackage[colorlinks, linkcolor=blue, citecolor=red, urlcolor=cyan]{hyperref}

\numberwithin{equation}{section}

\usepackage{cleveref}
\crefname{thm}{Theorem}{Theorems}
\crefname{cor}{Corollary}{Corollaries}
\crefname{lem}{Lemma}{Lemmas}
\crefname{sublem}{Sublemma}{Sublemmas}
\crefname{prop}{Proposition}{Propositions}
\crefname{dfn}{Definition}{Definitions}
\crefname{defi}{Definition}{Definitions}
\crefname{ex}{Example}{Examples}
\crefname{claim}{Claim}{Claims}
\crefname{conj}{Conjecture}{Conjectures}
\crefname{conv}{Notation}{Notations}
\crefname{rem}{Remark}{Remarks}
\crefname{rmk}{Remark}{Remarks}
\crefname{prob}{Problem}{Problems}
\crefname{figure}{Figure}{Figures}
\crefname{table}{Table}{Tables}
\crefname{section}{Section}{Sections}
\crefname{subsection}{Section}{Sections}
\crefname{appendix}{Appendix}{Appendices}

\newtheorem{thm}{Theorem}[section]
\newtheorem{prop}[thm]{Proposition}
\newtheorem{cor}[thm]{Corollary}
\newtheorem{lem}[thm]{Lemma}

\theoremstyle{definition}
\newtheorem{dfn}[thm]{Definition}
\newtheorem{defi}[thm]{Definition}
\newtheorem{ex}[thm]{Example}

\newtheorem{conj}[thm]{Conjecture}
\newtheorem{conv}[thm]{Notation}
\theoremstyle{remark}
\newtheorem{rmk}[thm]{Remark}
\newtheorem{rem}[thm]{Remark}

\newcommand{\CORRECTED}
{}

\newcommand{\too}{\longrightarrow}

\newcommand*{\chom}{\mathcal{H}\kern -.5pt om}

\newcommand{\bZ}{\mathbb{Z}}
\newcommand{\bQ}{\mathbb{Q}}
\newcommand{\bR}{\mathbb{R}}

\newcommand{\bS}{\mathbb{S}}
\newcommand{\bT}{\mathbb{T}}

\newcommand{\bD}{\mathbb{D}}

\newcommand{\bE}{\mathbb{E}}

\newcommand{\bs}{{\boldsymbol{s}}}
\newcommand{\indedge}{t \overbar{\indk} t'}

\newcommand{\indi}{i}

\newcommand{\indk}{k}
\newcommand{\numi}{i}
\newcommand{\numj}{j}
\newcommand{\numk}{k}
\newcommand{\spe}{|M_\partial|}

\newcommand{\cTr}{\mathsf{G}}

\newcommand{\A}{\mathcal{A}}
\newcommand{\cA}{\mathcal{A}}

\newcommand{\cC}{\mathcal{C}}

\newcommand{\cE}{\mathcal{E}}
\newcommand{\cF}{\mathcal{F}}

\newcommand{\cH}{\mathcal{H}}
\newcommand{\cI}{\mathcal{I}}

\newcommand{\cL}{\mathcal{L}}

\newcommand{\cO}{\mathcal{O}}

\newcommand{\cR}{\mathcal{R}}
\newcommand{\cS}{\mathcal{S}}

\newcommand{\cU}{\mathcal{U}}

\newcommand{\X}{\mathcal{X}}
\newcommand{\cX}{\mathcal{X}}

\newcommand{\cZ}{\mathcal{Z}}

\newcommand{\fC}{\mathfrak{C}}
\newcommand{\fF}{\mathfrak{F}}
\newcommand{\fS}{\mathfrak{S}}

\newcommand{\tri}{\triangle}
\newcommand{\sgn}{\mathrm{sgn}}
\newcommand{\trop}{\mathrm{trop}}
\newcommand{\pos}{\mathbb{R}_{>0}}
\newcommand{\stab}{\mathrm{stab}}

\newcommand{\ML}{\mathcal{ML}}

\newcommand{\eML}{\widehat{\mathcal{ML}}}

\newcommand{\Tri}{\mathrm{Tri}^{\bowtie}}
\newcommand{\bTri}{\bT \mathrm{ri}^{\bowtie}}

\newcommand{\Exch}{\mathrm{Exch}}
\newcommand{\bExch}{\bE \mathrm{xch}}
\newcommand{\uf}{\mathrm{uf}}
\newcommand{\f}{\mathrm{f}}

\newcommand{\triv}{\mathrm{triv}}
\newcommand{\Teich}{Teichm\"uller}
\newcommand{\grp}{\mathrm{grp}}

\newcommand{\hL}{\widehat{L}}
\newcommand{\redtext}[1]{\textcolor{mygreen}{#1}}

\DeclareMathOperator{\im}{\mathrm{im}}

\DeclareMathOperator{\interior}{\mathrm{int}}

\newcommand{\oline}[1]{\overline{#1}}

\makeatletter
\newcommand{\oset}[3][0ex]{%
  \mathrel{\mathop{#3}\limits^{
    \vbox to#1{\kern-2\ex@
    \hbox{$\scriptstyle#2$}\vss}}}}
\makeatother
\newcommand{\overbar}[1]{\oset{#1}{-\!\!\!-\!\!\!-}}

\makeatletter
\newcommand{\osetnear}[3][0ex]{%
  \mathrel{\mathop{#3}\limits^{
    \vbox to#1{\kern-.3\ex@
    \hbox{$\scriptstyle#2$}\vss}}}}
\makeatother
\newcommand{\overbarnear}[1]{\osetnear{#1}{-\!\!\!-\!\!\!-}}

\newcommand\qarrow[2]{\draw[->,shorten >=4pt,shorten <=4pt] (#1) -- (#2) [thick];} 

\tikzset{
  mid arrow/.style={postaction={decorate,decoration={
        markings,
        mark=at position .5 with {\arrow[#1]{stealth}}
      }}},
}

\tikzset{
  symbol/.style={
    draw=none,
    every to/.append style={
      edge node={node [sloped, allow upside down, auto=false]{$#1$}}}
  }
}

\title[Sign stability of mapping classes on marked surfaces II]{Sign stability of mapping classes on marked surfaces II: general case via reductions}

\author[Tsukasa Ishibashi]{Tsukasa Ishibashi}
\address{Tsukasa Ishibashi, Research Institute for Mathematical Sciences, Kyoto University, Kitashirakawa Oiwake-cho, Sakyo-ku, Kyoto 606-8502, Japan.}
\email{ishiba@kurims.kyoto-u.ac.jp}

\author[Shunsuke Kano]{Shunsuke Kano}
\address{Shunsuke Kano, Department of Mathematics, Tokyo Institute of Technology, 2-12-1 Ookayama, Meguro, Tokyo 152-8551, Japan.}
\email{kano.s.ab@m.titech.ac.jp}

\date{\today}

\begin{document}
\maketitle

\begin{abstract}
We give a cluster algebraic description of the reduction procedure of mapping classes along a multicurve.  
Based on this description, we characterize pseudo-Anosov mapping classes on a general marked surface in terms of a weaker version of the uniform sign stability, generalizing the main result in the previous paper \cite{IK20}. Moreover we axiomatize a general reduction procedure of mutation loops parametrized by a rational polyhedral cone in the tropical cluster $\X$-variety, which includes both the reduction along a multicurve and the cluster reduction introduced in \cite{Ish19}.
\end{abstract}

\tableofcontents

\section{Introduction}
This paper continues the study of the sign stability of mutation loops in the cluster modular group $\Gamma_\Sigma$
associated with a marked surface\footnote{For a slight difference between the groups $\Gamma_\Sigma$ and  $\Gamma_{\bs_\Sigma}$ for a fourth-punctured sphere, see \cite[Theorem 4.5]{IK20}.}. The notion of the \emph{sign stability} of a mutation loop (or more precisely, its representation path) has been introduced in \cite{IK19}, aiming at generalizing the pseudo-Anosov property of a mapping class on a surface. In our previous paper \cite{IK20}, we confirmed the following:

\begin{thm}[{\cite[Corollary 7.7]{IK20}}]\label{prop:pA_NS_SS}
Let $\Sigma$ be a punctured surface, that is, a marked surface without boundary.
For a mutation loop $\phi \in \Gamma_\Sigma$, the following conditions are equivalent:
\begin{enumerate}
    \item The mutation loop $\phi$ is pseudo-Anosov.
    \item The mutation loop $\phi$ is uniformly sign-stable.
    \item The mutation loop $\phi$ has NS dynamics on $\bS \cX_{\Sigma}^\uf(\bR^\trop)$, and its attracting and repelling points are $\cX$-filling.
\end{enumerate}
In this case, the cluster stretch factor of the mutation loop $\phi$ coincides with the stretch factor of the underlying pseudo-Anosov mapping class of $\phi$.
\end{thm}
Combining with the main result in \cite{IK19}, we have proved that the algebraic entropies of the cluster $\A$- and $\X$-transformations induced by a pseudo-Anosov mapping class both coincide with the topological entropy. 
These results are satisfactory to assert that the notion of sign stability is a correct generalization of the pseudo-Anosov property.

In this paper, we investigate the sign stability of the mapping classes on a general marked surface,  which is a compact oriented surface with boundary equpped with a finite set of marked points. 
One of our motivation is a possible connection to the theory of quiver representations, where the theory is rather well-developed for the marked surfaces with non-empty boundary.
For instance, the quiver with potential associated with an ideal triangulation of a marked surface $\Sigma$ with non-empty boundary is rigid \cite[Theorem 31]{LF09}, while it is typically not rigid for a punctured surface \cite{Lad12} except for a special choice of the potential data.

A mapping class $\phi$ on a marked surface $\Sigma$ naturally induces a mapping class $\pi(\phi)$ on the punctured surface $\bar{\Sigma}$ obtained from $\Sigma$ by shrinking each boundary component to a puncture. We say that $\phi$ is pseudo-Anosov if the induced mapping class $\pi(\phi)$ is pseudo-Anosov. We are interested in a cluster algebraic characterization of pseudo-Anosov mapping classes in this sense.

\subsection{Reduction of mutation loops}
Among other ways, we can view the surface $\bar{\Sigma}$ as a connected component of the marked surface obtained by cutting $\Sigma$ along the multicurve $\cR_\partial$ consisting of boundary-parallel curves.
This point of view lead us to the study of the classical reduction procedure of mapping classes along a multicurve $\cR=\{C_k \}_{k=1}^K$ due to Thurston in the framework of the cluster algebra, and especially its relation to the sign stability.

Let $\Sigma_\cR$ denote the marked surface obtained by cutting $\Sigma$ along $\cR$. 
A mapping class $\phi$ which fixes $\cR$ (but may permute the component curves) induces a mapping class $\pi_\cR^\grp(\phi)$ on $\Sigma_\cR$. Then we have a group homomorphism $\pi_\cR=\pi_\cR^\grp:\Gamma_{\Sigma,\cR} \to \Gamma_{\Sigma_\cR}$, where the former denotes the subgroup of $\Gamma_\Sigma$ consisting of mutation loops which fix $\cR$. 
The description of a point of $\X_\Sigma^\uf(\bR^\trop)$ as a measured geodesic lamination allows us to define a projection $\pi_\cR=\pi_\cR^\trop:\X_\Sigma^\uf(\bR^\trop) \to \X_{\Sigma_\cR}^\uf(\bR^\trop)$.

We first observe that a tagged triangulation $\tri$ of $\Sigma$ induces that $\oline{\tri}$ of $\Sigma_\cR$ via the map $\pi_\cR^\trop$. Moreover, it can both happen that a flip $f_\alpha:\tri \to \tri'$ does not change $\oline{\tri}$ or induces a non-trivial flip $f_{\oline{\alpha}}:\oline{\tri} \to \oline{\tri'}$. These two possibilities are distinguished by the shear coordinates of the curves in $\cR$ along the arc $\alpha$ (\cref{lem:collapsed_tri}). 
Although this observation does not tell us how to project the vertical edges in the exchange graph, we can describe a way 
to construct a representation path $\oline{\gamma}$ of $\pi_\cR^\grp(\phi) \in \Gamma_{\Sigma_\cR}$ from a given representation path $\gamma$ of $\phi \in \Gamma_{\Sigma,\cR}$ (\cref{prop:pi_MCG}).
This gives a cluster algebraic description of the reduction homomorphism $\pi_\cR^\grp$. 
Our result is summarized as follows:

\begin{thm}[
\cref{cor:reduction_spec}]\label{introthm:reduction_spec}
    If a path $\gamma:(\tri_0,\ell_0) \to (\tri,\ell)$ in $\bExch_\Sigma$ represents a mutation loop $\phi \in \Gamma_{\Sigma,\cR}$ and sign-stable on an $\bR_{>0}$-invariant subset $\Omega \subset \X_{(\tri_0,\ell_0)}^\uf(\bR^\trop)$, then the path $\oline{\gamma}$ is sign-stable on $\pi_\cR(\Omega) \subset \X_{(\oline{\tri_0},\lambda_0)}^\uf(\bR^\trop)$. Moreover we have $\rho(E_{\phi,\Omega}^{(\tri_0,\ell_0)})= \rho(E_{\pi_\cR(\phi),\pi_\cR(\Omega)}^{(\oline{\tri_0},\lambda_0)})$. Here $\lambda_0$ is an arbitrary labeling  of $\oline{\tri}_0$.

\end{thm}

\subsection{Sign stability of mapping classes on a general marked surface}
For a marked surface $\Sigma$, let $\bar{\Sigma}$ be the punctured surface obtained from $\Sigma$ by shrinking each boundary component to a puncture. Let $\pi:\Gamma_\Sigma \to \Gamma_{\bar{\Sigma}}$ denote the $\Gamma_{\bar{\Sigma}}$-component of the reduction homomorphism $\pi_{\cR_\partial}^\grp$.
First we prove that a pseudo-Anosov mapping class in this sense induces a North-South dynamics on the space of measured geodesic laminations off a certain subset:

\begin{thm}[\cref{thm:pA_NS_general}]\label{introthm:NS dynamics}
A mutation loop $\phi \in \Gamma_{\Sigma}$ is pseudo-Anosov if and only if 
there exist two points $L^\pm_{\phi} \in \iota(\cU^\uf_{\bar{\Sigma}}(\bR^\trop)) \subset \cX^\uf_\Sigma(\bR^\trop)$ such that
\begin{align}
    \lim_{n \to \infty} \phi^{\pm n}([\hL]) = [L^\pm_{\phi}] \quad \mbox{in }\  \bS\X_\Sigma^\uf(\bR^\trop)
\end{align}
for any $\hL \in \cX^\uf_\Sigma(\bR^\trop) \setminus \pi^{-1}(\bR_{\geq 0}\cdot \pi(L^\mp_{\phi}))$.
\end{thm}
This is a generalization of \cref{prop:pA_NS_SS} $(1) \Longleftrightarrow (3)$. 
While this result supports the naturality of our definition of the pseudo-Anosov mapping class on a marked surface, it turns out that such a mapping class may fail to be sign-stable (\cref{ex:3bdries+1pct_sph}).
A crucial reason for this is that the stable and unstable laminations $L^\pm_{\phi}$ can have vanishing shear coordinates and thus non-strict signs, unlike the punctured case (\cite[Proposition 2.14]{IK20}).

Based on these observations, we introduce a weaker notion of stability called the \emph{weak sign stability} (\cref{def:weak_SS}). It requires the stability of signs of a representation path in a weaker sense. 
Combining with the observations on the reductions of representation paths, we also introduce the notion of \emph{$\cC$-hereditariness} (\cref{d:hereditary}) for a rational polyhedral cone $\cC \subset \X_\bs^\uf(\bR^\trop)$. When $\cC=\cC(\cR_\partial)$ is generated by the curves in $\cR_\partial$, it turns out that the $\cC$-hereditariness of a representation path $\gamma$ of $\phi \in \Gamma_\Sigma$ is nothing but the sign stability of the reduced path $\oline{\gamma}$ obtained via the reduction along $\cR_\partial$. Then we obtain the following characterization of the pseudo-Anosov property, which generalizes \cref{prop:pA_NS_SS} $(1) \Longleftrightarrow (2)$:

\begin{thm}[\cref{thm:weak_SS}]\label{introthm:weak SS}
Let $\Sigma$ be a marked surface, and $\phi \in \Gamma_\Sigma$ a mutation loop.
Then, $\phi$ is pseudo-Anosov if and only if any representation path $\gamma: (\tri, \ell) \to (\tri', \ell')$ of $\phi$ is weakly sign-stable on $\Omega_{(\tri, \ell)}^\bQ \setminus D_{\cX, \Sigma}$ and $\cC(\cR_\partial)$-hereditary.
\end{thm}
Unfortunately, we have not yet obtained an analogue of a cluster stretch factor which is intrinsic to a mutation loop admitting a weakly sign-stable representation path. The reason is that the presentation matrices in the stable range do not satisfy the Perron--Frobenius property \cite[Theorem 5.11]{IK20}.
Essentially due to this, we have not yet obtained a satisfactory generalization of \cite[Theorem 1.3]{IK20} on the coincidence of the algebraic and topological entropies of a pseudo-Anosov mapping class on a punctured surface. Nevertheless, we have the following partial estimate:

\begin{thm}[\cref{cor:trop/alg-entropy}]\label{introthm:entropy_2}
For any pseudo-Anosov mutation loop $\phi$ on a marked surface $\Sigma$, We have 
\begin{align}\label{eq:inequality_weak_SS}
\cE_\phi^x \geq \log \lambda_{\pi(\phi)}=\cE_{\pi(\phi)}^{\mathrm{top}}.
\end{align}
\end{thm}
See \cref{subsec:discussion_A-transf} for an estimate of the algebraic entropy of the cluster $\A$-transformation. The authors do not know any examples with strict inequality in \eqref{eq:inequality_weak_SS}. 

\subsection{Cluster reduction and its generalization}
A cluster algebraic analogue of the reduction along a multicurve, called the \emph{cluster reduction}, has been proposed in \cite{Ish19}.
It works for an arbitrary seed pattern $\bs$, and we obtain a new seed pattern by \lq\lq freezing" the mutation directions corresponding to a fixed simplex $S$ in the Fomin--Zelevinsky cluster complex $\Delta_\bs^\mathrm{FZ}$ (\cite{FST}). The dual graph $\cTr \subset \Exch_\bs$ of the star of $S$ is identified with the exchange graph of the new seed pattern $\bs|_\cTr$. We can consider the subgroup $\Gamma_\bs^\cTr \subset \Gamma_\bs$ consisting of the mutation loops that admit representation paths \lq\lq inside $\cTr$", and get an injective homomorphism $\pi_\cTr: \Gamma_\bs^\cTr \to \Gamma_{\bs|_\cTr}$ as well as a projection $\pi_\cTr:\X_\bs^\uf(\bR^\trop) \to \X_{\bs|_\cTr}^\uf(\bR^\trop)$. We investigate the effect of this cluster reduction procedure on the sign stability, and obtain the following:
\begin{thm}[\cref{prop:SS_cluster_reduction}]
Let $\gamma: v \to v'$ be a path contained in $\cTr$ which represents a mutation loop in $\Gamma_\bs^\cTr \subset \Gamma_\bs$. Then the path $\pi_\cTr(\gamma)$ is sign-stable on an $\bR_{>0}$-invariant set $\Omega \subset \X^\uf_{\bs|_\cTr}(\bR^\trop)$ if and only if $\gamma$ is sign-stable on $\widetilde{\Omega}:=\pi_\cTr^{-1}(\Omega\setminus \{0\}) \subset \X^\uf_\bs(\bR^\trop)$. Moreover we have 
\begin{align*}
    \rho(E_{\phi,\widetilde{\Omega}}^{(v_0)})=\rho(E_{\pi_{\cTr}(\phi),\Omega}^{(v_0)}),
\end{align*}
where $\phi:=[\gamma]_\bs \in \Gamma_\bs^{\cTr}$.
\end{thm}
This should be compared with \cref{introthm:reduction_spec}. In the surface case, the cluster reduction corresponds to the cutting operation of a marked surface along a multiarc. Therefore it is natural to expect a general framework both containing the reduction along a mutlicurve and the cluster reduction. 

Our observation is that the both reduction procedures can be parametrized by rational polyhedral cones $\cC \subset \X_\bs^\uf(\bR^\trop)$.
Indeed, the reduction along a multicurve $\cR$ is obviously related to the cone $\cC(\cR)$ generated by the curves in $\cR$ regarded as integral laminations, and the input data of the cluster reduction (\emph{i.e.}, a simplex of the Fomin--Zelevinsky cluster complex) is in one-to-one correspondence with the cones in the geometric realization of the Fock--Goncharov cluster complex (\cref{prop:cTr<->cone}), which can be regarded as a subset of $\cX_\bs^\uf(\bR^\trop)$.

Based on these observations, we axiomatize the \emph{$\cC$-reduction} (\cref{d:C-reduction}) for a rational polyhedral cone $\cC \subset \X_\bs^\uf(\bR^\trop)$ so that it includes both the reduction along a multicurve and the cluster reduction. In particular, we can consider the reduction along a collection of curves containing both ideal arcs and simple closed curves (\cref{prop:mixed reduction}). In view of the fact that the reduction data for the cluster reductions are precisely the strata of the \emph{special completion} introduced by Fock--Goncharov \cite{FG16}, it would be interesting to think about a generalized completion of cluster varieties. 







\subsection*{Organization of the paper}
In \cref{sec:sign stability}, we briefly recall the notation regarding cluster ensembles, the seed pattern associated with a marked surface, and the sign stability. For a detail, the reader is referred to \cite[Sections 2--5]{IK20}.

In \cref{sec:reduction}, the reduction procedure of mutation loops along a multicurve is studied. 
Not only that the case where the multicurve consists of boundary-parallel curves is the basis for the study in \cref{sec:weak_SS}, but also some of the observations obtained in this section lead us to the axiom of the $\cC$-reduction (\cref{d:C-reduction}).

In \cref{sec:weak_SS}, the notions of weak sign stability and hereditariness are introduced. \cref{introthm:NS dynamics,introthm:weak SS} are proved. 

In \cref{subsec:entropy}, we discuss the estimate of the algebraic entropy of the pseudo-Anosov mutation loops. 

In \cref{sec:cluster_reduction}, we recall the cluster reduction procedure from \cite{Ish19} and study its effect on the sign stability. As a slight generalization of \cite[Proposition 6.1]{IK20}, we prove the sign stability of a shortest path for a cluster Dehn twist.

Finally in \cref{sec:C-reduction}, we formulate a \emph{$\cC$-reduction} of a general cluster variety.

\subsection*{Acknowledgements}
T. I. would like to express his gratitude to his former supervisor Nariya Kawazumi for his continuous guidance and encouragement in the earlier stage of this work. 
S. K. is also deeply grateful to his supervisor Yuji Terashima for his advice and giving him a lot of knowledge.

T. I. is partially supported by JSPS KAKENHI Grant Numbers 18J13304 and 20K22304, and the Program for Leading Graduate Schools, MEXT, Japan.

\section{Sign stability of mutation loops}\label{sec:sign stability}
In this section, we recall the terminologies around the sign-stability of a representation path of a mutation loop from \cite{IK19}. 
Here we assuming that the notion in \cite[Appendix A]{IK20}.

\subsection{The exchange graph}
Let us fix a finite set $I = I_\uf \sqcup I_\f$ and a seed pattern $\bs: \bT_{I_\uf} \ni t \mapsto (N^{(t)}, B^{(t)})$.
Considering the graph $\bE_I$ with
\begin{itemize}
    \item vertices given by the pairs $(t,\sigma)$ for $t \in \bT_{I_\uf}$ and $\sigma \in \fS_I$, and
    \item edges of the following two types:
    \begin{itemize}
        \item (\emph{horizontal edges})
        For each $\indedge$ in $\bT_{I_\uf}$ and $\sigma \in \fS_I := \fS_{I_\uf} \times \fS_{I_\f}$, there is an edge $(t, \sigma) \overbar{\sigma(\numk)} (t', \sigma)$. 
        \item (\emph{vertical edges})
        For each vertex $t \in \bT_{I_\uf}$, $\sigma \in \fS_I$ and a transposition $(\numi \ \numj) \in \fS_I$, there is an edge
        $(t, \sigma) \overbar{(\numi\ \numj)} (t, (\numi\ \numj) \sigma)$.
    \end{itemize}
\end{itemize}
Then, the seed pattern $\bs$ lifts to a labeled seed pattern $\widetilde{\bs}: \bE_I \ni (t,\sigma) \mapsto (N^{(t, \sigma)},B^{(t, \sigma)})$ defined as
\begin{align*}
    N^{(t, \sigma)} := \bigoplus_{\numi \in I} \bZ e^{(t,\sigma)}_\numi, \quad
    B^{(t, \sigma)} := \sigma.B^{(t)}.
\end{align*}
We define a equivalence relation $\sim_\triv$ on $\bE_I$ as follows:
For two vertices $(t,\sigma),(t',\sigma') \in \bE_I$, we write $(t,\sigma) \sim_\triv (t',\sigma')$ if they are $\bs$-equivalent (\emph{i.e.}, $B^{(t, \sigma)} = B^{(t', \sigma')}$) and for some edge path $\gamma: (t,\sigma) \to (t',\sigma')$ in $\bE_I$, the composition
    \begin{align}\label{eq:automorphism}
        \cZ_{(t,\sigma)} \xrightarrow{\mu_\gamma^z} \cZ_{(t',\sigma')} \xrightarrow{\sim} \cZ_{(t,\sigma)}
    \end{align}
is identity.
Here $(z,\cZ) = (a, \cA), (x, \cX^\uf)$, $\mu_{\gamma_\nu}^z$ is the associated cluster transformation and the isomorphism $\cZ_{(v')} \xrightarrow{\sim} \cZ_{(v)}$ is induced by the seed isomorphism.
We denote the equivalence class for $\sim_\triv$ containing $(t,\sigma)$ by $[t,\sigma]_\triv$.
The \emph{labeled exchange graph} is the quotient graph
\begin{align*}
    \bExch_\bs := \bE_I / \sim_\triv,
\end{align*}
where two edges with equivalent endpoints are identified (so that there are no multiple edges in $\bExch_\bs$). 
Since the relation $\sim_\triv$ respects horizontal/vertical properties of an edge of $\bE_I$, the seed pattern $\widetilde{\bs}$ on $\bE_I$ descends to a ``seed pattern" $\bs_{\mathrm{Ex}}: \bExch_\bs \ni v \mapsto (N^{(v)}, B^{(v)})$ satisfying $B^{(v)} = \sigma.B^{(t)}$ for some $(t, \sigma) \in v$. 

Moreover, the vertical projection $\pi_\bT: \bE_I \to \bT_{I_\uf}$ induces a graph projection $\pi_{\mathrm{Ex}}: \bExch_\bs \to \Exch_\bs$, and the latter graph $\Exch_\bs$ is called the \emph{exchange graph}:
\[\begin{tikzcd}
\bE_I \ar[r] \ar[d, "\pi_\bT"'] & \bExch_\bs \ar[d, "\pi_{\mathrm{Ex}}"] \\
\bT_{I_\uf} \ar[r] & \Exch_\bs\, .
\end{tikzcd}\]
Then, define $\cZ_{[t,\sigma]_\triv} := \bigcup_{(t',\sigma') \in [t,\sigma]_\triv} \cZ_{(t',\sigma')}$, where the latter is obtained by patching the tori by trivial cluster transformations, so that we have 
\begin{align*}
    \cZ_\bs = \bigcup_{[t,\sigma]_\triv \in \bExch_\bs} \cZ_{[t,\sigma]_\triv}.
\end{align*}

Following the previous paper \cite{IK20}, we define mutation loops as certain equivalence classes of edge paths in the labeled exchange graph $\bExch_\bs$.
Let $\gamma_\nu: v_\nu \to v'_\nu$ be an edge path in $\bExch_\bs$ such that $v_\nu \sim_\bs v'_\nu$ for $\nu=1,2$.
We say that $\gamma_1$ and $\gamma_2$ are \emph{$\bs$-equivalent}
if there exists path $\delta: v_1 \to v_2$ such that  the following diagram commutes:
\begin{equation}\label{eq:equivalence of paths}
\begin{tikzcd}
    \cZ_{(v_1)} \ar[r, "\mu^z_{\gamma_1}"] \ar[d, "\mu^z_\delta"'] &\cZ_{(v'_1)} \ar[r, "\sim"]& \cZ_{(v_1)} \ar[d, "\mu_{\delta}^z"]\\
    \cZ_{(v_1)} \ar[r, "\mu^z_{\gamma_2}"] & \cZ_{(v'_2)} \ar[r, "\sim"] & \cZ_{(v_2)}. 
\end{tikzcd}
\end{equation}
Here $(z,\cZ) = (a, \cA), (x, \cX^\uf)$.
The $\bs$-equivalence class containing an edge path $\gamma$ is denoted by $[\gamma]_{\bs}$. 

\begin{dfn}[mutation loops]
A \emph{mutation loop} is an $\bs$-equivalence class of an edge path $\gamma:v \to v'$ in the labeled exchange graph $\bExch_\bs$ such that $v \sim_\bs v'$. We call $\gamma$ a \emph{representation path} of the mutation loop $\phi=[\gamma]_{\bs}$. 
\end{dfn}
The concatenation of two paths $\gamma_1: v_0 \to v_1$ and $\gamma_2: v_1 \to v_2$ is denoted by $\gamma_1 \ast \gamma_2: v_0 \to v_1 \to v_2$.

\begin{dfn}[cluster modular group]
The group $\Gamma_{\bs}$ consisting of all the mutation loops is called the \emph{cluster modular group}. 
Here the multiplication of two mutation loops $\phi_1$ and $\phi_2$ is defined as $\phi_2 \cdot \phi_1:= [\gamma_1 \ast \gamma_2]_\bs$, where $\gamma_\nu$ is a representation path of $\phi_\nu$ for $\nu = 1,2$ such that the terminal vertex of $\gamma_1$ coincides with the initial vertex of $\gamma_2$.
\end{dfn}

\paragraph{\textbf{$C$- and $G$-matrices.}}
Recall the assignment of the $C$- and $G$-matrices can be extended to $\bE_I$ by fixing a vertex $(t_0,\sigma_0) \in \bE_I$ as
\[C^{\widetilde{\bs};(t_0,\sigma_0)}_{(t,\sigma)} :=  P_\sigma C^{\bs;t_0}_{t}P_{\sigma_0}^{-1}, \quad
G^{\widetilde{\bs};(t_0,\sigma_0)}_{(t,\sigma)} :=  P_\sigma G^{\bs;t_0}_{t}P_{\sigma_0}^{-1},
\]
where $P_\sigma$ denotes the presentation matrix of $\sigma \in \fS_N$.
Then the separation formula \cite[Proposition 3.13, Corollary 6.4]{FZ-CA4} tells us that the assignments
\begin{align*}
     C^{v_0}_v:=C^{\widetilde{\bs};(t_0,\sigma_0)}_{(t,\sigma)},\quad 
     G^{v_0}_v:=G^{\widetilde{\bs};(t_0,\sigma_0)}_{(t,\sigma)}
\end{align*}
for $v_0:=[t_0,\sigma_0]_\triv$, $v=[t,\sigma]_\triv \in \bExch_\bs$ are well-defined.
From the tropical duality \cite{NZ12}, we get 
\begin{align}\label{eq:tropical_duality_exchange}
    G^{v_0}_v = \check{C}^{v_0}_v.
\end{align}

\subsection{The seed pattern associated with a marked surface}\label{sec:surf_seed_pattern}
Here, we recall the seed pattern associated with a marked surface and some of the associated objects.
For more details, see \cite[Sections 2 and 4]{IK20}.

A marked surface $\Sigma$ is a compact oriented surface with a fixed non-empty finite set of \emph{marked points} on it.
A marked point is called a \emph{puncture} if it lies in the interior of $\Sigma$, and a \emph{special point} otherwise. 
Let $P=P(\Sigma)$ (resp. $M_\partial=M_\partial(\Sigma)$) denote the set of punctures (resp. special points).
A marked surface is called a \emph{punctured surface} if it has empty boundary (and hence $M_\partial=\emptyset$). 
We denote by $g$ the genus of $\Sigma$, $h$ the number of punctures, and $b$ the number of boundary components in the sequel. We always assume the following conditions:
\begin{enumerate}
    \item[(S1)] Each boundary component (if exists) has at least one marked point.
    \item[(S2)] $3(2g-2+h+b)+2\spe >0$.
    \item[(S3)] If $g=0$ and $b=0$, then $h \geq 4$.
\end{enumerate}

For an ideal triangulation $\tri$ of $\Sigma$, let $\tri_\uf \subset \tri$ denote the subset consisting of internal arcs.
Let $I = I(\Sigma) := \{1, \dots, 3(2g-2+h+b)+2\spe\}$ and $I_\uf = I_\uf(\Sigma) := \{1, \dots, 3(2g-2+h+b)+\spe\}$.
A \emph{labeled triangulation} consists of an ideal triangulation $\tri$ and a bijection
$\ell: I \to \tri$ called a \emph{labeling} such that
$\ell(I_\uf) = \tri_\uf$.

A \emph{signed triangulation} of $\Sigma$ is a pair $(\tri,\varsigma)$ of an ideal triangulation $\tri$ of $\Sigma$ and a tuple $\varsigma \in \{+, -\}^P$ of signatures.
Two signed triangulations are equivalent if the underlying ideal triangulations are the same and the signatures differ only on univalent punctures.
The equivalence classes are called \emph{tagged triangulations}.
A \emph{labeled signed triangulation} is similarly defined, and a \emph{labeled tagged triangulation} is defined as a certain equivalence class of a labeled signed triangulation.

If $\tri$ has no self-folded triangles, for a triangle $t$ of $\tri$, we define a matrix $B(t)=(b^{\tri}_{\alpha\beta}(t))_{\alpha,\beta \in \tri}$ by
\begin{align*}
    b^{\tri}_{\alpha\beta}(t):= \begin{cases}
        1 & \mbox{if $t$ has $\alpha$ and $\beta$ as its consecutive edges in the clockwise order}, \\
        -1 & \mbox{if the same holds with the counter-clockwise order}, \\
        0 & \mbox{otherwise}.
    \end{cases}
\end{align*}
Then we set $B^\tri:=\sum_t B(t)$, where $t$ runs over all triangles of $\tri$.
See \cite[Section 2.2]{IK20} for the case $\tri$ has self-folded triangles and the tagged case.
For a labeled tagged triangulation $(\tri, \varsigma, \ell)$, we define a skew-symmetric matrix $B^{(\tri, \varsigma, \ell)}$ by
\[ B^{(\tri, \varsigma, \ell)} := (b^{\tri}_{\ell(i),\ell(j)})_{i,j\in I}. \]
Then for a labeled flip $\mu_k: (\tri,\varsigma,\ell) \to (\tri',\varsigma,\ell')$, we have $B^{(\tri',\varsigma,\ell')}=\mu_k(B^{(\tri,\varsigma,\ell)})$. Thus we obtain a seed pattern $\bs_\Sigma$ by choosing an \lq\lq initial" correspondence between a vertex of $\bT_{I_\uf}$ and a labeled tagged triangulation and extending it via mutations. Its isomorphism class only depends on the marked surface $\Sigma$.

The labeled exchange graph $\bExch_\Sigma := \bExch_{\bs_\Sigma}$ is identified with the graph consisting of the labeled tagged triangulations, as follows.
Let $\Tri(\Sigma)$ (resp. $\bTri(\Sigma)$) denote the graph whose vertices are tagged triangulations (resp. labeled tagged triangulations) of $\Sigma$ and adjacent labeled tagged triangulations are related by a flip (resp. labeled flip or the action of a transposition of labelings in $\fS_I$).

\begin{thm}[See {\cite[Theorem 4.1]{IK20}}]\label{thm:Tri_Exch}
If $\Sigma$ is a marked surface except for a once-punctured surface, there exist graph isomorphisms 
\begin{align*}
    \bTri(\Sigma) \xrightarrow{\sim} \bExch_\Sigma \quad\mbox{and} \quad \Tri(\Sigma) \xrightarrow{\sim} \Exch_\Sigma.
\end{align*}
If $\Sigma$ is a once-punctured surface, then each of the graphs $\bTri(\Sigma)$ and $\Tri(\Sigma)$ has two connected components, which are isomorphic to each other. The graph $\bExch_\Sigma$ (resp. $\Exch_\Sigma$) is isomorphic to the connected component of $\bTri(\Sigma)$ (resp. $\Tri(\Sigma)$) containing the initial labeled tagged triangulation (resp. the initial tagged triangulation).
\end{thm}
Henceforth we will identify the graphs $\bTri(\Sigma)$ and $\Tri(\Sigma)$ with (connected components of) $\bExch_\Sigma$ and $\Exch_\Sigma$, respectively.

We have the following geometric models for the associated tropical cluster varieties. 

\begin{thm}[See {\cite[Section 2]{IK20}}]\label{thm:geom_models}
Fix a hyperbolic structure $F$ on $\Sigma$ as in \cite[Section 2.5]{IK20}. 
Then we have the following PL isomorphisms:
\begin{align*}
    \cX^\uf_\Sigma(\bR^\trop) &\cong \cL^x(\Sigma, \bR) \cong \eML(F),\\
    \cU^\uf_\Sigma(\bR^\trop) &\cong \cL^u(\Sigma, \bR) \cong \ML_0(F).
\end{align*}
Here $\cL^z(\Sigma, \bR)$ is the space of real $\cZ$-laminations for $(\cZ, z) \in \{(\cX^\uf, x), (\cU^\uf, u)\}$, $\eML(F)$ is the space of ($\iota$-invariant) measured geodesic laminations on (the double of) $F$, and $\ML_0(F) \subset \eML(F)$ is the subspace consisting of measured laminations with compact support.
\end{thm}
For a description of $\A_\Sigma(\bR^\trop)$ and $\cU_\Sigma^\uf(\bR^\trop)$ in terms of measured foliations, see \cite[Section 2.4]{IK20}.

We define the group $\Gamma_\Sigma$ as
\begin{align*}
    \Gamma_\Sigma := \begin{cases}
    MC(\Sigma) & \mbox{if $\Sigma$ is a punctured surface with exactly one puncture}, \\    
    MC(\Sigma) \ltimes (\bZ/2)^P & \mbox{otherwise}.
    \end{cases}
\end{align*}
Then this group is ``almost isomorphic" to the cluster modular group $\Gamma_{\bs_\Sigma}$ associated with the seed pattern $\bs_\Sigma$.

\begin{thm}[See {\cite[Theorem 4.5]{IK20}}]
\label{t:cluster modular_surface case}
If $\Sigma$ is not a sphere with 4 punctures, then the group $\Gamma_\Sigma$ is isomorphic to the cluster modular group $\Gamma_{\bs_\Sigma}$.
Otherwise, $\Gamma_\Sigma$ is embedded into $\Gamma_{\bs_\Sigma}$ as a subgroup of index $\geq 2$.
\end{thm}
The graph isomorphisms in \cref{thm:Tri_Exch} and the PL isomorphisms in \cref{thm:geom_models} are equivariant with respect to the embedding $\Gamma_\Sigma \to \Gamma_{\bs_\Sigma}$, and therefore we can study the actions of mapping classes with reflections on these objects in terms of the corresponding mutation loops.

\subsection{Sign stability}\label{subsec:sign stability}
Let us briefly recall the presentation matrices of a piecewise-linear map (PL map for short). For a vector space $V$, we have the translation isomorphism $t_x:V \xrightarrow{\sim} T_xV$, $v \mapsto \left.\frac{d}{ds}\middle|\right._{s=0}(x+sv)$ for any point $x \in V$. For a PL map $\psi: V \to W$ and a differentiable point $x \in V \setminus \{0\}$, observe that the linear map
\begin{align*}
    V \xrightarrow{t_x} T_x V \xrightarrow{(d\psi)_x} T_{\psi(x)}W \xrightarrow{t_{\psi(x)}^{-1}} W
\end{align*}
coincides with the linear extension of the germ of $\psi$ at $x$. When $V$ and $W$ are equipped with bases, we define the \emph{presentation matrix of $\psi$ at a differentiable point $x$} to be the presentation matrix of the tangent map $(d\psi)_x:T_x V \to T_{\psi(x)} W$ with respect to the given bases. By the above observation, it coincides with the presentation matrix of the linear extension of the germ of $\psi$ at $x$. The following lemma clarifies the meaning of an \lq\lq eigenvector" of a PL map:
\begin{lem}
Suppose that $n=\dim V= \dim W$, and $V$ and $W$ are equipped with bases $(e_i)_{i=1}^n$ and $(e'_j)_{j=1}^n$, respectively. 
Consider the presentation matrix $E_x \in \mathrm{Mat}_{n\times n}(\bR)$ of a PL map $\psi:V \to W$ at a differentiable point $x \in V$. Then for $\lambda \in \bR$, the coordinate vector $\boldsymbol{v}=(v_i)_{i=1}^n \in \bR^n$ of $v=\sum_i v_i e_i \in V$ satisfies $E_x\boldsymbol{v} = \lambda \boldsymbol{v}$ if and only if $(d\psi)_x(t_x(v)) =\lambda t_{\psi(x)}(v)$.
\end{lem}
When $V=\X_{(v)}(\bR^\trop)$ and $W=\X_{(v')}(\bR^\trop)$ for some $v,v' \in \bExch_\bs$, we always consider the bases $(f_i^{(v)})_{i \in I}$ and $(f_i^{(v')})_{i \in I}$ respectively, unless otherwise specified. 

For a PL map $f: M \to N$ between two PL manifolds with specified PL atlases, we say that $f$ is \emph{intrinsically differentiable} at $x \in M$ if its coordinate expression is differentiable at the corresponding point for any charts in the given atlases. 
When $M=\X_\bs^\uf(\bR^\trop)$ and $N=\X_{\bs'}^\uf(\bR^\trop)$ for some seed patterns $\bs$ and $\bs'$, we discuss the intrinsic differentiability with respect to the cluster atlases unless otherwise specified.

\bigskip

Now recall the tropical cluster $\X$-transformation $\mu_\indk: \X^\uf_{(v)}(\bR^\trop) \to \X^\uf_{(v')}(\bR^\trop)$ associated with an edge $v \overbar{k} v'$ in $\bExch_\bs$.
For a real number $a \in \bR$, let $\sgn(a)$ denote its sign:
\[
\sgn(a):=
\begin{cases}
    + & \mbox{ if } a>0,\\
    0 & \mbox{ if } a=0,\\
    - & \mbox{ if } a<0.
\end{cases}
\]

\begin{lem}[{\cite[Lemma 3.1]{IK19}}]\label{l:x-cluster signed}
Fix a point $w \in \X^\uf_{(v)}(\bR^\trop)$. 
Then the tropical cluster $\X$-transformation $\mu_\indk: \X^\uf_{(v)}(\bR^\trop) \to \X^\uf_{(v')}(\bR^\trop)$ is given by
\begin{align}\label{eq:sign x-cluster}
    x^{(v')}_\indi(\mu_\indk(w)) =
\begin{cases}
    -x^{(v)}_\indk(w) & \mbox{if $\indi=\indk$}, \\
    x^{(v)}_\indi(w)+[\sgn(x^{(v)}_\indk(w))b^{(v)}_{\indi\indk}]_+x^{(v)}_\indk(w) & \mbox{if $\indi \neq \indk$}.
\end{cases}
\end{align}
\end{lem}
This lemma says that the tropical cluster $\X$-transformation $\mu_\indk: \X^\uf_{(v)}(\bR^\trop) \to \X^\uf_{(v')}(\bR^\trop)$ is linear on each of the half-spaces
\begin{align*}
    \cH_{\indk,\epsilon}^{x,(v)}:= \{ w \in \X^\uf_{(v)}(\bR^\trop) \mid \epsilon x^{(v)}_\indk(w) \geq 0 \}
\end{align*}
for $\indk \in I_\uf$, $\epsilon \in \{+,-\}$ and $v \in \bExch_\bs$.

In order to introduce the \emph{sign} of a path in $\bExch_\Sigma$, we use the following notation.
\begin{conv}\label{path convention}
\begin{enumerate}
    \item For an edge path $\gamma: v_0 \overbar{k_0} v_1 \overbar{k_1} \cdots \overbar{k_{m-1}} v_m$ in $\bExch_\bs$ and $i=1,\dots,m$, let $\gamma_{\leq i}: v_0 \xrightarrow{(k_0,\dots,k_{i-1})} v_i$ be the sub-path of $\gamma$ from $v_0$ to $v_i$. Let $\gamma_{\leq 0}$ be the constant path at $v_0$. 
    Let $(k_{i(0)},\dots,k_{i(h-1)})$ be the subsequence of $\mathbf{k}=(k_0,\dots,k_{m-1})$ which corresponds to horizontal edges. 
    \item Fixing the initial vertex $v_0$, we simply denote the tropicalization of the coordinate expression $\phi_{(v_0)}^x: \cX^\uf_{(v_0)} \xrightarrow{\mu_\gamma^x} \cX^\uf_{(v_m)} \xrightarrow{\sim} \cX^\uf_{(v_0)}$ of $\phi$ with respect to $v_0$ by $\phi:=\phi_{(v_0)}: \X_{(v_0)}^\uf(\bR^\trop) \to \X_{(v_0)}^\uf(\bR^\trop)$, when no confusion can occur.
\end{enumerate}

\end{conv}

\begin{dfn}[sign of a path]\label{d:sign}
Let the notation as above, and fix a point $w \in \X_{(v_0)}^\uf(\bR^\trop)$.
\item The \emph{sign} of $\gamma$ at $w$ is the sequence $\boldsymbol{\epsilon}_\gamma(w)=(\epsilon_0,\dots,\epsilon_{h-1}) \in \{+,0,-\}^h$ of signs defined by
\[\epsilon_{\nu} := \sgn(x^{(v_{i(\nu)})}_{k_{i(\nu)}}(\mu_{\gamma_{\leq {i(\nu)}}} (w)))\]
for $\nu=0,\dots,h-1$.
\end{dfn}

\begin{dfn}[sign stability]\label{d:sign stability}
Let $\gamma$ be a path as above and suppose that $\phi:=[\gamma]_{\bs}$ is a mutation loop. Let $\Omega \subset \X_{(v_0)}^\uf(\bR^\trop)$ be a subset which is invariant under the rescaling action of $\bR_{> 0}$. 
Then we say that $\gamma$ is \emph{sign-stable} on $\Omega$ if there exists a sequence $\boldsymbol{\epsilon}^\stab_{\gamma,\Omega} \in \{+,-\}^h$ of strict signs such that for each $w \in \Omega \setminus \{0\}$, there exists an integer $n_0 \in \mathbb{N}$ such that   \[\boldsymbol{\epsilon}_\gamma(\phi^n(w)) = \boldsymbol{\epsilon}^\stab_{\gamma, \Omega} \]
for all $n \geq n_0$. 
We call $\boldsymbol{\epsilon}_{\gamma,\Omega}^\stab$ the \emph{stable sign} of $\gamma$ on $\Omega$.
\end{dfn}
For a mutation loop $\phi \in \Gamma_\bs$ and a vertex $v_0 \in \bE_I$, its presentation matrix at a differentiable point $w \in  \X_{(v_0)}^\uf(\bR^\trop)$ is the presentation matrix $E_\phi^{(v_0)}(w) \in GL_N(\bZ)$ of the tangent map 
\begin{align*}
    (d\phi)_w: T_w \X_{(v_0)}^\uf(\bR^\trop) \to T_{\phi(w)} \X_{(v_0)}^\uf(\bR^\trop)
\end{align*}
with respect to the basis $(f_i^{(v_0)})_{i \in I_\uf}$ of $M^{(v_0)}_\uf \otimes_\bZ \bR \cong \cX^\uf_{(v_0)}(\bR^\trop)$.

For a vertex $v_0 \in \bExch_\bs$, let $\cC^\pm_{(v_0)}:=\{w \in \X_{(v_0)}^\uf(\bR^\trop) \mid \pm x_i^{(v_0)}(w) \geq 0,~i \in I_\uf\}$
and let
\begin{align*}
    \Omega^{\mathrm{can}}_{(v_0)}:=\interior\cC^+_{(v_0)} \cup \interior\cC^-_{(v_0)}.
\end{align*}
It turns out that it is sufficient for the computation of the algebraic entropy of cluster transformations \cite{IK19}.
\begin{dfn}\label{d:cluster stretch factor}
If a mutation loop $\phi$ admits a representation path $\gamma:v_0 \to v$ which is sign-stable on the set $\Omega^{\mathrm{can}}_{(v_0)}$, then we call
$\lambda_\phi:=\lambda_{\phi,\Omega^{\mathrm{can}}_{(v_0)}}$ the \emph{cluster stretch factor} of $\phi$. 
\end{dfn}

\paragraph{\textbf{Uniform sign stability and the Fock--Goncharov cluster complex.}}
We define a kind of the \lq\lq strongest" sign stability, called the \emph{uniform sign stability}.

\begin{dfn}[Uniform sign stability]\label{d:uniform stability}
A mutation loop $\phi \in \Gamma_\bs$ is said to be \emph{uniformly sign-stable} if any representation path $\gamma:v_0 \to v$ of $\phi$ is sign-stable on $\bR_{>0}\cdot \X_{(v_0)}^\uf(\bQ^\trop)$. 
\end{dfn}
The uniform sign stability automatically implies the sign stability on a larger subset. 
We call the subset
\begin{align*}
    |\fF_{(v_0)}^\pm|:=\bigcup_{v \in \bExch_\bs} \mu_{\gamma(v)}^{-1}(\cC_{(v)}^\pm) \subset \X_{(v_0)}^\uf(\bR^\trop)
\end{align*}
the (geometric realization of) the \emph{Fock--Goncharov cluster complex}, according to \cite{FG09,GHKK}.
Here $\gamma(v): v_0 \to v$ is any path in $\bExch_\bs$.
It is known that the corresponding subsets $|\fF_\bs^\pm| \subset \X_\bs^\uf(\bR^\trop)$ are independent of the choice of $v_0 \in \bExch_\bs$ \cite[Theorem 2.13]{GHKK}.
Therefore we can consider the canonically defined $\bR_{>0}$-invariant subset 
\begin{align*}
    \Omega_\bs^\bQ:= |\fF_\bs^+| \cup |\fF_\bs^-| \cup \bR_{>0}\cdot \X_\bs^\uf(\bQ^\trop) \subset \X_\bs^\uf(\bR^\trop)
\end{align*}
and the corresponding sets $\Omega_{(v_0)}^\bQ \subset \X_{(v_0)}^\uf(\bR^\trop)$ in the coordinate charts. Then \cite[Proposition 3.14]{IK19} tells us that the sign stability on $\bR_{>0}\cdot \X_{(v_0)}^\uf(\bQ^\trop)$ is equivalent to that on the set $\Omega_{(v_0)}^\bQ$. Since $\Omega^{\mathrm{can}}_{(v_0)} \subset \Omega_{(v_0)}^\bQ$, a uniformly sign-stable mutation loop comes with its cluster stretch factor $\lambda_\phi$.

\subsection{Common differentiable points}
Recall that the PL automorphism on $\X_\bs^\uf(\bR^\trop) \to \X_\bs^\uf(\bR^\trop)$ induced by a mutation loop $\phi \in \Gamma_\bs$ is intrinsically differentiable at a point $w \in \X_\bs^\uf(\bR^\trop)$ if for any vertex $v_0 \in \bExch_\bs$, the PL automorphism $\phi_{(v_0)}$ on $\X_{(v_0)}^\uf(\bR^\trop)$ is differentiable at the point corresponding to $w$. 

The following lemma shows that we have enough differentiable points:

\begin{lem}\label{lem:differentiable_domain}
Let $\phi \in \Gamma_\bs$ be a mutation loop. Then for any vertex $v_0 \in \bExch_\bs$, the PL automorphism $\phi_{(v_0)}$ is differentiable at any point in the open subset
\begin{align*}
    \mathfrak{U}_{(v_0)}:=\bigcup_{v \in \bExch_\bs} \mu_{\gamma(v)}^{-1}(\interior\cC_{(v)}^+ \cup \interior\cC_{(v)}^-) \subset \X_{(v_0)}^\uf(\bR^\trop).
\end{align*}
In particular, the PL action of $\phi$ on the PL manifold $\X_\bs^\uf(\bR^\trop)$ is intrinsically differentiable at any point in the corresponding $\Gamma_\bs$-invariant subset $\mathfrak{U}_\bs \subset \X_\bs^\uf(\bR^\trop)$.
\end{lem}

\section{Reduction of mutation loops}\label{sec:reduction}
As a preparation for the study of the sign stability of a pseudo-Anosov mapping class on a marked surface (see \cref{def:hyp_bdry}), we investigate the classical reduction procedure of geodesic laminations and mapping classes along a multicurve in terms of the cluster varieties and mutation loops. 
Some of the observations in this section lead us to a generalized reduction procedure, which will be axiomatized in \cref{sec:C-reduction}.

Let $\Sigma$ be a marked surface, and $\cR=\{C_i\}_{k=1}^K$ a multicurve on $\Sigma$, that is, the isotopy class of a collection of mutually disjoint simple closed curves which are neither peripheral nor contractible. Consider the subgroup 
\begin{align*}
    MC_\cR(\Sigma):=\{ \phi \mid \phi(\cR) = \cR\} \subset MC(\Sigma)
\end{align*}
consisting of mapping classes which preserve the multicurve $\cR$ (but allowed to permute the curves in $\cR$). Let $\Sigma_\cR$ be the marked surface obtained by cutting $\Sigma$ along the curves in $\cR$ and then shrinking the new boundary components to punctures. The set of punctures on $\Sigma_\cR$ is $P_\cR:=P \sqcup \{p_k^+,p_k^- \mid k=1,\dots,K\}$, where $p_k^\pm$ denote the new punctures arising from the component $C_k$. Note that a mapping class $\phi \in MC_\cR(\Sigma)$ induces a mapping class $\pi_\cR(\phi) \in MC(\Sigma_\cR)$, and we have the exact sequence
\begin{align*}
    1 \to \bZ \langle T_{C_k} \mid k=1,\dots,K \rangle \to MC_\cR(\Sigma) \xrightarrow{\pi_\cR} MC(\Sigma_\cR) \to 1.
\end{align*}
We extend this group homomorphism $\pi_\cR$ to the group $\Gamma_{\Sigma,\cR}:=MC_\cR(\Sigma) \ltimes (\bZ/2)^P \subset \Gamma_\Sigma$ in the following way:
\begin{align}\label{eq:pi_S_cluster_modular}
    \pi_\cR^\mathrm{grp}: \Gamma_{\Sigma,\cR} \xrightarrow{\pi_\cR \times \mathrm{id}} MC(\Sigma_\cR) \ltimes (\bZ/2)^P \hookrightarrow MC(\Sigma_\cR) \ltimes (\bZ/2)^{P_\cR}=\Gamma_{\Sigma_\cR},
\end{align}
where the second map is induced by the natural inclusion $P \hookrightarrow P_\cR$, and the composite is again denoted by $\pi_\cR$ by an abuse of notation. We call the group homomorphism $\pi_\cR^\mathrm{grp}:\Gamma_{\Sigma,\cR} \to \Gamma_{\Sigma_\cR}$ the \emph{reduction homomorphism} with respect to the multicurve $\cR$.

Recall from \cref{thm:geom_models} that fixing a hyperbolic structure $F$ on $\Sigma$,
we have the PL isomorphism $\X_\Sigma^\uf(\bR^\trop) \cong  \eML(F)$.
Straighten each curve $C_k$ to a geodesic for $k=1,\dots,K$. Cutting the surface $F$ along these geodesics, we get a hyperbolic surface $F_\cR$ with two geodesic boundary components $\partial_k^\pm$ arising from the geodesics $C_k$ for $k=1,\dots,K$, whose interior is homeomorphic to $\Sigma_\cR$. 
Henceforth in this section, we tacitly use the PL isomorphisms $\X_\Sigma^\uf(\bR^\trop) \cong \eML(F)$ and $\X_{\Sigma_\cR}^\uf(\bR^\trop) \cong \eML(F_\cR)$ obtained in this way. 

Each measured geodesic lamination $(G,\mu)$ on $F$ induces a measured geodesic lamination on $F_\cR$, as follows. Cutting $L$ along the geodesics $C_k$ for $k=1,\dots,K$, we get a measured geodesic lamination $(\bar{G},\bar{\mu})$ on $F_\cR$ possibly with leaves perpendicular to the geodesic boundary. Let $\pi_\cR^\trop(G,\mu)$ be the unique measured geodesic lamination on $F_\cR$ which shares the compact part and the homotopy class of the non-compact part with $(\bar{G},\bar{\mu})$, but with the positive spiralling around each component $\partial_k^\pm$. It is uniquely determined by \cite[Lemma 2.18]{IK20}. Then we get a map
\begin{align}\label{eq:pi_S}
    \pi_\cR^\trop: \X_\Sigma^\uf(\bR^\trop) \to \X_{\Sigma_\cR}^\uf(\bR^\trop).
\end{align}

\begin{lem}\label{lem:pi_S_equivariance}
The map $\pi_\cR^\trop$ is continuous and  $\Gamma_{\Sigma,\cR}$-equivariant, where the latter acts on the target via the reduction homomorphism $\pi_\cR^\mathrm{grp}$.
\end{lem}

\begin{proof}
The continuity follows from the equation
\begin{align*}
    \cI_C(\pi_\cR^\trop(\hL)) = \cI_C(\hL)
\end{align*}
for $\hL \in \X_\Sigma^\uf(\bR^\trop)$ and a simple closed curve $C$ in $\Sigma_\cR$, where the latter is naturally lifted to a curve in $\Sigma$. The $\Gamma_{\Sigma,\cR}$-equivariance is clear.
\end{proof}
When it is clear from the context, we simply denote the maps $\pi_\cR^\grp$ and $\pi_\cR^\trop$ by the same symbol $\pi_\cR$.

\subsection{Reduction of triangulations and the differentiability of the map \texorpdfstring{$\pi_\cR^\trop$}{pi R trop}}\label{subsec:differentiability}
For an ideal arc $\alpha$ on $\Sigma$, let $\pi_\cR(\alpha)$ be the collection of ideal arcs in $\Sigma_\cR$ naturally obtained from the intersection $\alpha \cap (\Sigma \setminus \bigcup\cR)$.
For a collection $C=\{\alpha_i\}$ of mutually disjoint ideal arcs on $\Sigma$, let $\pi_\cR(C):=\bigcup_i \pi_\cR(\alpha_i)$.
Then the following is a direct consequence of \cite[Lemma 2.18]{IK20}:

\begin{lem}\label{lem:pi_S_ideal_arc}
Let $C=\{\alpha_i\}$ be a collection of mutually disjoint ideal arcs on $\Sigma$. Assigning any transverse measure and the positive sign at each puncture to which some of curves in $C$ incident, we get a measured geodesic lamination $\hL_C \in \eML(F) = \X_\Sigma^\uf(\bR^\trop)$. Then the homotopy class of the support of $\pi_\cR(\hL_C) \in \eML(F_\cR) = \X_{\Sigma_\cR}^\uf(\bR^\trop)$ coincides with $\bigcup \pi_\cR(C)$.
\end{lem}

For a tagged triangulation $\tri = (\tri, \varsigma_\tri) \in \Exch_\Sigma$ of $\Sigma$, define a tagged triangulation $\oline{\tri} = (\oline{\tri}, \oline{\varsigma_\tri})$ of $\Sigma_\cR$ as follows:
\begin{itemize}
    \item $\oline{\varsigma_\tri}(p) :=
    \begin{cases}
    \varsigma_\tri(p) & \mbox{if } p \in P,\\
    + & \mbox{if $p=p_k^\pm$ for some $k=1,\dots,K$.}
    \end{cases}$
    \item $\oline{\tri} := \pi_\cR(\tri)$.
\end{itemize}
See \cref{fig:cutting_example} for an example.

\begin{lem}\label{lem:collapsed_tri}
\begin{enumerate}
    \item The collection $\oline{\tri}$ of arcs give a triangulation of $\Sigma_\cR$.
    \item For any $\alpha \in \tri$, we have
    \begin{align*}
    \oline{f_\alpha(\tri)} =
        \begin{cases}
            f_{\oline{\alpha}}(\oline{\tri}) & \mbox{if $x_\alpha^\tri(C_k)=0$ for all $k=1,\dots,K$},\\
            \oline{\tri} & \mbox{otherwise}.
        \end{cases}
    \end{align*}
    Here $\oline{\alpha} \in \pi_\cR(\alpha)$ is a unique ideal arc along which we can perform a flip of $\oline{\tri}$. 
\end{enumerate}
\end{lem}

\begin{proof}
(1): 
Evidently $\oline{\tri}$ is a collection of ideal arcs in $\Sigma_\cR$. 
Some of them may coincide to each other, and each complementary region of $\tri$ either degenerates to a bigon (and collapsed) or descends to a triangle in $\Sigma_\cR$. By the construction, the surface $\Sigma_\cR$ is covered by the resulting triangles. 

(2):
Consider the quadrilateral $Q_\alpha$ in $\tri$ which contains $\alpha$ as a diagonal. It is cutted by the curves in $\cR$ into several pieces. If $x_\alpha^\tri(C_k)=0$ for all $k=1,\dots,K$, then the intersection $Q_\alpha \cap \cR$ looks as in the left of \cref{fig:quadrilateral_intersection}, and there exists a unique connected component of $Q_\alpha \setminus (Q_\alpha \cap \cR)$ which gives rise to a quadrilateral after the shrinking procedure. Moreover it contains a diagonal $\oline{\alpha}$ coming from $\pi_\cR(\alpha)$ and we get $\oline{f_\alpha(\tri)} = f_{\oline{\alpha}}(\oline{\tri})$. On the other hand, if $x_\alpha^\tri(C_k)\neq 0$ for some $k=1,\dots,K$, then the intersection $Q_\alpha \cap \cR$ looks as in the right of \cref{fig:quadrilateral_intersection}. In this case no quadrilateral survives after the shrinking procedure, and in particular $Q_\alpha$ and $f_\alpha(Q_\alpha)$ results in the same configuration. 
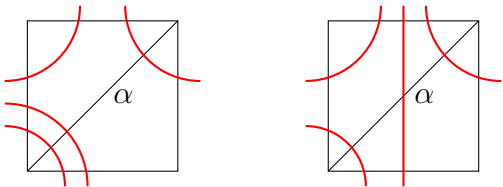
\begin{figure}[h]
\begin{tikzpicture}
\draw (0,2) -- (2,2) -- (2,0) -- (0,0) --cycle;
\draw (2,2) -- node[midway,right]{$\alpha$} (0,0);
\draw[red,thick] (0.7,2.2) arc[x radius=1cm, y radius=1cm, start angle=0, end angle=-90];
\draw[red,thick] (1.3,2.2) arc[x radius=1cm, y radius=1cm, start angle=180, end angle=270];
\draw[red,thick] (0.5,-0.2) arc[x radius=0.8cm, y radius=0.8cm, start angle=0, end angle=90];
\draw[red,thick] (0.8,-0.2) arc[x radius=1.1cm, y radius=1.1cm, start angle=0, end angle=90];
\begin{scope}[xshift=4cm]
\draw (0,2) -- (2,2) -- (2,0) -- (0,0) --cycle;
\draw (2,2) -- node[midway,right]{$\alpha$} (0,0);
\draw[red,thick] (1,2.2) -- (1,-0.2);
\draw[red,thick] (0.7,2.2) arc[x radius=1cm, y radius=1cm, start angle=0, end angle=-90];
\draw[red,thick] (1.3,2.2) arc[x radius=1cm, y radius=1cm, start angle=180, end angle=270];
\draw[red,thick] (0.5,-0.2) arc[x radius=0.8cm, y radius=0.8cm, start angle=0, end angle=90];
\end{scope}
\end{tikzpicture}
    \caption{The intersection $Q_\alpha \cap \cR$. The curves in $\cR$ are shown in red.}
    \label{fig:quadrilateral_intersection}
\end{figure}
\end{proof}

\begin{dfn}\label{def:compatible_arcs}
We say that an ideal arc $\alpha \in \tri$ satisfying $x_\alpha^\tri(C_k)=0$ for all $k=1,\dots,K$ is \emph{$\cR$-compatible}. 
\end{dfn}
Note that this is equivalent to the condition $x_\alpha^\tri(L)=0$ for all $L \in \cC(\cR)$, where $\cC(\cR)$ is the rational polyhedral cone in $\cU_\Sigma^\uf(\bZ^\trop)$ consisting of the integral laminations of the form $L=\bigsqcup_{k=1}^K w_k C_k$ with $w_k \geq 0$. 

From \cref{lem:collapsed_tri} we get a map of graphs $\pi_\cR:\Exch_\Sigma \to \Exch_{\Sigma_\cR}$, where an edge in $\Exch_\Sigma$ corresponding to an $\cR$-compatible arc is projected to an edge in $\Exch_{\Sigma_\cR}$ corresponding to $\oline{\alpha} \in \oline{\tri}$, while the others are projected to vertices.

\begin{ex}
Here is an example of the cutting procedure.
Let $\Sigma$ be an annulus with 2 special points on each of the boundary component.
Consider the unique essential simple closed curve $C$ in $\Sigma$, and let $\cR := \{C\}$.
Then the resulting surface $\Sigma_\cR$ has two connected components: a once-punctured disk with 4 special points and a once-punctured disk with 2 special points.

Let us consider the triangulation $\tri$ of $\Sigma$ like in the configuration of the middle left of \cref{fig:cutting_example}, which induces the triangulation $\oline{\tri}$ of $\Sigma_\cR$ as shown in the middle right.
Then for example, the ideal arc in $\tri$ labeled by $1$ is $\cR$-compatible, while the arc labeled by $2$ is not. The flip along the former arc induces a flip on the induced triangulation (as shown in the bottom row), while the latter does not (as in the top row).  

\begin{figure}[h]
    \centering
    \begin{tikzpicture}[scale=.9]
\draw  (0,3.5) node (v1) {} ellipse (0.5 and 0.5);
\draw  (v1) node (v2) {} ellipse (2 and 2);
\draw [cyan] (v2) ellipse (1.25 and 1.25);
\node [fill, circle, inner sep=1.3] (v4) at (0,4) {};
\node [fill, circle, inner sep=1.3] at (0,3) {};
\node [fill, circle, inner sep=1.3] (v3) at (0,5.5) {};
\node [fill, circle, inner sep=1.3] (v5) at (0,1.5) {};
\node [fill, circle, inner sep=1.3] at (-2,3.5) {};
\node [fill, circle, inner sep=1.3] at (2,3.5) {};
\draw [blue] (v3) -- (v4);
\draw [blue] (0,3) -- (v5);
\draw [blue] (0,4) .. controls (-0.5,4) and (-1.25,4) .. (-2,3.5);
\draw [blue] (0,3) .. controls (-0.5,3) and (-1.25,3) .. (-2,3.5);
\draw [blue] (0,4) .. controls (0.5,4) and (1.25,4) .. (2,3.5);
\draw [blue] (0,3) .. controls (0.5,3) and (1.25,3) .. (2,3.5);
\node [blue] at (0.25,5.05) {\scriptsize 1};
\node [blue] at (-1.5,4) {\scriptsize 2};
\node [blue] at (-1.5,3) {\scriptsize 3};
\node [blue] at (-0.2,1.95) {\scriptsize 4};
\node [blue] at (1.5,3) {\scriptsize 5};
\node [blue] at (1.5,4) {\scriptsize 6};
\node [cyan] at (0.95,2.25) {$C$};

\draw [thick, ->](3,3.5) -- (4.5,3.5);
\draw (6.5,3.5) ellipse (1 and 1);
\draw  (9.5,3.5) ellipse (1 and 1);
\node [fill, circle, inner sep=1.3] (v6) at (6.5,4.5) {};
\node [fill, circle, inner sep=1.3] (v8) at (5.5,3.5) {};
\node [fill, circle, inner sep=1.3] (v7) at (6.5,2.5) {};
\node [fill, circle, inner sep=1.3] (v9) at (7.5,3.5) {};
\node [fill, circle, inner sep=1.3] at (6.5,3.5) {};
\node [fill, circle, inner sep=1.3] (v10) at (9.5,4.5) {};
\node [fill, circle, inner sep=1.3] (v11) at (9.5,2.5) {};
\node [fill, circle, inner sep=1.3] at (9.5,3.5) {};
\draw [blue]  (v6) edge (v7);
\draw [blue]  (v8) edge (v9);
\draw [blue]  (v10) edge (v11);
\node [blue] at (6.7,4.05) {\scriptsize 1};
\node [blue] at (6,3.65) {\scriptsize 2,3};
\node [blue] at (6.3,2.95) {\scriptsize 4};
\node [blue] at (7,3.3) {\scriptsize 5,6};
\node [blue] at (9.9,3.95) {\scriptsize 6,2,1};
\node [blue] at (9.1,3.05) {\scriptsize 5,4,3};
\node at (3.75,3.9) {$\pi_{\mathcal{R}}$};
\node [blue] at (-1.9,5.4) {$\triangle$};
\node at (2.6,1.25) {\small $x_{\triangle}(C) = (0, -1, 1, 0,-1, 1)$};

\draw  (0,10) node (v1) {} ellipse (0.5 and 0.5);
\draw  (v1) node (v2) {} ellipse (2 and 2);
\draw [cyan] (v2) ellipse (1.25 and 1.25);
\node [fill, circle, inner sep=1.3] (v4) at (0,10.5) {};
\node [fill, circle, inner sep=1.3] at (0,9.5) {};
\node [fill, circle, inner sep=1.3] (v3) at (0,12) {};
\node [fill, circle, inner sep=1.3] (v5) at (0,8) {};
\node [fill, circle, inner sep=1.3] at (-2,10) {};
\node [fill, circle, inner sep=1.3] at (2,10) {};
\draw [blue] (v3) -- (v4);
\draw [blue] (0,9.5) -- (v5);
\draw [blue] (0,9.5) .. controls (-0.5,9.5) and (-1.25,9.5) .. (-2,10);
\draw [blue] (0,10.5) .. controls (0.5,10.5) and (1.25,10.5) .. (2,10);
\draw [blue] (0,9.5) .. controls (0.5,9.5) and (1.25,9.5) .. (2,10);
\node [blue] at (0.25,11.55) {\scriptsize 1};
\node [blue] at (-0.95,10.15) {\scriptsize 2};
\node [blue] at (-1.5,9.5) {\scriptsize 3};
\node [blue] at (-0.2,8.45) {\scriptsize 4};
\node [blue] at (1.5,9.5) {\scriptsize 5};
\node [blue] at (1.5,10.5) {\scriptsize 6};
\node [cyan] at (0.95,8.75) {$C$};

\draw [thick, ->](3,10) -- (4.5,10);
\draw (6.5,10) ellipse (1 and 1);
\draw  (9.5,10) ellipse (1 and 1);
\node [fill, circle, inner sep=1.3] (v6) at (6.5,11) {};
\node [fill, circle, inner sep=1.3] (v8) at (5.5,10) {};
\node [fill, circle, inner sep=1.3] (v7) at (6.5,9) {};
\node [fill, circle, inner sep=1.3] (v9) at (7.5,10) {};
\node [fill, circle, inner sep=1.3] at (6.5,10) {};
\node [fill, circle, inner sep=1.3] (v10) at (9.5,11) {};
\node [fill, circle, inner sep=1.3] (v11) at (9.5,9) {};
\node [fill, circle, inner sep=1.3] at (9.5,10) {};
\draw [blue]  (v6) edge (v7);
\draw [blue]  (v8) edge (v9);
\draw [blue]  (v10) edge (v11);
\node [blue] at (6.75,10.5) {\scriptsize 1,2};
\node [blue] at (6,10.15) {\scriptsize 3};
\node [blue] at (6.3,9.45) {\scriptsize 4};
\node [blue] at (7,9.8) {\scriptsize 5,6};
\node [blue] at (9.8,10.45) {\scriptsize 6,1};
\node [blue] at (9.5,9.55) {\scriptsize 5,4 3,2};
\node at (3.75,10.4) {$\pi_{\mathcal{R}}$};
\node [blue] at (-1.9,11.9) {$f_2(\triangle)$};

\draw [-implies, double distance=2pt](0,1) -- (0,-0.5);
\draw [-implies, double distance=2pt](0,6) -- (0,7.5);

\draw  (0,-3) node (v1) {} ellipse (0.5 and 0.5);
\draw  (v1) node (v2) {} ellipse (2 and 2);
\draw [cyan] (v2) ellipse (1.25 and 1.25);
\node [fill, circle, inner sep=1.3] (v4) at (0,-2.5) {};
\node [fill, circle, inner sep=1.3] at (0,-3.5) {};
\node [fill, circle, inner sep=1.3] (v3) at (0,-1) {};
\node [fill, circle, inner sep=1.3] (v5) at (0,-5) {};
\node [fill, circle, inner sep=1.3] at (-2,-3) {};
\node [fill, circle, inner sep=1.3] at (2,-3) {};
\draw [blue] (0,-3.5) -- (v5);
\draw [blue] (0,-2.5) .. controls (-0.5,-2.5) and (-1.25,-2.5) .. (-2,-3);
\draw [blue] (0,-3.5) .. controls (-0.5,-3.5) and (-1.25,-3.5) .. (-2,-3);
\draw [blue] (0,-2.5) .. controls (0.5,-2.5) and (1.25,-2.5) .. (2,-3);
\draw [blue] (0,-3.5) .. controls (0.5,-3.5) and (1.25,-3.5) .. (2,-3);
\node [blue] at (0,-1.35) {\scriptsize 1};
\node [blue] at (-1.45,-2.55) {\scriptsize 2};
\node [blue] at (-1.5,-3.5) {\scriptsize 3};
\node [blue] at (-0.2,-4.55) {\scriptsize 4};
\node [blue] at (1.5,-3.5) {\scriptsize 5};
\node [blue] at (1.45,-2.55) {\scriptsize 6};
\node [cyan] at (0.95,-4.25) {$C$};

\draw [thick, ->](3,-3) -- (4.5,-3);
\draw (6.5,-3) ellipse (1 and 1);
\draw  (9.5,-3) ellipse (1 and 1);
\node [fill, circle, inner sep=1.3] (v6) at (6.5,-2) {};
\node [fill, circle, inner sep=1.3] (v8) at (5.5,-3) {};
\node [fill, circle, inner sep=1.3] (v7) at (6.5,-4) {};
\node [fill, circle, inner sep=1.3] (v9) at (7.5,-3) {};
\node [fill, circle, inner sep=1.3] (v12) at (6.5,-3) {};
\node [fill, circle, inner sep=1.3] (v10) at (9.5,-2) {};
\node [fill, circle, inner sep=1.3] (v11) at (9.5,-4) {};
\node [fill, circle, inner sep=1.3] at (9.5,-3) {};
\draw [blue]  (v8) edge (v9);
\draw [blue]  (v10) edge (v11);
\node [blue] at (6.55,-2.35) {\scriptsize 1};
\node [blue] at (6,-3.2) {\scriptsize 2,3};
\node [blue] at (6.3,-3.6) {\scriptsize 4};
\node [blue] at (7,-3.2) {\scriptsize 5,6};
\node [blue] at (9.8,-2.55) {\scriptsize 6,2};
\node [blue] at (9.1,-3.45) {\scriptsize 5,4,3};
\node at (3.75,-2.6) {$\pi_{\mathcal{R}}$};
\node [blue] at (-1.9,-1.1) {$f_1(\triangle)$};

\draw [blue](-2,-3) .. controls (-1.5,-1) and (1.5,-1) .. (2,-3);
\draw [blue](0,9.5) .. controls (-1.45,9.5) and (-0.5,11.5) .. (0,12);
\draw [blue] (v12) edge (v7);
\draw [blue](5.5,-3) .. controls (5.6,-2.35) and (7.4,-2.35) .. (7.5,-3);
\node at (0.5,0.25) {$f_1$};
\node at (0.5,6.75) {$f_2$};
\draw [-implies, double distance=2pt](8,1) -- (8,-0.5);
\draw [double distance=2pt](8,6) -- (8,7.5);
\node at (8.5,0.25) {$f_1$};

\node [blue] at (8,11.5) {$\overline{f_2(\triangle)}$};
\node [blue] at (8,5) {$\overline{\triangle}$};
\node [blue] at (8,-1.5) {$\overline{f_1(\triangle)}$};
\end{tikzpicture}
    \caption{An example of the cutting procedure and the induction of flips. Left column: the surface $\Sigma$ and the reduction system $\cR$. Right column: the reduced surface $\Sigma_\cR$.\CORRECTED}
    \label{fig:cutting_example}
\end{figure}
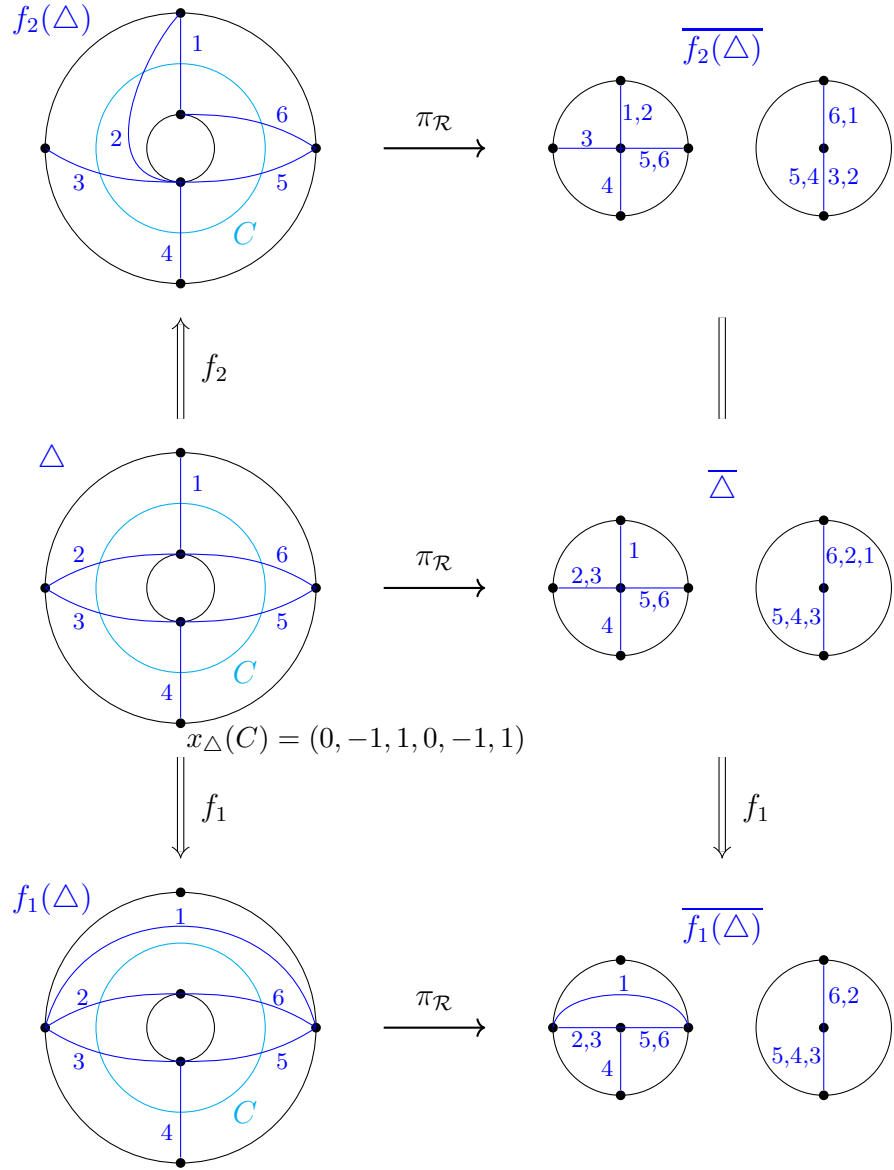

\end{ex}

Now let us study the differentiable points of the map \eqref{eq:pi_S}.
\begin{prop}\label{lem:pi_full_rk}
The set of points at which the map \eqref{eq:pi_S} is intrinsically differentiable contains the open $\Gamma_{\bs_\Sigma}$-invariant subset $\mathfrak{U}_\Sigma:=\mathfrak{U}_{\bs_\Sigma}$ (see \cref{lem:differentiable_domain}). Moreover, it is dense in $ \cX_\Sigma^\uf(\bR^\trop)$.
\end{prop}

\begin{proof}
Fix a tagged triangulation $\tri \in \Exch_\Sigma$ and consider the coordinate expression $\pi_\tri$ of the map $\pi_\cR$, which makes the following diagram commutative:
\[
\begin{tikzcd}[row sep=small]
\cX^\uf_{\Sigma}(\bR^\trop) \ar[r, "\pi_\cR"] & \cX_{\Sigma_\cR}^\uf(\bR^\trop)\\
\cX^\uf_{\tri}(\bR^\trop) \ar[u,"\rotatebox{90}{$\sim$}"] \ar[r, "\pi_{\tri}"'] & \cX_{\oline{\tri}}^\uf(\bR^\trop). \ar[u,"\rotatebox{90}{$\sim$}"']
\end{tikzcd}
\]
First we prove that $\pi_{\tri}$ is differentiable in the interior of $\cC_{\tri}^+$, and compute the tangent map there.
Recall from \cite[Lemma]{IK20} that a point $\hL \in \interior\cC_{\tri}^+$ is given by a collection of real weighted ideal arcs parallel to the ideal triangulation $\tri$. 
By \cref{lem:pi_S_ideal_arc}, the homotopy class of the support of $\pi_\cR(\hL)$ again coincides with $\pi_\cR(\tri)$. Hence we get\footnote{The formula looks similar to those for the amalgamation \cite[(11)]{FG06} or the cluster folding \cite[Definition 3.7]{FG06}.} 
\[x_{\alpha'}^{\oline{\tri}}(\pi_\cR(\hL)) = \sum_{\substack{\alpha \in \tri\\ \pi_\cR(\alpha) \ni \alpha'}} x_{\alpha}^\tri(\hL),\]
for all $\hL \in \cC_\tri^+$, 
and thus $\pi_\tri$ is differentiable on $\interior \cC_\tri^+$.

For any path $\gamma:\tri \to \tri'$ in $\Exch_\Sigma$, the map $\mu_\gamma$ is linear on $\interior\cC^+_{\tri}$ (\cite[Lemma 3.12]{IK19}). A similar statement holds for the path $\oline{\gamma}:\oline{\tri} \to \oline{\tri'}$ in $\Exch_{\oline{\Sigma}}$ obtained by applying the graph projection $\Exch_\Sigma \to \Exch_{\bar{\Sigma}}$ to $\gamma$. 
Since $\mu_{\oline{\gamma}}\circ\pi_\tri=\pi_{\tri'}\circ \mu_\gamma$ by \cref{lem:collapsed_tri}, $\pi_{\tri'}$ is linear on $\interior\cC^+_{\tri'}$, we see that $\pi_\tri$ is differentiable as well on $\mu_\gamma^{-1}(\interior\cC_{\tri'}^+) \subset \cX^\uf_{\tri}(\bR^\trop)$. 
A similar argument applies for the cones $\cC_{\tri'}^-$. 
Hence $\pi_\tri$ is differentiable at any point in the set $\mathfrak{U}_\tri \subset \X_\tri^\uf(\bR^\trop)$.
In our case, it is known \cite{Yur} that the subset $|\mathfrak{F}_{\bs_\Sigma}^+| \cup |\mathfrak{F}_{\bs_\Sigma}^-|$ is dense in $\X_\Sigma^\uf(\bR^\trop)$, and so is the open subset $\mathfrak{U}_\Sigma$. 
\end{proof}

\subsection{Reduction of representation paths}\label{subsec:reduction_path}
Although \cref{lem:collapsed_tri} does not tell us how to project each vertical edge in $\bExch_\Sigma$, it turns out that we can project a representation path of a mutation loop. 

Fix $(\tri_0, \ell_0) \in \bExch_\Sigma$. 
For any path $\gamma: (\tri_0, \ell_0) \to (\tri, \ell)$ which represents a mutation loop $\phi$, we are going to associate a path $\oline{\gamma}$ in $\bExch_{\bar{\Sigma}}$ which represents the mutation loop $\pi_\cR(\phi)$. 
First deform $\gamma$ by a natural deformation into the form $\gamma = \gamma_{\mathrm{hori}} * \gamma_{\mathrm{ver}}$, where 
\begin{align*}
    \gamma_{\mathrm{hori}}: (\tri_0, \ell_0) \overbar{k_0} (\tri_1, \ell_1) \overbar{k_1} \cdots \overbar{k_{h-1}} (\tri, \ell_h)
\end{align*}
consists of horizontal edges and 
\begin{align*}
    \gamma_{\mathrm{ver}}: (\tri, \ell_h) \xrightarrow{\sigma} (\tri, \ell)
\end{align*}
consists of vertical edges. 
By definition, we have $\tri = \phi^{-1}(\tri_0)$ and $\sigma.\ell_h = \ell = \phi^{-1} \circ \ell_0$.
Recall from \cite[Lemma 5.8]{IK20} that this deformation preserves the sign of the path. 
Further we fix a labeling $\lambda_0$ of the tagged triangulation $\oline{\tri_0}$, namely we fix a vertex $(\oline{\tri_0}, \lambda_0) \in \bExch_{\Sigma_\cR}$.
Then, we define a path $\oline{\gamma}$ in $\bExch_{\Sigma_\cR}$ as follows:
\begin{itemize}
    \item Take the subsequence $(\alpha_0, \alpha_1, \dots, \alpha_{\oline{h}-1})$ of $(\ell_0(k_0), \ell_1(k_1), \dots, \ell_{h-1}(k_{h-1}))$ of $\cR$-compatible arcs, where $\oline{h} \leq h$. Let 
    \[
    \oline{\gamma_{\mathrm{hori}}}: (\oline{\tri_0}, \lambda_0) \xrightarrow{(\lambda_0^{-1}(\oline{\alpha}_0), \dots, \lambda_{\oline{h}-1}^{-1}(\oline{\alpha}_{\oline{h}-1}))}(\oline{\tri}, \lambda_{\oline{h}}).
    \]
    be the corresponding horizontal path starting from $(\oline{\tri_0},\lambda_0)$, where the labelings $\lambda_i$ for $i = 0, \dots, \oline{h}$ are determined inductively by the labeled flip.
    We have $\oline{\tri}=\oline{\phi^{-1}(\tri_0)}=\pi_\cR(\phi)^{-1}(\oline{\tri_0})$ by \cref{lem:collapsed_tri,lem:pi_S_equivariance,lem:pi_S_ideal_arc}.  
    \item Take a permutation $\rho \in \fS_{I(\Sigma_\cR)}$ satisfying $\rho.\lambda_{\oline{h}} = \pi_\cR(\phi)^{-1}\circ\lambda_0$, and let 
    \begin{align*}
        \oline{\gamma_{\mathrm{ver}}}: (\oline{\tri},\lambda_{\oline{h}}) \xrightarrow{\rho} (\oline{\tri},\rho.\lambda_{\oline{h}})
    \end{align*}
    be some corresponding vertical path.
\end{itemize}
Then define $\oline{\gamma} := \oline{\gamma_{\mathrm{hori}}} * \oline{\gamma_{\mathrm{ver}}}$, which is a path in $\bExch_{\Sigma_\cR}$.

\begin{prop}\label{prop:pi_MCG}
The above path $\oline{\gamma}$ represents the mutation loop $\pi_\cR(\phi) \in \Gamma_{\Sigma_\cR}$. 
If the path $\gamma$ has a strict sign at a point $\hL \in \X_\Sigma^\uf(\bR^\trop)$, then so is the path $\oline{\gamma}$ at the point $\pi_\cR(\hL)$. 
\end{prop}

\begin{proof}
By construction, we have
\begin{align*}
    (\oline{\tri},\rho.\lambda_{\oline{h}})=(\pi_\cR(\phi)^{-1}(\oline{\tri_0}),\pi_\cR(\phi)^{-1}\circ \lambda_0) = \pi_\cR(\phi)^{-1}(\tri_0,\lambda_0).
\end{align*}
Thus, the path $\oline{\gamma}$ represents $\pi_\cR(\phi)$. 
The second statement follows, since the sign $\boldsymbol{\epsilon}_{\oline{\gamma}}(\pi_\cR(\hL))$ is a subsequence of the sign $\boldsymbol{\epsilon}_{\gamma}(\hL)$. 
\end{proof}

In particular, the reduction homomorphism $\pi_\cR:\Gamma_{\Sigma,\cR} \to \Gamma_{\Sigma_\cR}$ is given by $\phi=[\gamma] \mapsto \pi_\cR(\phi)=[\oline{\gamma}]$. 

\begin{cor}\label{cor:reduction_SS}
If a path $\gamma:(\tri_0,\ell_0) \to (\tri,\ell)$ in $\bExch_\Sigma$ represents a mutation loop $\phi \in \Gamma_{\Sigma,\cR}$ and sign-stable on an $\bR_{>0}$-invariant subset $\Omega \subset \X_{(\tri_0,\ell_0)}^\uf(\bR^\trop)$, then the path $\oline{\gamma}$ is sign-stable on $\pi_\cR(\Omega) \subset \X_{(\oline{\tri_0},\lambda_0)}^\uf(\bR^\trop)$.
\end{cor}

\begin{proof}
For any $\hL \in \Omega \setminus \{0\}$, the sign stability of $\gamma$ tells us that $\boldsymbol{\epsilon}_\gamma(\phi^n(\hL)) = \boldsymbol{\epsilon}_\gamma^\stab$ with a strict sign $\boldsymbol{\epsilon}_\gamma^\stab$ for sufficiently large $n$. Now write an arbitrary point in $\pi_\cR(\Omega) \setminus \{0\}$ as $\pi_\cR(\hL)$ for some $\hL \in \Omega \setminus \{0\}$. Then 
\begin{align*}
    \boldsymbol{\epsilon}_{\oline{\gamma}}(\pi_\cR(\phi^n)(\pi_\cR(\hL))) = \boldsymbol{\epsilon}_{\oline{\gamma}}(\pi_\cR(\phi^n(\hL))) 
\end{align*}
by \cref{lem:pi_S_equivariance}, and the latter is strict by \cref{prop:pi_MCG}. Thus $\oline{\gamma}$ is sign-stable on $\pi_\cR(\Omega)$. 
\end{proof}


\subsection{Block decomposition of the infinitesimal action}
Using the map $\pi_\cR: \X_\Sigma^\uf(\bR^\trop) \to \X_{\Sigma_\cR}^\uf(\bR^\trop)$, which is intrinsically differentiable almost everywhere by \cref{lem:pi_full_rk}, we give a block decomposition of the tangent map $d\phi$ of a $\cR$-reducible mapping class $\phi \in MC_\cR(\Sigma)$. 

In this subsection, we write $\hL^*:=\pi_\cR(\hL)$ for $\hL \in \mathfrak{U}_\Sigma \subset \X_\Sigma^\uf(\bR^\trop)$, and $\phi^*:=\pi_\cR(\phi)$ for $\phi \in MC_\cR(\Sigma)$.
Note that $\phi(\hL)^*=\phi^*(\hL^*)$ for all $\hL \in \mathfrak{U}_\Sigma$ and $\phi \in MC_\cR(\Sigma)$.
Let us begin with the study of the image and the kernel of the tangent map of $\pi_\cR$.

\begin{lem}\label{lem:im_pi_S}
The image of the map $\pi_\cR$ is characterized by the conditions $\theta_{p_k^+} = \theta_{p_k^-} \geq 0$ for all $k=1,\dots,K$. In particular, for any $\hL \in \mathfrak{U}_\Sigma$ with $\cI_{C_k}(\hL) >0$ for $k=1,\dots,K$, the image of the tangent map 
\begin{align*}
    (d\pi_\cR)_{\hL}: T_{\hL} \X_\Sigma^\uf(\bR^\trop) \to T_{\hL^*}\X_{\Sigma_\cR}^\uf(\bR^\trop)
\end{align*}
is characterized by the linear equations $\theta_{p_k^+} = \theta_{p_k^-}$ for all $k=1,\dots,K$.
\end{lem}

\begin{proof}
By construction, it is clear that any lamination in the image of $\pi_\cR$ satisfies the conditions $\theta_{p_k^+} = \theta_{p_k^-} \geq 0$ for all $k=1,\dots,K$. Conversely, 
let $\hL' \in \X_{\Sigma_\cR}^\uf(\bR^\trop)$ be a measured geodesic lamination on the hyperbolic surface $F_\cR$ satisfying the conditions. Then the equality $\theta_{p_k^+}(\hL') = \theta_{p_k^-}(\hL')=:w_k$ tells us that the total transverse measures of the leaves spiralling around the both sides of $C_k$ are equal, so we can somehow glue these leaves together to get a geodesic measured lamination on $F$. This can be done as follows. 

First straighten the spiralling leaves so that they hit the curve $C_k$ in $F$ transversely. 
Fix a linear ordering on the resulting geodesics on each side, and label them as $g_1,\dots,g_{n_1}$ (with transverse measures $\mu_1,\dots,\mu_{n_1}$) on one side, and $g'_1,\dots,g'_{n_2}$ (with transverse measures $\nu_1,\dots,\nu_{n_2}$) on the other. 
Then we get two partitions $(\mu_1,\dots,\mu_{n_1})$ and $(\nu_1,\dots,\nu_{n_2})$ of the total measure $w_k$. Take their common refinement $(\rho_1,\dots,\rho_m)$, each of which corresponds to a geodesic of the form $g_i \cup g'_j$ with the transverse measure $\rho_q$ for $q=1,\dots,m$. Performing this construction for each $k=1,\dots,K$, we get a measured geodesic lamination $\hL$ which is clearly projected to $\hL'$ via the projection $\pi_\cR$.
\end{proof}

Recall the notations on train tracks from \cref{sec:train track}.
Let us consider the open dense subset
\begin{align*}
    \mathfrak{V}_{\Sigma, \cR} := \bigcup_{\tau \in CT_\cR} \interior \widehat{V}(\tau) \subset \cX_\Sigma^\uf(\bR^\trop).
\end{align*}
Here, $CT_\cR$ is the set of complete train tracks which carry the curves in $\cR$. 
Since $CT_\cR$ is $\Gamma_{\Sigma,\cR}$-invariant, so is $\mathfrak{V}_{\Sigma, \cR}$.
The density directly follows from \cref{thm:tt_cone_decomp_gen} and \cref{rem:connector_folded} if $\cR$ is a pants decomposition.
If not, consider a pants decomposition $\widetilde{\cR}$ containing $\cR$. Then $CT_{\widetilde{\cR}} \subset CT_{\cR}$ and $\mathfrak{V}_{\Sigma,\widetilde{\cR}} \subset \mathfrak{V}_{\Sigma,\cR}$, hence the latter is also dense. 
Indeed, the subset $\mathfrak{V}_{\Sigma, \cR}$ is obtained from $\mathfrak{V}_{\Sigma, \widetilde{\cR}}$ by adding the faces of $\widehat{V}(\tau)$ corresponding to the curves in $\widetilde{\cR} \setminus \cR$ by \cref{thm:tt_atlas}.

For $\tau \in CT_\cR$, let $W(\tau)$ denote the vector space of functions $\nu: \{\mbox{edges of } \tau\} \to \bR$ satisfying the switch conditions. Let $V_0(\tau) \subset \widehat{V}(\tau)$ be the subset consisting of the enhanced measures on $\tau$ with zero measures on the edges transverse to the boundary (if exists). Then we have a canonical inclusion $V_0(\tau) \subset W(\tau)$, and an identification
\begin{align*}
    T_{\hL} \X_\Sigma^\uf(\bR^\trop) \xrightarrow{(d\psi_\tau^{-1})_{\hL}} T_{\widehat{\nu}}\widehat{V}(\tau) \cong W(\tau)
\end{align*}
for $\hL=\psi_\tau(\widehat{\nu})$. Here the latter isomorphism is given by $\left.\frac{d}{ds}\middle|\right._{s=0}(\widehat{\nu}+s \mu) \mapsto \mu$. 

\begin{lem}\label{lem:ker_pi_S}
For $\hL \in \mathfrak{U}_{\Sigma} \cap \mathfrak{V}_{\Sigma,\cR}$, let $\tau$ be a complete train track which carries $\hL$ and $C_k$ for $k=1,\dots,K$ (recall \cref{rem:connector_folded}). Then under the identification $T_{\hL}\X_\Sigma^\uf(\bR^\trop) \cong W(\tau)$, we get
\begin{align*}
    \ker (d\pi_\cR)_{\hL} = T_{\hL} (\pi_\cR^{-1}(\hL^*)) \cong \bigoplus_{k=1}^K \bR C_k.
\end{align*}
Here $C_k \in V_0(\tau)$ is regarded as a tangent vector via the canonical inclusion $V_0(\tau) \subset W(\tau)$.
\end{lem}

\begin{proof}
Let $\hL_s \in \mathfrak{U}_\Sigma \cup \mathfrak{V}_{\Sigma,\cR}$ be a differentiable path with $\hL_0=\hL$. Since $\mathfrak{U}_\Sigma \cup \mathfrak{V}_{\Sigma,\cR}$ is open, there exists a path $\widehat{\nu}_s=(\nu_s,\sigma_{\nu}) \in \interior \widehat{V}(\tau)$ such that $\psi_\tau(\widehat{\nu}_s) = \hL_s$ and having the constant signature $\sigma_{\nu}$ for sufficiently small $s \in \bR$. 
For $k=1,\dots,K$, consider the subtrack of $\tau$ around the curve $C_k$ as shown in \cref{fig:connector_ver_Hat}, say of the type $\tau^{\epsilon_k}$ for some $k \in \{+,-\}$, whose edges are now labeled as $e_{1,(k)}^{\epsilon_k}$, $e_{2,(k)}^{\epsilon_k}$ and $e_{3,(k)}^{\epsilon_k}$. 

Suppose that the path $\hL_s$ lies in the fiber $\pi_\cR^{-1}(\hL^*)$. Then the measure $\nu_s(e)$ on an edge $e$ must be constant, except for the edges $e_{2,(k)}^{\epsilon_k}$ and $e_{3,(k)}^{\epsilon_k}$ for $k=1,\dots,K$. 
Moreover, the signature $\sigma_\nu$ must be constant. 
Hence the derivative $\left.\frac{d}{ds}\middle|\right._{s=0} \nu_s(e)$ is zero except for these edges, and we have 
\begin{align*}
    \left.\frac{d}{ds}\middle|\right._{s=0} \nu_s(e_{2,(k)}^{\epsilon_k})=\left.\frac{d}{ds}\middle|\right._{s=0} \nu_s(e_{3,(k)}^{\epsilon_k})=:w_k
\end{align*}
by the switch condition. Then the measure $\mu(\boldsymbol{w}) \in W(\tau)$ given by $\mu(e_{2,(k)}^{\epsilon_k})=\mu(e_{3,(k)}^{\epsilon_k}):=w_k$ for $k=1,\dots,K$ and $\mu(e):=0$ for other edges  coincides with the image of the vector $\sum_k w_kC_k$ under the canonical inclusion $V_0(\tau) \subset W(\tau)$, and thus we get 
\begin{align*}
    T_{\hL}(\pi_\cR^{-1}(\hL^*)) \to \bigoplus_{k=1}^K \bR C_k, \quad \left.\frac{d}{ds}\middle|\right._{s=0}\hL_s   \mapsto \mu(\boldsymbol{w})=\sum_k w_kC_k,
\end{align*}
which is clearly a linear isomorphism.
\end{proof}

For an $\cR$-reducible mapping class $\phi \in MC_\cR(\Sigma)$, let $\sigma_\cR(\phi) \in \fS_K$ denote the induced permutation of the curves in $\cR$.

\begin{lem}
Let $\phi \in MC_\cR(\Sigma)$.
Then the restriction 
\begin{align*}
    (d\phi)_{\hL}: T_{\hL} (\pi_\cR^{-1}(\hL^*)) \to T_{\phi(\hL)} (\pi_\cR^{-1}(\phi(\hL)^*))
\end{align*}
of the tangent map is presented by the matrix $\sigma_\cR(\phi)$ with respect to the basis $(C_1,\dots,C_K)$ given in \cref{lem:ker_pi_S}.
\end{lem}

\begin{proof}
Recall that for any mapping class $\phi' \in MC(\Sigma)$ and a train track $\tau$, we have a bijection $\phi'_*: \{\mbox{edges of $\tau$}\} \xrightarrow{\sim} \{\mbox{edges of $\phi'(\tau)$}\}$ and 
\begin{align*}
    \phi'(\psi_\tau((\nu,\varsigma_\nu))) = \psi_{\phi'(\tau)}((\nu\circ (\phi'_*)^{-1}, \sigma_\Sigma(\phi')^*\varsigma_\nu))
\end{align*}
for any enhanced measure $\widehat{\nu}=(\nu,\varsigma_\nu) \in \widehat{V}(\tau)$ by the equivariance of the construction explained in \cref{subsec:train track_lamination}.
Let us write $\phi(\widehat{\nu}):=(\nu\circ (\phi'_*)^{-1}, \sigma_\Sigma(\phi')^*\varsigma_\nu)$. 

For $\boldsymbol{w}=(w_1,\dots,w_K) \in \bR^K$, let $\hL_s(\boldsymbol{w}):=\psi_\tau(\widehat{\nu}+s\mu(\boldsymbol{w}))$ be a path in the fiber $\pi_\cR^{-1}(\hL^*)$ whose derivative gives $\mu(\boldsymbol{w})=\sum_k w_kC_k$. Then we get
\begin{align*}
    \phi(\hL_s(\boldsymbol{w})) = \psi_{\phi(\tau)}(\phi(\widehat{\nu})+s\mu(\sigma.\boldsymbol{w})),
\end{align*}
where $\sigma:=\sigma_\cR(\phi)$; although the types of the subtracks of $\phi(\tau)$ around the curves in $\cR$ may differ from those of $\tau$, the edges $e_{2,(k)}^{\epsilon_k}$ and $e_{3,(k)}^{\epsilon_k}$ are shared and just permuted by $\phi \in MC_\cR(\Sigma)$. 
Hence
\begin{align*}
    (d\phi)_{\hL}(\mu(\boldsymbol{w})) = \left.\frac{d}{ds}\middle|\right._{s=0}\phi(\hL_s(\boldsymbol{w})) = \mu(\sigma.\boldsymbol{w}) = \sum_k w_{\sigma(k)}C_k
\end{align*}
as desired.
\end{proof}

As a consequence of the discussion above, we get the following:

\begin{thm}\label{thm:block decomposition}
For $\hL \in \mathfrak{U}_{\Sigma,\cR}$ and $\phi \in \Gamma_{\Sigma,\cR}$, choose arbitrary linear sections $s_{\hL}: \im (d\pi_\cR)_{\hL} \to T_{\hL}\X_\Sigma^\uf(\bR^\trop)$ and $s_{\phi(\hL)}: \im (d\pi_\cR)_{\phi(\hL)} \to T_{\phi(\hL)}\X_\Sigma^\uf(\bR^\trop)$. Then with respect to the direct sum decompositions
\begin{align*}
    T_{\hL}\X_\Sigma^\uf(\bR^\trop) &= T_{\hL} (\pi_\cR^{-1}(\hL^*)) \oplus s_{\hL}(\im (d\pi_\cR)_{\hL}), \\
    T_{\phi(\hL)}\X_\Sigma^\uf(\bR^\trop) &= T_{\phi(\hL)} (\pi_\cR^{-1}(\phi(\hL)^*)) \oplus s_{\phi(\hL)}(\im (d\pi_\cR)_{\phi(\hL)}),
\end{align*}
the tangent map $(d\phi)_{\hL}$ is block-decomposed as
\begin{align*}
    (d\phi)_{\hL}=\begin{pmatrix}
    \sigma_\cR(\phi) & \ast \\
    0 & \overline{(d\phi^*)}_{\hL^*}
    \end{pmatrix}.
\end{align*}
Here $\overline{(d\phi^*)}_{\hL^*}:=(d\phi^*)_{\hL^*}|_{\im (d\pi_\cR)_{\hL}}: \im (d\pi_\cR)_{\hL} \to \im (d\pi_\cR)_{\phi(\hL)}$.
\end{thm}

\begin{proof}
When $\phi \in MC_\cR(\Sigma)$, the assertion directly follows from the discussions above. For a general mutation loop $\phi=(\psi,\varsigma)$, the tangent map factorizes as $(d\phi)_{\hL}=(d\psi)_{\varsigma(\hL)} \circ (d\varsigma)_{\hL}$. Since the paths $\hL_s$ in the fiber $\pi_\cR^{-1}(\hL^*)$ is invariant under reflections, we have the block-decomposition
\begin{align*}
    (d\varsigma)_{\hL}=\begin{pmatrix}
    1 & \ast \\
    0 & \overline{(d\varsigma^*)}_{\hL^*}
    \end{pmatrix}.
\end{align*}
Then the assertion is clear. 
\end{proof}

Since the spectral radius of the permutation matrix $\sigma_\cR(\phi)$ is $1$, we get the following:

\begin{cor}\label{cor:reduction_spec}
For a mutation loop $\phi \in \Gamma_{\Sigma,\cR}$ which admits a representation path $\gamma:(\tri_0,\ell_0) \to (\tri,\ell)$ in $\bExch_\Sigma$ as in \cref{cor:reduction_SS}, we have $\rho(E_{\phi,\Omega}^{(\tri_0,\ell_0)})= \rho(E_{\pi_\cR(\phi),\pi_\cR(\Omega)}^{(\oline{\tri_0},\lambda_0)})$. Here $\lambda_0$ is an arbitrary labeling  of $\oline{\tri}_0$.
\end{cor}

\begin{rem}\label{rem:pi_component}
When $\Sigma_\cR$ is disconnected, the group $\Gamma_{\Sigma_\cR}$ is decomposed as follows. 
Let $\Sigma_1,\dots,\Sigma_r$ be the connected components of $\Sigma_\cR$. Regroup them into homeomorphism classes, and write
\begin{align*}
    \{\Sigma_1,\dots,\Sigma_r\} = \bigsqcup_{\nu=1}^m \{\Sigma_1^{(\nu)},\dots,\Sigma_{i_\nu}^{(\nu)}\}. 
\end{align*}
Fix a model surface $\Sigma_\nu$ for each homeomorphism class $\nu=1,\dots,m$, and a marking homeomorphism $f_k^{(\nu)}:\Sigma_k^{(\nu)} \xrightarrow{\sim} \Sigma_\nu$ for $k=1,\dots,i_\nu$. Let $f_{k'k}^{(\nu)}:=(f_{k'}^{(\nu)})^{-1}\circ f_k^{(\nu)}: \Sigma_k^{(\nu)} \xrightarrow{\sim} \Sigma_{k'}^{(\nu)}$. When $\Sigma_i=\Sigma_k^{(\nu)}$ and $\Sigma_{i'}=\Sigma_{k'}^{(\nu)}$, we also write $f_{i'i}:=f_{k'k}^{(\nu)}$. 

Let $\fS_\cR:=\prod_{\nu=1}^m \fS_{i_\nu}$, which acts on $\{1,\dots,r\}$ in the natural way. We have a group homomorphism $\tau_\cR: \Gamma_{\Sigma_\cR} \to \fS_\cR$ defined by the permutation of connected components. 
Then we get a group isomorphism
\begin{align*}
    \iota_\cR: \Gamma_{\Sigma_\cR} \xrightarrow{\sim} \fS_\cR \ltimes \prod_{\nu=1}^r \Gamma_{\Sigma_i},\quad \phi \mapsto (\tau_\cR(\phi), (f_{i,\tau_\cR(\phi)(i)}\circ (\phi|_{\Sigma_i}))_{i=1}^r).
\end{align*}
The inverse map is given by $\iota^{-1}(\tau,(\phi_i)_{i=1}^r)|_{\Sigma_i} = f_{\tau(i),i}\circ \phi_i$. Note that we also have a PL isomorphism
\begin{align*}
    \iota_\cR: \X_{\Sigma_\cR}^\uf(\bR^\trop) \xrightarrow{\sim} \prod_{i=1}^r \X_{\Sigma_i}^\uf(\bR^\trop),
\end{align*}
which is $\Gamma_{\Sigma_\cR}$-equivariant. Here $(\tau,(\phi_i)_i) \in \fS_\cR \ltimes \prod_i \Gamma_{\Sigma_i}$ acts on $\prod_i \X_{\Sigma_i}^\uf(\bR^\trop)$ by 
\begin{align*}
    (\hL_i)_{i=1}^r \mapsto (\phi_{\tau^{-1}(i)}(\hL_{\tau^{-1}(i)}))_{i=1}^r.
\end{align*}
\end{rem}

\section{Sign stability of pseudo-Anosov mapping classes: general case}\label{sec:weak_SS}
In this section, we define an appropriate notion of pA mapping class on a general marked surface via reductions, and discuss their characterization in terms of the sign stability.

\subsection{Dynamics on the space of \texorpdfstring{$\cX$}{X}-laminations.}\label{subsec:dyn_gen_X}
Let $\Sigma$ be a marked surface with possibly non-empty boundary. Let $\cR_\partial=\{ C_i \}_{i=1}^b$ be the multicurve on $\Sigma$ consisting of the simple closed curve $C_i$ parallel to the $i$-th boundary component $\partial_i$ for $i=1,\dots,b$. Then we get
\begin{align*}
    \Sigma_{\cR_\partial} = \bar{\Sigma} \sqcup \bD(\partial_1) \sqcup \dots \sqcup \bD(\partial_b),
\end{align*}
where $\bar{\Sigma}$ is the connected component disjoint from $\partial \Sigma$, which is a punctured surface; $\bD(\partial_i)$ is the connected component containing $\partial_i$ for $i=1,\dots,b$, which is a once-punctured disk with some special points on its boundary. See \cref{fig:pi_partial}.

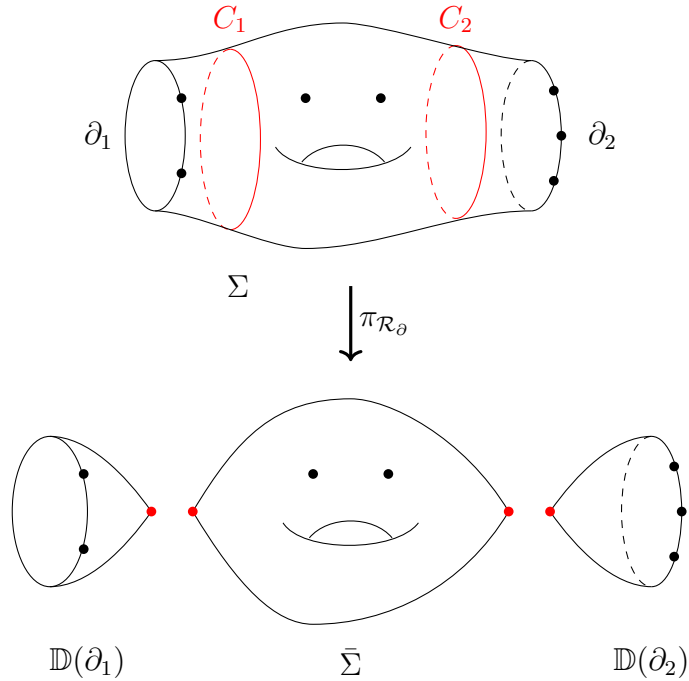
\begin{figure}[h]
    \centering
    \begin{tikzpicture}
    \draw  (-2.6,3) ellipse (0.4 and 1);
    \draw (2.4,2) coordinate (v2) arc [start angle=-90, end angle=90, x radius=.4cm, y radius=1cm];
    \draw [dashed] (2.4,4) coordinate (v1) arc [start angle=90, end angle=270, x radius=.4cm, y radius=1cm];
    
    \draw [red] (1.4,1.9) coordinate (v2) arc [start angle=-90, end angle=90, x radius=.4cm, y radius=1.15cm];
    \draw [red, dashed] (1.4,4.2) coordinate (v1) arc [start angle=90, end angle=270, x radius=.4cm, y radius=1.15cm];
    
    \draw [red] (-1.6,1.75) coordinate (v2) arc [start angle=-90, end angle=90, x radius=.4cm, y radius=1.2cm];
    \draw [red, dashed] (-1.6,4.15) coordinate (v1) arc [start angle=90, end angle=270, x radius=.4cm, y radius=1.2cm];
    
    \draw (-2.6,4) .. controls (-1.65,4) and (-1.1,4.5) .. (-0.1,4.5) .. controls (0.4,4.5) and (1.9,4) .. (2.4,4);
    \draw (-2.6,2) .. controls (-2.1,2) and (-1.1,1.5) .. (-0.6,1.5) .. controls (0.4,1.5) and (1.4,2) .. (2.4,2);
    \draw (-1,2.85) .. controls (-0.75,2.45) and (0.5,2.45) .. (0.8,2.85);
    \draw (-0.65,2.65) .. controls (-0.4,2.95) and (0.2,2.95) .. (0.45,2.65);
    \node [fill, circle, inner sep=1.3] at (-0.6,3.5) {};
    \node [fill, circle, inner sep=1.3] at (0.4,3.5) {};
    \node [fill, circle, inner sep=1.3] at (-2.25,3.5) {};
    \node [fill, circle, inner sep=1.3] at (-2.25,2.5) {};
    \node [fill, circle, inner sep=1.3] at (2.7,3.6) {};
    \node [fill, circle, inner sep=1.3] at (2.8,3) {};
    \node [fill, circle, inner sep=1.3] at (2.7,2.4) {};
    
    \draw (-2.1,-2) .. controls (-1.5,-1) and (-1,-0.5) .. (0,-0.5) .. controls (0.5,-0.5) and (1.5,-1) .. (2.1,-2);
    \draw (-2.1,-2) .. controls (-1.5,-3) and (-1,-3.5) .. (-0.5,-3.5) .. controls (0.5,-3.5) and (1.5,-3) .. (2.1,-2);
    \draw (-0.9,-2.15) .. controls (-0.65,-2.55) and (0.6,-2.55) .. (0.9,-2.15);
    \draw (-0.55,-2.35) .. controls (-0.3,-2.05) and (0.3,-2.05) .. (0.55,-2.35);
    \node [fill, circle, inner sep=1.3] at (-0.5,-1.5) {};
    \node [fill, circle, inner sep=1.3] at (0.5,-1.5) {};
    
    \draw  (-4,-2) ellipse (0.5 and 1);
    \draw (4,-3) coordinate (v2) arc [start angle=-90, end angle=90, x radius=.4cm, y radius=1cm];
    \draw [dashed] (4,-1) coordinate (v1) arc [start angle=90, end angle=270, x radius=.4cm, y radius=1cm];
    
    \draw (-4,-1) .. controls (-3.5,-1) and (-3,-1.5) .. (-2.65,-2);
    \draw (-4,-3) .. controls (-3.5,-3) and (-3,-2.5) .. (-2.65,-2);
    \draw (2.65,-2) .. controls (3,-1.5) and (3.5,-1) .. (4,-1);
    \draw (2.65,-2) .. controls (3,-2.5) and (3.5,-3) .. (4,-3);
    \node [fill, circle, inner sep=1.3] at (-3.55,-1.5) {};
    \node [fill, circle, inner sep=1.3] at (-3.55,-2.5) {};
    \node [fill, circle, inner sep=1.3] at (4.3,-1.4) {};
    \node [fill, circle, inner sep=1.3] at (4.4,-2) {};
    \node [fill, circle, inner sep=1.3] at (4.3,-2.6) {};
    \node [fill, circle, inner sep=1.3, red] at (-2.65,-2) {};
    \node [fill, circle, inner sep=1.3, red] at (-2.1,-2) {};
    \node [fill, circle, inner sep=1.3, red] at (2.1,-2) {};
    \node [fill, circle, inner sep=1.3, red] at (2.65,-2) {};
    \node [red] at (-1.6,4.55) {$C_1$};
    \node [red] at (1.4,4.55) {$C_2$};
    \node at (-1.5,1) {$\Sigma$};
    \node at (-3.5,-4) {$\mathbb{D}(\partial_1)$};
    \node at (0,-4) {$\bar{\Sigma}$};
    \node at (4,-4) {$\mathbb{D}(\partial_2)$};
    \node at (-3.35,3) {$\partial_1$};
    \node at (3.35,3) {$\partial_2$};
    \draw [->, very thick](0,1) -- (0,0);
    \node at (0.45,0.5) {$\pi_{\cR_\partial}$};
    \end{tikzpicture}
    \caption{An example of $\Sigma_{\cR_\partial}$.}
    \label{fig:pi_partial}
\end{figure}

Notice that $\Gamma_{\Sigma,\cR_\partial} = \Gamma_\Sigma$, since each mapping class preserves $\partial \Sigma$ setwise and hence perserves the set of boundary-parallel curves.
In view of \cref{rem:pi_component}, let us write the reduction homomorphism \eqref{eq:pi_S_cluster_modular} as
\begin{align*}
    \iota_{\cR_\partial}\circ \pi_{\cR_\partial}=(\tau;\pi,\pi_1,\dots,\pi_b): \Gamma_\Sigma \to \fS_{\cR_\partial} \ltimes \bigg(\Gamma_{\bar{\Sigma}} \times \prod_{i=1}^b \Gamma_{\bD(\partial_i)}\bigg),
\end{align*}
whose kernel being generated by the boundary Dehn twists $t_{C_i}$ for $i=1,\dots,b$. 
Fixing a hyperbolic structure $F$ on $\Sigma$, we identify $\X_\Sigma^\uf(\bR^\trop) \cong \eML(F)$.
The projection \eqref{eq:pi_S} is now of the form
\begin{align*}
    \iota_{\cR_\partial}\circ \pi_{\cR_\partial}=(\pi,\pi_1,\dots,\pi_b): \X_\Sigma^\uf(\bR^\trop) \to \X_{\bar{\Sigma}}(\bR^\trop) \times \prod_{i=1}^b \X_{\bD(\partial_i)}^\uf(\bR^\trop).
\end{align*}
The set of punctures on $\bar{\Sigma}$ is written as $\bar{P}$. We have $\bar{P}=P \sqcup \pi(\partial \Sigma)$, where $P$ is the set of punctures on $\Sigma$ and $\pi(\partial \Sigma)$ denotes the set of new punctures arising from the multicurve $\cR_\partial$ via the reduction. 
We identify $\X_{\bar{\Sigma}}(\bR^\trop) \cong \eML(\bar{F})$,
where $\bar{F}$ is the hyperbolic structure on $\bar{\Sigma}$ obtained as the restriction of $F_{\cR_\partial}$.

\begin{defi}\label{def:hyp_bdry}
A marked surface $\Sigma$ is \emph{of hyperbolic type} if the punctured surface $\bar{\Sigma}$ satisfies (S1), (S2) and (S3) in \cref{sec:surf_seed_pattern}. A mutation loop $\phi \in \Gamma_\Sigma$ is \emph{pseudo-Anosov} (again \emph{pA} for short) if the mapping class $\pi(\phi)$ on $\bar{\Sigma}$ is pseudo-Anosov in the usual sense.
\end{defi}
In what follows, we will consider a marked surface of hyperbolic type. 

First we are going to show that a pA mutation loop $\phi \in \Gamma_\Sigma$ has NS dynamics on $\bS \X_\Sigma^\uf(\bR^\trop)$ off a certain subset. To make the statement precise, let us prepare some notations. 
For two measured geodesic laminations $\hL_1,\hL_2 \in \X_{\Sigma}^\uf(\bR^\trop)$ with disjoint supports, let $\hL_1 \sqcup \hL_2 \in \X_{\Sigma}^\uf(\bR^\trop)$ be their disjoint union. This operation is clearly $\Gamma_\Sigma$-equivariant.

Let $D_{\cX, \Sigma}$ denote the $\bR_{>0}$-invariant subset in $\cX^\uf_\Sigma(\bR^\trop)$ consisting of positive real weighted curves which are either homotopic to some of boundary components or subarcs of them. 
The following lemma tells us that the laminations in $D_{\cX,\Sigma}$ are exactly those annihilated by $\pi$.


\begin{lem}\label{lem:barU_in_X}
\begin{enumerate}
    \item There exists an $\pos \times \Gamma_\Sigma$-equivariant section 
    \begin{align*}
        \iota:\cU_{\bar{\Sigma}}(\bR^\trop) \hookrightarrow \cU^\uf_\Sigma(\bR^\trop)
    \end{align*}
    of $\pi$. Here the equivariance means $\phi(\iota(\lambda L)) = \lambda\iota(\pi(\phi)(L))$ for any $L \in \cU_{\bar{\Sigma}}(\bR^\trop)$, $\lambda >0$ and $\phi \in \Gamma_\Sigma$. Similarly we have an $\pos \times\Gamma_\Sigma$-equivariant section
    \begin{align*}
        \iota_i: \cU_{\bD(\partial_i)}^\uf(\bR^\trop) \hookrightarrow \cU_\Sigma^\uf(\bR^\trop)
    \end{align*}
    of $\pi_i$ for $i=1,\dots,b$.
    \item We have $\pi^{-1}(L) = \{ \iota(L) \sqcup d \mid d \in D_{\cX, \Sigma}\}$ 
    for any $L \in \cU_{\bar{\Sigma}}(\bR^\trop)$. 
    \item The group $\Gamma_\Sigma$ preserves the subset $D_{\cX, \Sigma}$, where every orbits are finite.
\end{enumerate}
\end{lem}

\begin{proof}
(1): 
Recall from
\cref{thm:geom_models}
that the space  $\cU_{\Sigma}^\uf(\bR^\trop) \cong \ML_0(F)$ consists of compact laminations, and simlilarly for other marked surfaces. 
Since the laminations in $\cU_{\bar{\Sigma}}(\bR^\trop)$ do not have leaves incident to the punctures in $\bar{P}$, they can be naturally regarded as compact laminations on $\Sigma$. This gives the desired section $\iota$. The construction of $\iota_i$ is similar.


(2): 
Note that
\begin{align*}
    \pi^{-1}(L) = \left\{ \iota(L) \sqcup \bigsqcup_{i=1}^b \big( (C_i,w_i) \sqcup \iota_i(L_i) \big) ~\middle|~ w_k \geq 0,~L_i \in \cU_{\bD(\partial_i)}^\uf(\bR^\trop)\right\}
\end{align*}
from the construction. On the other hand, any measured geodesic lamination in $\X_{\bD(\partial_i)}^\uf(\bR^\trop)$ is given by a real weighted collection of arcs (\emph{i.e.}, there are no bi-infinite geodesic leaves) for $i=1,\dots,b$.\footnote{From the cluster algebraic point of view, this is a consequence of the fact that the corresponding seed pattern is of finite type $D_N$ for some $N$ and hence the cluster complex coincides with the ambient space (\cite[Corollary 5.5]{FG09}).} Hence the right-hand side coincides with the desired one.

(3): 
The $\Gamma_\Sigma$-invariance of the set $D_{\X,\Sigma}$ is clear from the description above. 
The orbits are finite since there are only finitely many homotopy classes of curves contained in $D_{\cX,\Sigma}$. 
\end{proof}


Now the following is a generalization of \cref{prop:pA_NS_SS} $(1) \Longleftrightarrow (3)$:

\begin{thm}\label{thm:pA_NS_general}
A mutation loop $\phi \in \Gamma_{\Sigma}$ is pA if and only if 
there exist two points $L^\pm_{\phi} \in \iota(\cU_{\bar{\Sigma}}(\bR^\trop)) \subset \cX^\uf_\Sigma(\bR^\trop)$ such that
\begin{align}\label{eq:weak_NS}
    \lim_{n \to \infty} \phi^{\pm n}([\hL]) = [L^\pm_{\phi}] \quad \mbox{in }\  \bS\X_\Sigma^\uf(\bR^\trop)
\end{align}
for any $\hL \in \cX^\uf_\Sigma(\bR^\trop) \setminus \pi^{-1}(\bR_{\geq 0}\cdot \pi(L^\mp_{\phi}))$.
\end{thm}

We call $L^\pm_\phi$ the attracting/repelling points of $\phi$, similarly to those of pA mapping classes on a punctured surface.

\begin{proof}
Recall that for a continuous map $f:X \to X$ on a compact Hausdorff space $X$, the \emph{($\omega$-)limit set} of a point $x \in X$ with respect to $f$ is defined to be 
\begin{align*}
    \Lambda^+_f(x):= \bigcap_{n \geq 0}\overline{\{f^k(x) \mid k \geq n\}},
\end{align*}
which is clearly $f$-stable (and $f$-invariant if $f$ is a homeomorphism). 

Let $\phi \in \Gamma_\Sigma$ be a pA mutation loop. 
Then $\pi(\phi) \in \Gamma_{\bar{\Sigma}}$ is pA by definition, so there exist attracting/repelling points $L^\pm_{\pi(\phi)} \in \cU_{\bar{\Sigma}}(\bR^\trop)$ of $\pi(\phi)$. 
Let $L^\pm_\phi := \iota(L^\pm_{\pi(\phi)})$ and consider a point $\hL \in \cX^\uf_\Sigma(\bR^\trop) \setminus \pi^{-1}(\bR_{\geq 0}\cdot L^\mp_{\pi(\phi)})$. It suffices to show that $\Lambda^+_\phi([\hL])=\{[L^+_{\phi}]\}$.
By \cref{prop:pA_NS_SS}, we have $\Lambda_{\pi(\phi)}([\pi(\hL)]) = \{ [L^+_{\pi(\phi)}] \}$.
Since the projection $\pi$ is continuous and $\phi$-equivariant by \cref{lem:pi_S_equivariance}, we have
\begin{align*}
    \pi(\Lambda^+_\phi([\hL])) &= \pi\bigg(\bigcap_{n \geq 0}\overline{\{\phi^k([\hL]) \mid k \geq n\}}\bigg) \\
    &\subset \bigcap_{n \geq 0}\pi\bigg(\overline{\{\phi^k([\hL]) \mid k \geq n\}}\bigg) \\
    &=\bigcap_{n \geq 0}\overline{\{\pi(\phi)^k(\pi(\hL)) \mid k \geq n\}} \\
    &= \Lambda^+_{\pi(\phi)}([\pi(\hL)])=\{[L_{\pi(\phi)}^+]\}.
\end{align*}
Combining with \cref{lem:barU_in_X} (3), we get $\Lambda^+_\phi([\hL]) \subset \pi^{-1}([L^+_{\pi(\phi)}])= \bS (\{d \sqcup L_{\phi}^+ \mid d \in D_{\X,\Sigma}\})$.
Moreover we have 
\begin{align*}
    \Lambda^+_\phi([\hL]) =\phi^n(\Lambda^+_\phi([\hL])) \subset \bS (\{\phi^n(d \sqcup L_{\phi}^+) \mid d \in D_{\X,\Sigma}\})
\end{align*}
for all $n \geq 0$, since $\Lambda^+_\phi([\hL])$ is $\phi$-invariant. \cref{lem:barU_in_X} (3) tells us that the orbit $\phi^n(d)$ is finite for all $d \in D_{\X,\Sigma}$, and the $\Gamma_\Sigma$-equivariance of $\iota$ implies that
\begin{align*}
    \phi^n(L^+_\phi)=\phi^n(\iota(L^+_{\pi(\phi)}))=\iota(\pi(\phi)^n(L^+_{\pi(\phi)}))=\iota(\lambda_{\pi(\phi)}^nL^+_{\pi(\phi)})=\lambda_{\pi(\phi)}^nL_\phi^+.
\end{align*}
Then by the $\Gamma_\Sigma$-equivariance of $\sqcup$ and that $\lambda_{\pi(\phi)} >1$,
\begin{align*}
    \phi^n([d \sqcup L^+_{\phi}]) = [\phi^n(d) \sqcup \phi^n(L^+_{\phi})] = [\lambda_{\pi(\phi)}^{-n}\phi^n(d) \sqcup L^+_{\phi}] \xrightarrow{n \to \infty} [L_{\phi}^+],
\end{align*}
and hence $\Lambda^+_\phi([\hL]) = \{[L^+_{\phi}]\}$. 
Applying the above argument for the pA mutation loop $\phi^{-1}$, we get the desired statement.

Conversely, suppose that $\phi$ satisfies \eqref{eq:weak_NS} on $\cX^\uf_\Sigma(\bR^\trop)$ for some $L^\pm_\phi \in \cU^\uf_\Sigma(\bR^\trop)$.
Then by applying the continuous $\pi$ on the both sides of \eqref{eq:weak_NS}, we get
\begin{align*}
    \lim_{n \to \infty} \pi(\phi)^{\pm n}([\hL]) = [\pi(L_\phi^\pm)] \quad \mbox{in $\bS \X_{\bar{\Sigma}}(\bR^\trop)$}
\end{align*}
for any $\hL \in \X_{\bar{\Sigma}}(\bR^\trop) \setminus \bR_{\geq 0}\cdot \pi(L_\phi^\mp)$.
It implies that $\pi(\phi)$ satisfies the condition (3) in \cref{prop:pA_NS_SS}, hence it is pA. 
\end{proof}

\subsection{Weak sign stability of pA mutation loops.}\label{subsec:weak_SS}
It turns out that the pA property is not equivalent to the uniform sign stability in the general case. 
A crucial difficulty arises from the fact that the equivalence of the arationality and the $\cX$-filling property given in \cite[Proposition 2.14]{IK20} (1) is no more valid in general. 
Therefore it happens that for a pA mutation loop $\phi = [\gamma]_{\bs_\Sigma} \in \Gamma_\Sigma$, the attracting point $L_{\phi}^+ \in \cX_\Sigma(\bR^\trop)$ in \cref{thm:pA_NS_general} may not be $\cX$-filling.
In particular, the sign $\boldsymbol{\epsilon}_\gamma(L_{\phi}^+)$ may not be strict. It means that the sign $\boldsymbol{\epsilon}_\gamma(\phi^n(\hL))$ for a point $[\hL] \in \bS\cX^\uf_\Sigma(\bR^\trop) \setminus \pi^{-1}([L^\mp_{\pi(\phi)}])$ might not stabilize to $\boldsymbol{\epsilon}_\gamma(L_\phi^+)$, in spite of the fact that $\phi^n([\hL])$ converges to $[L^+_{\phi}]$ as we have just seen in \cref{thm:pA_NS_general}. Indeed, the condition $\boldsymbol{\epsilon}_\gamma(\hL)=\boldsymbol{\epsilon}_\gamma(L_\phi^+)$ is now not an open condition for $\hL$. 
Nevertheless, one can observe that the subsequence of $\boldsymbol{\epsilon}_\gamma(\phi^n(\hL)) = (\epsilon^n_0, \dots, \epsilon^n_{h-1})$ corresponding to the non-zero entries in $\boldsymbol{\epsilon}_\gamma(L_\phi^+)$ stabilizes as $n \to \infty$. See \cref{ex:3bdries+1pct_sph} below. 

We are going to formulate this phenomenon, which we call the \emph{weak sign stability}. Let us prepare some notations.

\begin{defi}
\begin{enumerate}
    \item For two sign sequences $\boldsymbol{\epsilon}$, $\boldsymbol{\epsilon}' \in \{+,0,-\}^h$ of the same length $h >0$, we write $\boldsymbol{\epsilon} \leq \boldsymbol{\epsilon}'$ if $\boldsymbol{\epsilon}$ is obtained from $\boldsymbol{\epsilon}'$ by replacing some $+$ or $-$ components with $0$. For example, $(+, 0, 0, 0, +) \leq (+, -, +, 0, +)$. We write $\boldsymbol{\epsilon} < \boldsymbol{\epsilon}'$ if $\boldsymbol{\epsilon} \leq \boldsymbol{\epsilon}'$ and $\boldsymbol{\epsilon} \neq \boldsymbol{\epsilon}'$. 
    We say that $(\boldsymbol{\epsilon}_n)_{n \geq 0}$ \emph{weakly stabilizes to} $\boldsymbol{\epsilon}$ if $\boldsymbol{\epsilon}_n \geq \boldsymbol{\epsilon}$ for all but finitely many $n$.
    \item For a seed pattern $\bs$ and a path $\gamma: (t,\sigma) \to (t', \sigma')$ in $\bE_I$, let
    \[ \bS(\gamma) := \{ \boldsymbol{\epsilon} \mid \boldsymbol{\epsilon}_\gamma(x) = \boldsymbol{\epsilon} \mbox{ for some } x \in \cX^\uf_{(t, \sigma)}(\bR^\trop) \} \cap \{+, -\}^{h(\gamma)}, \]
    which is the set of \emph{realizable} sequence of \emph{strict} signs of length $h(\gamma)$.
\end{enumerate}
\end{defi}

\begin{ex}
Let us consider the seed pattern of type $A_2$ satisfying
\begin{align*}
    B^{(v_0)} = 
    \begin{pmatrix}
    0 & 1\\
    -1 & 0
    \end{pmatrix}
\end{align*}
and consider the following path:
    \[ \gamma: v_0 \overbar{1} v_1 \overbar{2} v_2 \overbar{1} 
    v_3. \]
    Then, 
    \[ \bS(\gamma) = \{(+,+,-), (+,-,-), (-,+,+), (-,-,+), (-,-,-)\}. \]

\begin{figure}[h]
    \centering
    \begin{tikzpicture}
    \draw (-3,0) -- (3,0);
    \draw (0,3) -- (0,-3);
    \draw (0,0) -- (-3,3);
    \node at (1.5,1.5) {$(+,+,-)$};
    \node at (-1,2.35) {$(-,+,+)$};
    \node at (-2.3,1) {$(-,-,+)$};
    \node at (-1.5,-1.5) {$(-,-,-)$};
    \node at (1.5,-1.5) {$(+,-,-)$};
    \node at (3.5,0) {$x_1$};
    \node at (0,3.35) {$x_2$};
    \end{tikzpicture}
    \caption{Signs of $\gamma$.\CORRECTED}
    \label{fig:fan_A_2}
\end{figure}
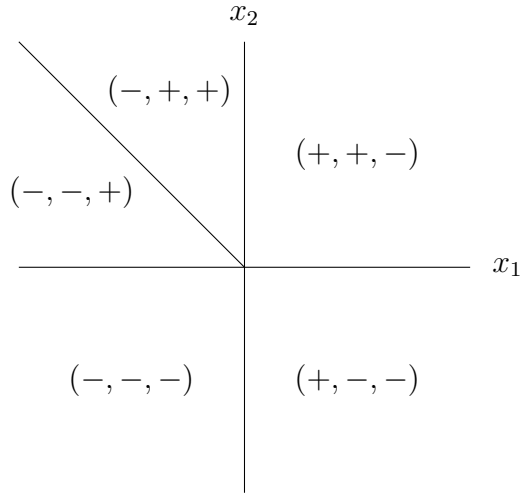
\end{ex}

\begin{defi}[weak sign stablility]\label{def:weak_SS}
Let $\gamma$ be a path which represents a mutation loop $\phi = [\gamma]_{\bs} \in \Gamma_\bs$.
We say that the path $\gamma: v_0 \xrightarrow{\mathbf{k}} v$ is \emph{weakly sign-stable} on an $\bR_{>0}$-invariant set $\Omega \subset \cX^\uf_{(v_0)}(\bR^\trop)$ if there is a sign sequence $\boldsymbol{\epsilon}^\stab_{\gamma, \Omega} \in \{+, 0, -\}^h \setminus \{(0, \dots, 0)\}$ such that
    for each $w \in \Omega$, there exists an integer $n_0 >0$ satisfying 
    \[ \boldsymbol{\epsilon}_\gamma(\phi^n(w)) \geq \boldsymbol{\epsilon}^\stab_{\gamma, \Omega} \]
    for all $n \geq n_0$.
    
We still call $\boldsymbol{\epsilon}^\stab_{\gamma, \Omega}$ the \emph{stable sign} of $\gamma$ on $\Omega$.

\end{defi}

The following additional condition is needed to characterize the pA property, which in some sense refer to a property of the mutation loop obtained after a reduction:

\begin{dfn}[Hereditariness]\label{d:hereditary}
A weakly sign-stable path $\gamma$ as above is said to be  \emph{$\cC$-hereditary} for some rational polyhedral cone $\cC \subset \cX^\uf_\bs(\bR^\trop)$ if the following condition for the stable sign $\boldsymbol{\epsilon}^\stab_{\gamma, \Omega} = (\epsilon_0, \epsilon_1, \dots, \epsilon_{h-1})$ holds:
\begin{itemize}
    \item The equation $x^{(v_{i(\nu)})}_{k_{i(\nu)}}(w) = 0$ for all $w \in \cC$ implies $\epsilon_\nu \neq 0$.
\end{itemize}
Here $(k_{i(0)}, k_{i(1)}, \dots, k_{i(h-1)})$ is the subsequence of $\mathbf{k}=(k_0, k_1, \dots, k_{m-1})$ which consists of horizontal indices.
\end{dfn}
For a multicurve $\cR$ on $\Sigma$, we denote the rational polyhedral cone generated by the curves in $\cR$ by $\cC(\cR)$.
The following result is a generalization of $(1) \Longleftrightarrow (2)$ in \cref{prop:pA_NS_SS}:

\begin{thm}\label{thm:weak_SS}
Let $\Sigma$ be a marked surface, 
and $\phi \in \Gamma_\Sigma$ a mutation loop.
Then, $\phi$ is pA if and only if any representation path $\gamma: (\tri, \ell) \to (\tri', \ell')$ of $\phi$ is weakly sign-stable on $\Omega_{(\tri, \ell)}^\bQ \setminus D_{\cX, \Sigma}$ and $\cC(\cR_\partial)$-hereditary.


\end{thm}

For a punctured surface $\Sigma_0$ and a subset $P' \subset P$, let $\bExch_{\Sigma_0}^{P'} \subset \bExch_{\Sigma_0}$ denote the subgraph consisting of labeled tagged triangulations $(\tri,\varsigma_\tri,\ell)$ such that $(\varsigma_\tri)_p =+$ for all $p \in P'$.

\begin{proof}
Let $\gamma: (\tri, \ell) \to (\tri', \ell')$ be any representation path of a pA mutation loop $\phi$ and we put $\boldsymbol{\epsilon}^\stab_\gamma := \boldsymbol{\epsilon}_\gamma(L^+_\phi)$. 
Since the star
\begin{align*}
    \bigcup_{\boldsymbol{\epsilon} \in \bS(\gamma),~ \boldsymbol{\epsilon}\geq \boldsymbol{\epsilon}_\gamma^\stab} \cC_\gamma^{\boldsymbol{\epsilon}} \subset \X_\Sigma^\uf(\bR^\trop)
\end{align*}
of $\cC_\gamma^{\boldsymbol{\epsilon}_\gamma^\stab}$
contains $L_\phi^+$ in its interior, \cref{thm:pA_NS_general} implies that the weak sign stability on $\Omega_{(\tri, \ell)}^\bQ \setminus D_{\cX, \Sigma}$ holds. Here note that the set $\Omega_{(\tri, \ell)}^\bQ \setminus D_{\cX, \Sigma}$ is contained in $\X_\Sigma^\uf(\bR^\trop) \setminus \pi^{-1}(\bR_{\geq 0}\cdot L_{\pi(\phi)}^-)$.
To prove the hereditariness, let us write $\gamma: (\tri,\ell) \xrightarrow{\mathbf{k}} (\tri',\ell')$, and consider the subsequence $(k_{i(0)},\dots,k_{i(h-1)})$ of $\mathbf{k}$ which consists of horizontal edges. 
From the construction, the index $k_{i(\nu)}$ appears in the path $\oline{\gamma}$ precisely if it satisfies the condition $x_{k_{i(\nu)}}^{(v_{i(\nu)}}(w)=0$ for all $w \in \cC(\cR_\partial)$. Since the path $\oline{\gamma}$ is sign-stable, the corresponding sign $\epsilon_\nu$ must be non-zero. Thus the path $\gamma$ is $\cC(\cR_\partial)$-hereditary.

Conversely, suppose that a representation path $\gamma: (\tri, \ell) \to (\tri', \ell')$ is weakly sign-stable on $\Omega^\bQ_{(\tri, \ell)}$ and $\cC(\cR_\partial)$-hereditary, and fix a labeling $\lambda$ of $\oline{\tri}$.
Then the $\cC(\cR_\partial)$-hereditariness implies that the path $\oline{\gamma}: (\oline{\tri}, \lambda) \to (\oline{\tri'}, \lambda')$ is sign-stable on $\pi(\Omega_{(\tri, \ell)}^\bQ \setminus D_{\cX, \Sigma}) \subset \pi(\Omega_{(\tri, \ell)}^\bQ) \setminus \pi(D_{\cX, \Sigma}) = \Omega^\bQ_{(\oline{\tri}, \lambda)} \setminus \{0\}$. The converse inclusion is also true. Indeed, since $\Omega^\bQ_{(\oline{\tri}, \lambda)} \setminus \{0\}$ consists of non-empty rational $\X$-laminations, it suffices to show that any curve on $\bar{\Sigma}$ can be lifted to a curve on $\Sigma$. Observe the followings: 
\begin{itemize}
    \item Each closed curve on $\bar{\Sigma}$ can be lifted by using the section $\iota$ defined in \cref{lem:barU_in_X}.
    \item For each ideal arc on $\bar{\Sigma}$, if both of its endpoints are punctures coming from $\Sigma$, then we can naturally regard it as an ideal arc on $\Sigma$.
    Otherwise, first lift it to a curve in $\Sigma_{\cR_\partial}$ with one or two endpoints on the multicurve $\cR_\partial$. Then extend it arbitrarily  across the multicurve $\cR_\partial$ and inside one of the components $\bD(\partial_i)$ until it hits a boundary interval. Then it gives a lift.
\end{itemize}
Thus for any representation path $\gamma$ as above, the reduced path $\oline{\gamma}$ is sign-stable on $\Omega^\bQ_{(\oline{\tri}, \lambda)} \setminus \{0\}$. 
By the following \cref{lem:graph_proj}, such paths $\oline{\gamma}$ exhaust the representation paths of $\pi(\phi)$ which are contained in the subgraph $\bExch_{\bar{\Sigma}}^{\pi(\partial \Sigma)} \subset \bExch_{\bar{\Sigma}}$.
Then $\pi(\phi)$ is pA by the implication (2) $\Longrightarrow$ (1) of \cref{prop:pA_NS_SS} and \cite[Remark 7.8]{IK19}. Hence $\phi$ is pA.
\end{proof}

Let $\Exch_{\Sigma_0}^{P'}$ denote the unlabeled version of $\bExch_{\Sigma_0}^{P'}$.

\begin{lem}\label{lem:graph_proj}
Let $\pi: \Exch_\Sigma \to \Exch_{\bar{\Sigma}}^{\pi(\partial \Sigma)}$ be the map of graphs obtained from the map $\pi_\cR:\Exch_\Sigma \to \Exch_{\Sigma_\cR}$ by forgetting the triangulation of the surfaces $\bD(\partial_i)$ for $i=1,\dots,b$. Then it is surjective. 
\end{lem}
\begin{figure}[h]
    \centering
    \begin{tikzpicture}[scale=1.2]
    \draw [blue] (4,2.5) coordinate (v1) -- (2.4,1.5) coordinate (v3) -- (3,-0.5) coordinate (v4) -- (5,-0.5) coordinate (v5) -- (5.6,1.5) coordinate (v7) -- (v1);
    \coordinate [fill, circle, inner sep=1.3] (v2) at (4,0.85) {};
    \draw [blue] (v2)-- (v1);
    \draw [blue] (v2) -- (v3);
    \draw [blue] (v2) -- (v4);
    \draw [red] (v2) -- (v5);
    \draw [blue] (v2) -- (v7);
    
    \draw  [blue] (-2.9,2.5) coordinate (v1) -- (-4.5,1.5) coordinate (v3) -- (-3.9,-0.5) coordinate (v4) -- (-1.9,-0.5) coordinate (v5) -- (-1.3,1.5) coordinate (v7) -- (v1);
    \coordinate [fill, circle, inner sep=1.3] (v2) at (-2.9,1.3) {} {} {} {};
    \draw [blue] (v2)-- (v1);
    \draw [blue] (v2) -- (v3);
    \draw [blue] (v2) -- (v7);
    \draw  (-2.9,1) ellipse (0.3 and 0.3);
    \node [fill, circle, inner sep=1.3] at (-3.2,1) {};
    \node [fill, circle, inner sep=1.3] at (-2.9,0.7) {};
    \node [fill, circle, inner sep=1.3] at (-2.6,1) {};
    \draw [blue] (-2.9,1.3) .. controls (-3.9,1.4) and (-3.9,0.5) .. (-3.9,-0.5);
    \draw [blue] (-2.9,1.3) .. controls (-1.9,1.4) and (-1.9,0.5) .. (-1.9,-0.5);
    \draw [blue] (-2.9,1.3) .. controls (-3.6,1.35) and (-3.4,0.5) .. (-2.9,0.7);
    \draw [blue] (-2.9,1.3) .. controls (-2.2,1.35) and (-2.4,0.5) .. (-2.9,0.7);
    \draw [blue] (-2.9,1.3) .. controls (-3.8,1.35) and (-3.55,0.35) .. (-2.9,0.35) .. controls (-2.2,0.35) and (-2,1.35) .. (-2.9,1.3);
    \draw [blue] (-2.9,1.3) .. controls (-4.25,1.35) and (-3.55,-0.3) .. (-1.9,-0.5);
    
    \draw [->, very thick](-0.5,1) -- (1.5,1);
    \node at (0.5,1.35) {$\pi$};
    \node at (3.7,0.7) {$p$};
    \node[red] at (4.6,0.4) {$\alpha'_p$};
    \end{tikzpicture}
    \caption{Right: an ideal triangulation of $\bar{\Sigma}$. Left: the image under the graph projection $\pi$.}
    \label{fig:graph_proj_inv}
\end{figure}
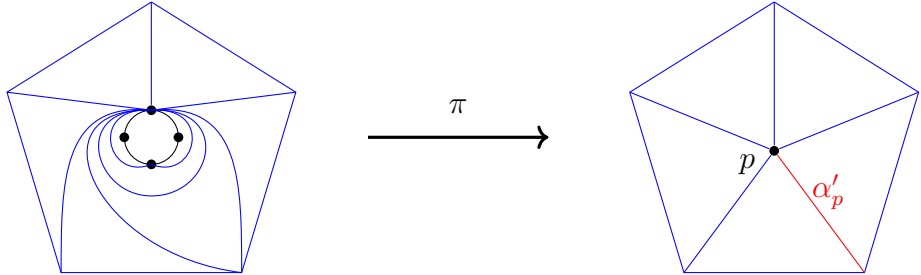
\begin{proof}
For any vertex $\tri' \in \Exch_{\bar{\Sigma}}^{\pi(\partial\Sigma)}$, by choosing an ideal arc $\alpha'_p$ incident to $p$ for each $p \in \pi(\partial \Sigma)$, we can find an ideal triangulation $\tri$ of $\Sigma$ such that
\begin{itemize}
    \item $\oline{\tri} = \tri'$,
    \item $\# \{\alpha \in \tri \mid \pi(\alpha)=\alpha' \} = 1$ for $\alpha' \in \tri' \setminus  \{\alpha'_p\}_{p \in \pi(\partial\Sigma)}$,
\end{itemize}
as shown in the left of \cref{fig:graph_proj_inv}.
Note that the second condition implies that
the unique ideal arc $\alpha$ is $\cR_\partial$-compatible for $\alpha' \in \tri' \setminus \{\alpha'_p\}_{p \in \pi(\partial\Sigma)}$. Given an edge $\tri'_1 \overbar{\alpha'} \tri'_2$ in $\Exch_{\bar{\Sigma}}^{\pi(\Sigma)}$, we can choose the arcs $\alpha'_p$ avoiding the given one $\alpha'$, since it is not the interior arc of a self-folded triangle and hence the valency at each endpoint is at least two. 
Then the unique edge $\tri_1 \overbarnear{\alpha} \tri_2$ such that $\pi(\alpha) = \alpha'$ projects to the edge $\tri'_1 \overbar{\alpha'} \tri'_2$.
\end{proof}

\begin{rem}
Essentially due to the constraint on the image of $\pi_{\cR_\partial}$ given in \cref{lem:im_pi_S}, the entire map $\pi_{\cR_\partial}: \mathrm{Exch}_\Sigma \to \mathrm{Exch}_{\Sigma_{\partial_\cR}}^{\bar{P}}$ is not surjective.
\end{rem}

\begin{ex}\label{ex:3bdries+1pct_sph}
Let $\Sigma$ be a sphere with one puncture and three boundary components $\partial_1$, $\partial_2$ and $\partial_3$, each of which has two special points. See \cref{fig:tri_1punc_3bdr}. 
Note that the reduced surface $\bar{\Sigma}$ is a sphere with 4 punctures.
Let us consider the following mapping classes on $\Sigma$:
\begin{itemize}
    \item $\tau_{ij}$: the braiding of the boundary components $\partial_i$ and $\partial_j$ in the clockwise direction, where $i \neq j \in \{1, 2, 3\}$.
    \item $\phi := \tau_{23} \circ \tau_{12}^{-1}$.
\end{itemize}
Then the mapping class $\pi(\phi) \in MC(\bar{\Sigma})$ is the one studied in \cite[Example 7.10]{IK20},
hence $\phi$ is a pA mutation loop.

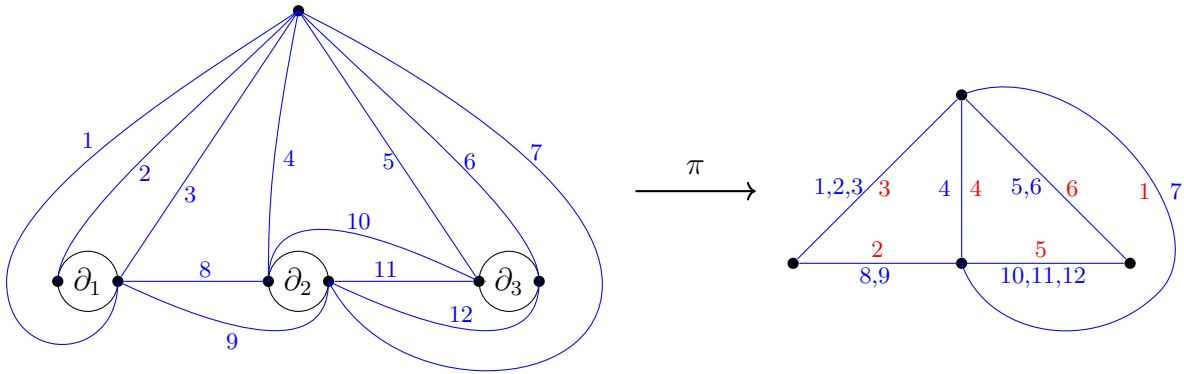
\begin{figure}[ht]
    \centering
    \begin{tikzpicture}[scale=0.8]
    \draw  (-3,0.5) ellipse (0.5 and 0.5);
    \draw  (0.5,0.5) ellipse (0.5 and 0.5);
    \draw  (4,0.5) ellipse (0.5 and 0.5);
    \node [fill, circle, inner sep=1.5pt] at (0.5,5) {};
    \node [fill, circle, inner sep=1.5pt] at (-3.5,0.5) {};
    \node [fill, circle, inner sep=1.5pt] (v5) at (-2.5,0.5) {};
    \node [fill, circle, inner sep=1.5pt] at (0,0.5) {};
    \node [fill, circle, inner sep=1.5pt] (v1) at (1,0.5) {};
    \node [fill, circle, inner sep=1.5pt] (v2) at (3.5,0.5) {};
    \node [fill, circle, inner sep=1.5pt] at (4.5,0.5) {};
    \draw [blue] (0.5,5) coordinate (v3) {} .. controls (0.05,2.7) and (0,1.8) .. (0,0.5) coordinate (v4) {};
    \draw [blue] (v1) -- (v2) -- (v2) -- (v3);
    \draw [blue] (v4) -- (v5) -- (v3);
    \draw [blue] (-3.5,0.5) .. controls (-3.5,1.5) and (-1,3.5) .. (0.5,5);
    \draw [blue] (4.5,0.5) .. controls (4.5,1.5) and (2,3.5) .. (0.5,5);
    \draw [blue] (0,0.5) .. controls (0,2.1) and (2.25,1.1) .. (3.5,0.5);
    \draw [blue] (4.5,0.5) .. controls (4.5,-1) and (2.25,-0.1) .. (1,0.5);
    \draw [blue] (1,0.5) .. controls (1,-1) and (-1.25,-0.1) .. (-2.5,0.5);
    \draw [blue] (-2.5,0.5);
    \draw [blue] (-2.5,0.5) .. controls (-2.5,-0.9) and (-4.35,-0.9) .. (-4.35,0.5) .. controls (-4.35,2) and (-1,4) .. (0.5,5);
    \draw [blue] (1,0.5) .. controls (2.05,-1.55) and (5.55,-1.4) .. (5.55,0.45) .. controls (5.55,2.2) and (2.3,4.1) .. (0.5,5);
    \node at (-3,0.5) {$\partial_1$};
    \node at (0.5,0.5) {$\partial_2$};
    \node at (4,0.5) {$\partial_3$};
    \node [blue] at (-3,2.85) {\scriptsize 1};
    \node [blue] at (-2.05,2.3) {\scriptsize 2};
    \node [blue] at (-1.3,1.95) {\scriptsize 3};
    \node [blue] at (0.35,2.55) {\scriptsize 4};
    \node [blue] at (2,2.5) {\scriptsize 5};
    \node [blue] at (3.35,2.5) {\scriptsize 6};
    \node [blue] at (4.45,2.65) {\scriptsize 7};
    \node [blue] at (-1.05,0.7) {\scriptsize 8};
    \node [blue] at (-0.6,-0.5) {\scriptsize 9};
    \node [blue] at (1.5,1.5) {\scriptsize 10};
    \node [blue] at (1.95,0.7) {\scriptsize 11};
    \node [blue] at (3.2,-0.05) {\scriptsize 12};
    
    \draw [thick, ->](6.1,2) -- (8.1,2);
    \node at (7.1,2.4) {$\pi$};
    
    \begin{scope}[scale=0.8, xshift=1.4cm, yshift=.5cm]
    \node [fill, circle, inner sep=1.5pt] (v6) at (13,4) {};
    \node [fill, circle, inner sep=1.5pt] (v9) at (13,0.5) {};
    \node [fill, circle, inner sep=1.5pt] (v7) at (9.5,0.5) {};
    \node [fill, circle, inner sep=1.5pt] (v8) at (16.5,0.5) {};
    \draw [blue](v6) -- (v7) -- (v8) -- (v6) -- (v9);
    \draw [blue](13,0.5) .. controls (13.5,-1) and (15.75,-1.45) .. (17.1,-0.1) .. controls (18.45,1.25) and (15.5,5) .. (13,4);
    \node [blue] at (10.45,2.05) {\scriptsize 1,2,3};
    \node [blue] at (12.65,2.05) {\scriptsize 4};
    \node [blue] at (14.35,2.05) {\scriptsize 5,6};
    \node [blue] at (17.45,2) {\scriptsize 7};
    \node [blue] at (11.2,0.2) {\scriptsize 8,9};
    \node [blue] at (14.7,0.2) {\scriptsize 10,11,12};
	\node [red] at (16.8,2) {\scriptsize 1};
	\node [red] at (11.25,0.8) {\scriptsize 2};
	\node [red] at (11.4,2.05) {\scriptsize 3};
	\node [red] at (13.3,2.05) {\scriptsize 4};
	\node [red] at (14.65,0.8) {\scriptsize 5};
	\node [red] at (15.3,2.05) {\scriptsize 6};
    \end{scope}
    \end{tikzpicture}
    \vspace{-5mm}
    \caption{Left: labeled triangulation $(\tri, \ell)$ of $\Sigma$, right: the triangulation $\oline{\tri}$ of $\bar{\Sigma}$ and red numbers denote a labeling $\lambda$ of $\oline{\tri}$}
    \label{fig:tri_1punc_3bdr}
\end{figure}

First, we give representation paths of these mapping classes.
Take a triangulation $\tri$ of $\Sigma$ and its labeling $\ell$ as in \cref{fig:tri_1punc_3bdr}.
The triangulations $\tau_{12}(\tri, \ell)$ and $\tau_{23}^{-1}(\tri, \ell)$ are shown in \cref{fig:tris_twisted}.
Then one can compute the representation paths of $\tau^{-1}_{12}$, as follows:
\begin{align*}
\gamma_1&: (\tri, \ell) \xrightarrow{(\textcolor{red}{7}, 8, 10, 3, 2, 9, 12, \textcolor{red}{11}, 9)} (\tri', \ell') \xrightarrow{(1\ 7\ 12\ 10\ 3\ 8) \circ (2\ 9\ 11\ 4)} \tau_{AB}(\tri, \ell)\\
\gamma_2&: (\tri, \ell) \xrightarrow{(\textcolor{red}{7}, 12, 10, 5, 9, \textcolor{red}{8}, 11)} (\tri'', \ell'') \xrightarrow{(4\ 5\ 12\ 8) \circ (6\ 7\ 8\ 10)} \tau_{BC}^{-1}(\tri, \ell)
\end{align*}
Here $(\tri_i, \ell_i) \xrightarrow{\sigma} (\tri_j, \ell_j)$ for $\sigma \in \fS_{12}$ denotes the sequence of vertical edges corresponds any reduced expression of $\sigma$, and the $\cR_\partial$-compatible edges are colored red.
Thus the concatenation of paths $\gamma' := \gamma_2 * \gamma_1$ represents the mutation loop $\phi$:
\[
\gamma: (\tri, \ell) \xrightarrow{(\textcolor{red}{7}, 12, 10, 5, 9, \textcolor{red}{8}, 11, \textcolor{red}{6}, 12, 9, 3, 2, 7, 5, \textcolor{red}{11}, 7)} (\tri''', \ell''') \xrightarrow{(1\ 7\ 11\ 4\ 5\ 10\ 6\ 12) \circ (2\ 9\ 3\ 8)} \phi^{-1}(\tri, \ell).
\]

The triangulation $\tri$ descends to a triangulation $\oline{\tri}$ of the sphere $\bar{\Sigma}$ with 4 punctures, which coincides with the one considered in \cite[Example 7.10]{IK20}.
Let $\lambda$ be a labeling of $\oline{\tri}$ as in \cite[Figure 28]{IK20}, which is also shown in the right of \cref{fig:tri_1punc_3bdr} in red. 
One can verify that the reduced path $\oline{\gamma}: (\oline{\tri}, \lambda) \to (\oline{\phi^{-1}(\tri)}, \lambda')$ obtained from the path $\gamma$ above via the procedure given in \cref{subsec:reduction_path} coincides with the path in \cite[Example 7.10]{IK20}.

Using the attracting/repelling points $L^\pm_{\pi(\phi)}$ of $\pi(\phi)$,
we can calculate\footnote{
A reader familiar with train tracks is suggested to draw the train track which carries $L^\pm_\phi$ and calculate its shear coordinates.
}
the coordinates of $L^\pm_\phi=\iota(L_{\pi(\phi)}^\pm)$ as
\begin{align*}
\boldsymbol{x}_{(\tri, \ell)}(L^+_\phi) 
&= \left(
1, 0, 0, \frac{1-\sqrt{5}}{2}, 0, -1, \frac{-1+\sqrt{5}}{2}, -1, 0, \frac{1+\sqrt{5}}{2}, \frac{-1-\sqrt{5}}{2}, 1\right),\\
\boldsymbol{x}_{(\tri, \ell)}(L^-_\phi) 
&= \left(
0, 0, -1, \frac{-1+\sqrt{5}}{2}, 1, 0, \frac{1-\sqrt{5}}{2}, \frac{1-\sqrt{5}}{2}, \frac{1+\sqrt{5}}{2}, 0, -1, 0\right).
\end{align*}
Then the stable sign should be
\begin{align*}
    \boldsymbol{\epsilon}^\stab_{\gamma} := \boldsymbol{\epsilon}_{\gamma}(L^+_\phi) = (
    \textcolor{red}{+}, +, +, 0, 0, \textcolor{red}{-}, +, \textcolor{red}{-}, -, +, 0, 0, -, +, \textcolor{red}{+}, +). 
\end{align*}
The signs corresponding to $\cR_\partial$-compatible edges (colored red in the equation above) are strict and coincides with the stable sign of $\oline{\gamma}$.

Let us observe the weak stabilization behavior of signs to $\boldsymbol{\epsilon}^\stab_\gamma$. 
Consider the points $l^\pm_\tri$ with coordinates $\boldsymbol{x}_{(\tri, \ell)}(l^\pm_\tri) = (\pm1, \dots, \pm1)$, which are not contained in $\bR_{\geq 0}\pi^{-1}(L_{\pi(\phi)}^-)$. Then by \cref{thm:pA_NS_general}, $\phi^n([\ell^\pm_\tri])$ converges to $[L_\phi^\pm]$. 
The signs of $\phi^n(\ell^\pm_\tri)$ for small $n$ are shown in \cref{tab:signs_l_pm}.
By looking at the subsequences colored by green, one can observe the signs $\boldsymbol{\epsilon}_\gamma(\phi^n(l^\pm_\tri)))$ weakly stabilizes to $\boldsymbol{\epsilon}^\stab_\gamma$.

\begin{table}[h]
    \centering
    \caption{Orbit of signs.}
    \vspace{-2mm}
    \begin{tabular}{c||c|c}
    \hline
    $n$ & $\boldsymbol{\epsilon}_\gamma(\phi^n(l^+_\tri))$ & $\boldsymbol{\epsilon}_\gamma(\phi^n(l^-_\tri))$  \\
    \hline
    1 & $\scriptstyle (\redtext{+,+,+},+,+,\redtext{+,+,+,-,+},+,+,\redtext{-,+,+,+})$ 
    & $\scriptstyle (\redtext{-,-,-},-,-,\redtext{-,-,-,-,+},-,-,\redtext{-,+,+,+})$ \\
    2 & $\scriptstyle (\redtext{+,+,+},+,-,\redtext{-,+,+,-,+},+,+,\redtext{-,+,+,+})$
    & $\scriptstyle (\redtext{-,+,+},-,-,\redtext{-,+,-,-,+},+,+,\redtext{-,+,+,+})$\\
    3 & $\scriptstyle (\redtext{+,+,+},-,-,\redtext{-,+,-,-,+},-,+,\redtext{-,+,+,+})$
    & $\scriptstyle (\redtext{+,+,+},-,-,\redtext{-,+,-,-,+},+,+,\redtext{-,+,+,+})$\\
    4 & $\scriptstyle (\redtext{+,+,+},-,-,\redtext{-,+,-,-,+},+,+,\redtext{-,+,+,+})$
    & $\scriptstyle (\redtext{+,+,+},-,-,\redtext{-,+,-,-,+},+,+,\redtext{-,+,+,+})$\\
    5 & $\scriptstyle (\redtext{+,+,+},+,-,\redtext{-,+,-,-,+},+,+,\redtext{-,+,+,+})$
    & $\scriptstyle (\redtext{+,+,+},-,-,\redtext{-,+,-,-,+},+,+,\redtext{-,+,+,+})$\\
    6 & $\scriptstyle (\redtext{+,+,+},-,-,\redtext{-,+,-,-,+},-,+,\redtext{-,+,+,+})$
    & $\scriptstyle (\redtext{+,+,+},-,-,\redtext{-,+,-,-,+},+,+,\redtext{-,+,+,+})$\\
    7 & $\scriptstyle (\redtext{+,+,+},-,-,\redtext{-,+,-,-,+},+,+,\redtext{-,+,+,+})$
    & $\scriptstyle (\redtext{+,+,+},-,-,\redtext{-,+,-,-,+},+,+,\redtext{-,+,+,+})$\\
    8 & $\scriptstyle (\redtext{+,+,+},+,-,\redtext{-,+,-,-,+},+,+,\redtext{-,+,+,+})$
    & $\scriptstyle (\redtext{+,+,+},-,-,\redtext{-,+,-,-,+},+,+,\redtext{-,+,+,+})$\\
    9 & $\scriptstyle (\redtext{+,+,+},-,-,\redtext{-,+,-,-,+},-,+,\redtext{-,+,+,+})$
    & $\scriptstyle (\redtext{+,+,+},-,-,\redtext{-,+,-,-,+},+,+,\redtext{-,+,+,+})$\\
    10 & $\scriptstyle (\redtext{+,+,+},-,-,\redtext{-,+,-,-,+},+,+,\redtext{-,+,+,+})$
    & $\scriptstyle (\redtext{+,+,+},-,-,\redtext{-,+,-,-,+},+,+,\redtext{-,+,+,+})$\\
    11 & $\scriptstyle (\redtext{+,+,+},+,-,\redtext{-,+,-,-,+},+,+,\redtext{-,+,+,+})$
    & $\scriptstyle (\redtext{+,+,+},-,-,\redtext{-,+,-,-,+},+,+,\redtext{-,+,+,+})$\\
   \vdots & \vdots & \vdots
    \end{tabular}
    \label{tab:signs_l_pm}
\end{table}

Finally we verify the $\cC(\cR_\partial)$-hereditariness of $\gamma$.
In \cref{fig:tris_corresp_zero}, the triangulations $(\tri_1, \ell_1)$ and $(\tri_3, \ell_3)$ which appear in the following subpath of $\gamma$ is shown:
\begin{align*}
    (\tri, \ell) \xrightarrow{(7,12,10)} (\tri_1, \ell_1) \overbar{5} (\tri_2, \ell_2) \xrightarrow{(8,11,6,12,9)} (\tri_3, \ell_3) \overbar{3} (\tri_4, \ell_4) \overbar{2} \cdots.
\end{align*}
Then we have
\begin{align*}
    x^{(\tri_1,\ell_1)}_5(C_2) = 1,\ x^{(\tri_2,\ell_2)}_9(C_3) = 1,\ x^{(\tri_3,\ell_3)}_3(C_1) = -1,\ x^{(\tri_4,\ell_4)}_2(C_1) = -1,
\end{align*}
and thus $\gamma$ satisfies the $\cC(\cR_\partial)$-hereditary condition.

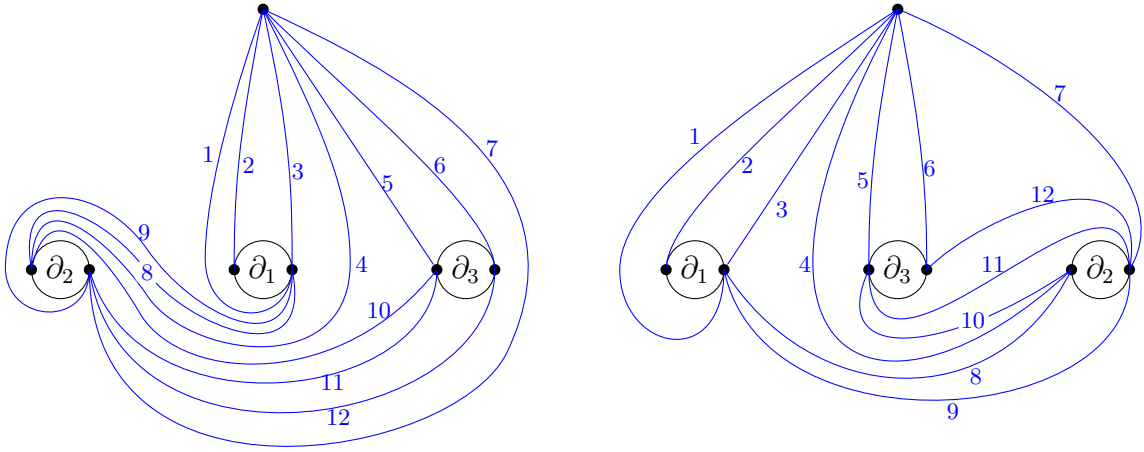
\begin{figure}[h]
    \[
    \begin{tikzpicture}[scale=0.77, baseline=(v1)]
    \draw  (-3,0.5) ellipse (0.5 and 0.5);
    \draw  (0.5,0.5) ellipse (0.5 and 0.5);
    \draw  (4,0.5) ellipse (0.5 and 0.5);
    \node [fill, circle, inner sep=1.5pt] at (0.5,5) {};
    \node [fill, circle, inner sep=1.5pt] at (-3.5,0.5) {};
    \node [fill, circle, inner sep=1.5pt] (v5) at (-2.5,0.5) {};
    \node [fill, circle, inner sep=1.5pt] at (0,0.5) {};
    \node [fill, circle, inner sep=1.5pt] (v1) at (1,0.5) {};
    \node [fill, circle, inner sep=1.5pt] (v2) at (3.5,0.5) {};
    \node [fill, circle, inner sep=1.5pt] at (4.5,0.5) {};
    \draw [blue] (0.5,5) coordinate (v3) {} .. controls (0.05,2.7) and (0,1.8) .. (0,0.5) coordinate (v4) {};
    \draw [blue] (v2) -- (v3);
    \draw [blue] (4.5,0.5) .. controls (4.5,1.5) and (2,3.5) .. (0.5,5);
    \draw [blue] (4.5,0.5) .. controls (4.5,-2.25) and (-1.95,-3.3) .. (-2.5,0.5);
    \draw [blue] (-2.5,0.5) .. controls (-2.5,-3.85) and (3.75,-2.8) .. (4.7,-0.95) .. controls (6.25,2.25) and (2.65,4) .. (0.5,5);
    \node at (-3,0.5) {$\partial_2$};
    \node at (0.5,0.5) {$\partial_1$};
    \node at (4,0.5) {$\partial_3$};
    \draw [blue] (3.5,0.5) .. controls (3.5,-1.75) and (-1.65,-2.45) .. (-2.5,0.5);
    \draw [blue] (1,0.5) .. controls (1,1.8) and (1,2.7) .. (0.5,5);
    \draw [blue] (1,0.5);
    \draw [blue] (1,0.5) .. controls (0.95,-0.55) and (-0.5,-0.5) .. (-0.5,0.7) .. controls (-0.5,2.3) and (0.5,5) .. (0.5,5);
    \draw [blue] (1,0.5) .. controls (1.05,-1.2) and (-0.75,-0.25) .. (-1.5,0.85) .. controls (-2.3,2) and (-3.95,2.2) .. (-3.95,0.5);
    \draw [blue] (1,0.5) .. controls (1.4,-1.45) and (-0.9,-0.4) .. (-1.5,0.45) .. controls (-2.1,1.35) and (-3.8,2.3) .. (-3.5,0.5);
    \draw [blue] (0.5,5);
    \draw [blue] (0.5,5) .. controls (1.5,3) and (2.05,1.75) .. (2,0.5) .. controls (1.9,-1.4) and (-0.8,-1) .. (-1.5,0.05) .. controls (-2.1,0.95) and (-3.5,2.2) .. (-3.5,0.5);
    \draw [blue] (-3.95,0.5) .. controls (-3.9,-0.45) and (-2.5,-0.5) .. (-2.5,0.5);
    \draw [blue] (-3.5,0.5) .. controls (-3.2,2) and (-2.15,0.65) .. (-1.6,-0.3) .. controls (-0.85,-1.45) and (1.85,-1.6) .. (3.5,0.5);
    \node at (-0.45,2.5) {\textcolor{blue}{\scriptsize 1}};
    \node at (0.25,2.3) {\textcolor{blue}{\scriptsize 2}};
    \node at (1.1,2.2) {\textcolor{blue}{\scriptsize 3}};
    \node at (2.2,0.6) {\textcolor{blue}{\scriptsize 4}};
    \node at (2.65,2) {\textcolor{blue}{\scriptsize 5}};
    \node at (3.55,2.3) {\textcolor{blue}{\scriptsize 6}};
    \node at (4.45,2.6) {\textcolor{blue}{\scriptsize 7}};
    \node [fill, white, circle, inner sep=1pt] at (-1.5,0.45) {\textcolor{blue}{\scriptsize 8}};
    \node at (-1.55,1.15) {\textcolor{blue}{\scriptsize 9}};
    \node at (2.5,-0.2) {\textcolor{blue}{\scriptsize 10}};
    \node at (1.7,-1.45) {\textcolor{blue}{\scriptsize 11}};
    \node at (1.8,-2.05) {\textcolor{blue}{\scriptsize 12}};
    \end{tikzpicture}
    \hspace{3mm}
    \begin{tikzpicture}[scale=0.77, baseline=(v1)]
    \draw  (-3,0.5) ellipse (0.5 and 0.5);
    \draw  (0.5,0.5) ellipse (0.5 and 0.5);
    \draw  (4,0.5) ellipse (0.5 and 0.5);
    \node [fill, circle, inner sep=1.5pt] at (0.5,5) {};
    \node [fill, circle, inner sep=1.5pt] at (-3.5,0.5) {};
    \node [fill, circle, inner sep=1.5pt] (v5) at (-2.5,0.5) {};
    \node [fill, circle, inner sep=1.5pt] at (0,0.5) {};
    \node [fill, circle, inner sep=1.5pt] (v1) at (1,0.5) {};
    \node [fill, circle, inner sep=1.5pt] (v2) at (3.5,0.5) {};
    \node [fill, circle, inner sep=1.5pt] at (4.5,0.5) {};
    \draw [blue] (0.5,5) coordinate (v3) {} .. controls (0.05,2.7) and (0,1.8) .. (0,0.5) coordinate (v4);
    \draw [blue] (-3.5,0.5) .. controls (-3.5,1.5) and (-1,3.5) .. (0.5,5);
    \draw [blue] (4.5,0.5) .. controls (4.75,-2.2) and (-1.8,-2.8) .. (-2.5,0.5);
    \draw [blue] (-2.5,0.5);
    \draw [blue] (-2.5,0.5) .. controls (-2.5,-1.3) and (-4.3,-0.9) .. (-4.3,0.5) .. controls (-4.3,2) and (-1,4) .. (0.5,5);
    \draw [blue] (4.5,0.5) .. controls (5.55,2.2) and (2.25,4.3) .. (0.5,5);
    \draw [blue] (v5) -- (v3);
    \node at (-3,0.5) {$\partial_1$};
    \node at (0.5,0.5) {$\partial_3$};
    \node at (4,0.5) {$\partial_2$};
    \draw [blue] (1,0.5) .. controls (1,1.8) and (0.95,2.7) .. (0.5,5);
    \draw [blue] (0.5,5) .. controls (-2.9,-1.25) and (0.25,-2.55) .. (3.5,0.5);
    \draw [blue] (3.5,0.5) .. controls (2.2,-2.1) and (-1,-1.9) .. (-2.5,0.5);
    \draw [blue] (0,0.5);
    \draw [blue] (0,0.5) .. controls (-0.75,-1.15) and (1.3,-1) .. (3.5,0.5);
    \draw [blue] (0,0.5) .. controls (0,-2.35) and (4.5,3.1) .. (4.5,0.5);
    \draw [blue] (1,0.5) .. controls (2,1.5) and (4.9,2.65) .. (4.5,0.5);
    \node [blue] at (-3,2.8) {\scriptsize 1};
    \node [blue] at (-2.1,2.3) {\scriptsize 2};
    \node [blue] at (-1.5,1.55) {\scriptsize 3};
    \node [blue] at (-1.1,0.6) {\scriptsize 4};
    \node [blue] at (-0.1,2.05) {\scriptsize 5};
    \node [blue] at (1.05,2.25) {\scriptsize 6};
    \node [blue] at (3.3,3.55) {\scriptsize 7};
    \node [blue] at (1.85,-1.35) {\scriptsize 8};
    \node [blue] at (1.45,-1.95) {\scriptsize 9};
    \node [fill, white, circle, inner sep=.5pt] at (1.8,-0.38) {\textcolor{blue}{\scriptsize 10}};
    \node [blue] at (2.15,0.6) {\scriptsize 11};
    \node [blue] at (3,1.8) {\scriptsize 12};
    \end{tikzpicture}
    \]
    \vspace{-1.2cm}
    \caption{Left: $\tau_{12}(\tri, \ell)$, right: $\tau_{23}^{-1}(\tri, \ell)$}
    \label{fig:tris_twisted}
\end{figure}

\begin{figure}[h]
    \centering
    \begin{tikzpicture}[scale=.75, baseline=(b)]
    \draw  (-3,0.5) ellipse (0.5 and 0.5);
    \draw  (0.5,0.5) ellipse (0.5 and 0.5);
    \draw  (4,0.5) ellipse (0.5 and 0.5);
    \node [fill, circle, inner sep=1.5pt] (b) at (0.5,5) {};
    \node [fill, circle, inner sep=1.5pt] at (-3.5,0.5) {};
    \node [fill, circle, inner sep=1.5pt] (v5) at (-2.5,0.5) {};
    \node [fill, circle, inner sep=1.5pt] at (0,0.5) {};
    \node [fill, circle, inner sep=1.5pt] (v1) at (1,0.5) {};
    \node [fill, circle, inner sep=1.5pt] (v2) at (3.5,0.5) {};
    \node [fill, circle, inner sep=1.5pt] at (4.5,0.5) {};
    \draw [blue] (0.5,5) coordinate (v3) {} .. controls (0.05,2.7) and (0,1.8) .. (0,0.5) coordinate (v4) {};
    \draw [blue] (v1) -- (v2) -- (v2) -- (v3);
    \draw [blue] (v4) -- (v5) -- (v3);
    \draw [blue] (-3.5,0.5) .. controls (-3.5,1.5) and (-1,3.5) .. (0.5,5);
    \draw [blue] (4.5,0.5) .. controls (4.5,1.5) and (2,3.5) .. (0.5,5);
    \draw [blue] (1,0.5) .. controls (1,-1) and (-1.25,-0.1) .. (-2.5,0.5);
    \draw [blue] (-2.5,0.5);
    \draw [blue] (-2.5,0.5) .. controls (-2.5,-0.9) and (-4.35,-0.9) .. (-4.35,0.5) .. controls (-4.35,2) and (-1,4) .. (0.5,5);
    \node at (-3,0.5) {$\partial_1$};
    \node (v6) at (0.5,0.5) {$\partial_2$};
    \node (v7) at (4,0.5) {$\partial_3$};
    \node [blue] at (-3,2.85) {\scriptsize 1};
    \node [blue] at (-2.05,2.3) {\scriptsize 2};
    \node [blue] at (-1.3,1.95) {\scriptsize 3};
    \node [blue] at (-0.2,2) {\scriptsize 4};
    \node [blue] at (2,2.5) {\scriptsize 5};
    \node [blue] at (3.35,2.5) {\scriptsize 6};
    \node [blue] at (3.25,-1.35) {\scriptsize 7};
    \node [blue] at (-1.05,0.7) {\scriptsize 8};
    \node [blue] at (-0.45,-0.05) {\scriptsize 9};
    \node [blue] at (1.3,2) {\scriptsize 10};
    \node [blue] at (1.95,0.7) {\scriptsize 11};
    \node [blue] at (2.45,-0.55) {\scriptsize 12};
    \draw [blue](1,0.5) .. controls (1,1.8) and (1,2.7) .. (0.5,5);
    \draw [blue](-2.5,0.5) .. controls (-1,-1.25) and (2,-1) .. (3.5,0.5);
    \draw [blue](-2.5,0.5) .. controls (-1,-2.25) and (4.5,-2) .. (4.5,0.5);
    \draw [red] (v6) ellipse (1 and 1);
    \draw [red] (v7) ellipse (1 and 1);
    \node [red] at (0.5,-1) {$C_2$};
    \node [red] at (4,-1) {$C_3$};
    \end{tikzpicture}
    \hspace{1cm}
    \begin{tikzpicture}[scale=.75, baseline=(b)]
    \draw  (-3,0.5) ellipse (0.5 and 0.5);
    \draw  (0.5,0.5) ellipse (0.5 and 0.5);
    \draw  (4,0.5) ellipse (0.5 and 0.5);
    \node [fill, circle, inner sep=1.5pt] (b) at (0.5,5) {};
    \node [fill, circle, inner sep=1.5pt] at (-3.5,0.5) {};
    \node [fill, circle, inner sep=1.5pt] (v5) at (-2.5,0.5) {};
    \node [fill, circle, inner sep=1.5pt] at (0,0.5) {};
    \node [fill, circle, inner sep=1.5pt] (v1) at (1,0.5) {};
    \node [fill, circle, inner sep=1.5pt] (v2) at (3.5,0.5) {};
    \node [fill, circle, inner sep=1.5pt] at (4.5,0.5) {};
    \draw [blue] (0.5,5) coordinate (v3) {} .. controls (0.05,2.7) and (0,1.8) .. (0,0.5) coordinate (v4);
    \draw [blue] (-3.5,0.5) .. controls (-3.5,1.5) and (-1,3.5) .. (0.5,5);
    \draw [blue] (4.5,0.5) .. controls (4.75,-2.2) and (-1.8,-2.8) .. (-2.5,0.5);
    \draw [blue] (-2.5,0.5);
    \draw [blue] (-2.5,0.5) .. controls (-2.5,-1.3) and (-4.3,-0.9) .. (-4.3,0.5) .. controls (-4.3,2) and (-1,4) .. (0.5,5);
    \draw [blue] (v5) -- (v3);
    \node (v6) at (-3,0.5) {$\partial_1$};
    \node at (0.5,0.5) {$\partial_2$};
    \node at (4,0.5) {$\partial_3$};
    \draw [blue] (1,0.5) .. controls (1,1.8) and (0.95,2.7) .. (0.5,5);
    \draw [blue] (0.5,5) .. controls (-2.85,-1.25) and (0.25,-2.55) .. (3.5,0.5);
    \draw [blue] (0,0.5);
    \draw [blue] (0,0.5) .. controls (0,-2.35) and (4.5,3.45) .. (4.5,0.5);
    \draw [blue] (1,0.5) .. controls (2,1.5) and (4.9,2.65) .. (4.5,0.5);
    \node [blue] at (-3,2.8) {\scriptsize 1};
    \node [blue] at (-2.1,2.3) {\scriptsize 2};
    \node [blue] at (-2,1.5) {\scriptsize 3};
    \node [blue] at (2.6,-0.5) {\scriptsize 8};
    \node [blue] at (0.25,2) {\scriptsize 4};
    \node [blue] at (1.15,2.25) {\scriptsize 10};
    \node [blue] at (-0.55,1) {\scriptsize 9};
    \node [blue] at (-1.6,0) {\scriptsize 12};
    \node [blue] at (3.05,-1.55) {\scriptsize 7};
    \node [fill, white, circle, inner sep=.5pt] at (4.2,-1.7) {\textcolor{blue}{\scriptsize 6}};
    \node [blue] at (2.05,0.65) {\scriptsize 11};
    \node [blue] at (3,1.8) {\scriptsize 5};
    \draw [blue] (1,0.5) .. controls (2.7,2.75) and (5.1,2.95) .. (5.1,0) .. controls (5.1,-2.5) and (-2,-3.5) .. (-2.5,0.5);
    \draw [blue](4.5,0.5) .. controls (4.4,-1) and (2,-1.5) .. (0.5,-1.5) .. controls (-1.5,-1.5) and (-2.45,0.25) .. (0.5,5);
    \draw [red] (v6) ellipse (1 and 1);
    \node [red] at (-3.5,-1) {$C_1$};
    \draw [blue](0.5,5) .. controls (-1.7,-1) and (0.4,-1.5) .. (1.95,-0.15) .. controls (3.5,1.15) and (4.4,1.6) .. (4.5,0.5);
    \end{tikzpicture}
    \vspace{-1cm}
    \caption{Left: $(\tri_1, \ell_1)$, $C_2$ and $C_3$, right: $(\tri_3, \ell_3)$ and $C_1$.}
    \label{fig:tris_corresp_zero}
\end{figure}
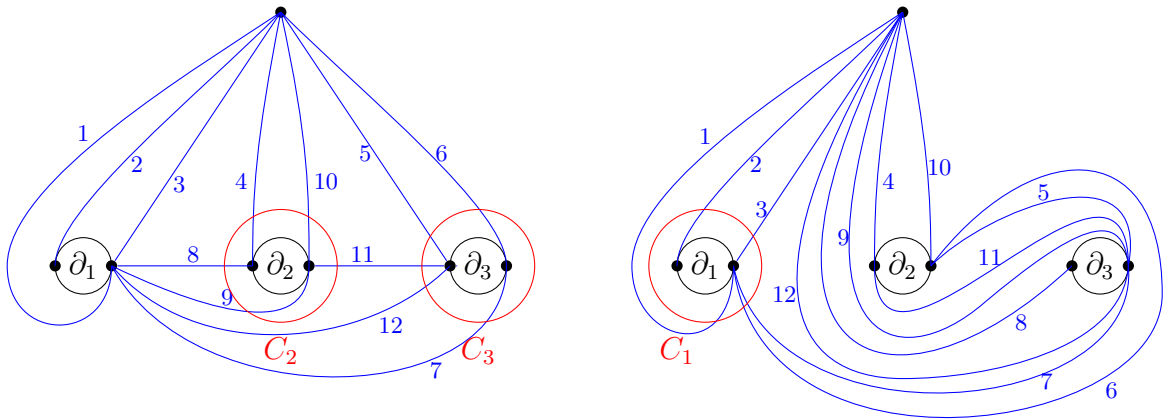
\end{ex}

\section{Topological and algebraic entropies}\label{subsec:entropy}

In this section, we give a partial estimate for the algebraic entropies of the cluster transformations induced by a pA mutation loops on the general marked surface.

\subsection{Algebraic entropy of weakly sign-stable mutation loops}
Recall the following result from the previous paper \cite{IK20}:

\begin{thm}[{\cite[Corollary 1.2]{IK19}}]\label{thm:alg_entropy}
Let $\phi=[\gamma]_{\bs}$ be a mutation loop represented by a path $\gamma: v_0 \to v$ which is sign-stable on the set $\Omega^{\mathrm{can}}_{(v_0)}$. Assuming that the spectral duality conjecture \cite[Conjecture 5.14]{IK20} holds true for the stable sign of $\gamma$ on $\Omega^{\mathrm{can}}_{(v_0)}$, we have
\begin{align*}
    \cE_\phi^a = \cE_\phi^x = \log \lambda_\phi,
\end{align*}
where $\lambda_\phi \geq 1$ is the cluster stretch factor.
\end{thm}  

Until now we have no complete analogue of \cref{thm:alg_entropy} for mutation loops with a weakly sign-stable representation path. 
Indeed, the presentation matrices $E_\gamma^{\boldsymbol{\epsilon}}$ for $\boldsymbol{\epsilon} \in \bS(\gamma)$, $\boldsymbol{\epsilon}\geq \boldsymbol{\epsilon}_\gamma^\stab$ do not have the Perron--Frobenius property as in \cite[Theorem 5.11]{IK20}, so that we have no satisfactory analogue of the cluster stretch factor which is intrinsic to the mutation loop in the sense explained in \cite[Remark 5.13]{IK20}. 
Nevertheless, we have the following stability result:


\begin{prop}\label{prop:entropy_weak_SS}
Let $\phi$ be a pA mutation loop on a marked surface $\Sigma$, and $\Omega:=\Omega^\bQ \setminus D_{\X,\Sigma}$.  
Then for any representation path $\gamma:(\tri,\ell) \to (\tri',\ell')$ and $\boldsymbol{\epsilon} \in \bS(\gamma)$ satisfying $\boldsymbol{\epsilon} \geq \boldsymbol{\epsilon}^\stab_{\gamma, \Omega}$, the presentation matrices $E_{\gamma}^{\boldsymbol{\epsilon}}$ have the stretch factor $\lambda_{\pi(\phi)}$ as a positive eigenvalue. Moreover, there exists a corresponding eigendirection $\bR \boldsymbol{x}_{\phi,\Omega} \subset \bR^I$ which does not depend on $\boldsymbol{\epsilon}$.
\end{prop}

\begin{proof}
One of the desired eigenvector of $E_\gamma^{\boldsymbol{\epsilon}}$ is given by $L_\phi^+=\iota(L_{\pi(\phi)}^+)$ (recall \cref{thm:pA_NS_general}). Indeed, we can compute
\begin{align*}
    (d\phi)_{\hL}(t_{\hL}(L_\phi^+)) 
    &= t_{\phi(\hL)}\circ \phi \circ \iota(L_{\pi(\phi)}^+) && (\mbox{\cite[Lemma 5.1]{IK20}}) \\
    &= t_{\phi(\hL)}\circ \iota \circ \pi(\phi)(L_{\pi(\phi)}^+) &&  (\mbox{\cref{lem:barU_in_X} (1)})\\
    &= t_{\phi(\hL)}\circ \iota (\lambda_{\pi(\phi)}\cdot L_{\pi(\phi)}^+) \\
    &=\lambda_{\pi(\phi)}\cdot t_{\phi(\hL)}(L_\phi^+) &&  (\mbox{\cref{lem:barU_in_X} (1)}).
\end{align*}
Here, $t_x: V \xrightarrow{\sim} T_x V$ is the translation isomorphism given by $v \mapsto \left.\frac{d}{ds}\middle|\right._{s=0}(x+sv)$
for a vector space $V$ and a point $x \in V$.
Then \cite[Lemma 5.2]{IK20} tells us that the coordinate vector $\boldsymbol{x}_{\phi,\Omega}:=\boldsymbol{x}_{(\tri,\ell)}(L_\phi^+) \in \bR^I$ is an eigenvector of $E_\gamma^{\boldsymbol{\epsilon}}$ with eigenvector $\lambda_{\pi(\phi)}$, which is manifestly independent of the sign $\boldsymbol{\epsilon}$. 
\end{proof}

As a corollary, we obtain an estimate of the algebraic entropy of the cluster $\X$-transfomation induced by a pA mutation loop on a marked surface:

\begin{cor}\label{cor:trop/alg-entropy}
We have \[\cE_\phi^x \geq \log \lambda_{\pi(\phi)}=\cE_{\pi(\phi)}^{\mathrm{top}}.\]
\end{cor}

\begin{proof}
For a non-zero vector $w \in \Omega:=\Omega_{(\tri,\ell)}^{\mathrm{can}} \subset \Omega^\bQ_{(\tri,\ell)} \setminus D_{\X,\Sigma}$, there exists an integer $n_0 \geq 0$ such that $\boldsymbol{\epsilon}_\gamma(\phi^n(w)) \geq \boldsymbol{\epsilon}_\gamma^\stab$ for all $n \geq n_0$. Then we have
\begin{align*}
    \boldsymbol{x}_{(\tri,\ell)}(\phi^n(\phi^{n_0}(w))) = E_n E_{n-1} \dots E_1 \cdot \boldsymbol{x}_{(\tri,\ell)}(\phi^{n_0}(w)) 
\end{align*}
for all $n \geq 0$. Here $E_m:=E_\gamma^{\boldsymbol{\epsilon}_m}$ with $\boldsymbol{\epsilon}_m:=\boldsymbol{\epsilon}_\gamma(\phi^{m-1}(\phi^{n_0}(w)))$, which satisfies $\boldsymbol{\epsilon}_m \in \bS(\gamma)$ and $\boldsymbol{\epsilon}_m \geq \boldsymbol{\epsilon}$ from the definition of $\Omega^{\mathrm{can}}_{(\tri,\ell)}$ and the condition (1) in \cref{def:weak_SS} for $m=1,\dots,n$. 

Let $\boldsymbol{x}_{\phi,\Omega} \in \bR^I$ be the common eigenvector of the matrices $E_\gamma^{\boldsymbol{\epsilon}}$ for the signs $\boldsymbol{\epsilon} \in \bS(\gamma)$, $\boldsymbol{\epsilon} \geq \boldsymbol{\epsilon}_\gamma^\stab$. 
Then by using a suitable invertible matrix $Q$, we can write
\begin{align*}
    E'_m:=QE_mQ^{-1}=
    \begin{pmatrix}
    \lambda_{\pi(\phi)} & T_m \\
    0 & E_m^0
    \end{pmatrix}
\end{align*}
for some $E_m^0 \in GL_{N-1}(\bR)$ and an $1\times (N-1)$ matrix $T_m$. 
Writing $Q\cdot\boldsymbol{x}_{(\tri,\ell)}(\phi^{n_0}(w)) = (a,\boldsymbol{v}) \in \bR \oplus \bR^{N-1}$, one can inductively compute
\begin{align*}
    Q\cdot\boldsymbol{x}_{(\tri,\ell)}(\phi^n(\phi^{n_0}(w)))
    &=E'_n \dots E'_1 (a,\boldsymbol{v}) \\
    &= \left(\lambda_{\pi(\phi)}^n a + \sum_{m=1}^n \lambda_{\pi(\phi)}^{n-m}T_mE^0_{m-1}\dots E^0_1\boldsymbol{v},~ E^0_n\dots E^0_1 \boldsymbol{v}\right) \\
    &= \lambda_{\pi(\phi)}^n\left( a + \sum_{m=1}^n \lambda_{\pi(\phi)}^{-m}T_mE^0_{m-1}\dots E^0_1\boldsymbol{v},~ \lambda_{\pi(\phi)}^{-n}E^0_n\dots E^0_1 \boldsymbol{v}\right).
\end{align*}
For the standard Euclidean norm $\|\cdot\|_2$ on $\bR^N$, we get
\begin{align*}
    \|E'_n \dots E'_1 (a,\boldsymbol{v})\|_2 \geq \lambda_{\pi(\phi)}^n \left\|a + \sum_{m=1}^n \lambda_{\pi(\phi)}^{-m}T_mE^0_{m-1}\dots E^0_1\boldsymbol{v}\right\|_2.
\end{align*}
Since the Lyapunov exponent
\begin{align*}
    \mathsf{\Lambda}_\phi(w)&:=\limsup_{n \to \infty} \frac{1}{n} \log \| \phi^n(w)\| 
    = \limsup_{n \to \infty} \frac{1}{n} \log \| \phi^n(\phi^{n_0}(w))\| 
\end{align*}
does not depend on the choice of the norm $\|\cdot\|$ on $\X_{(\tri,\ell)}^\uf(\bR^\trop)$ or the choice of a coordinate system, 
we can estimate it as 
\begin{align}\label{eq:estimate_Lyapunov}
    \mathsf{\Lambda}_\phi(w) = \limsup_{n \to \infty} \frac{1}{n} \log \| E'_n\dots E'_1 (a,\boldsymbol{v})\|_2 \geq \lambda_{\pi(\phi)}
\end{align}
for any $w \in \Omega$. 

Writing $\gamma:(\tri,\ell) \xrightarrow{\mathbf{k}} (\tri_1,\ell_1)$, let 
\begin{align*}
    \gamma^n: (\tri,\ell) \xrightarrow{\mathbf{k}} (\tri_1,\ell_1)\xrightarrow{\mathbf{k}} (\tri_2,\ell_2)\xrightarrow{\mathbf{k}} \dots \xrightarrow{\mathbf{k}} (\tri_n,\ell_n)
\end{align*}
be the iterated path, which represents the mutation loop $\phi^n$. 
Let $C^{(n)}:=C^{(\tri,\ell)}_{(\tri_n,\ell_n)}$ be the $C$-matrix assigned to the vertex $(\tri_n,\ell_n) \in \bExch_\Sigma$, which satisfies
\begin{align*}
    C^{(n+n_0)}= E_n\dots E_1 C^{(n_0)}
\end{align*}
for all $n \geq 0$. Its growth is estimated from below as
\begin{align*}
    \limsup_{n \to \infty} \frac{1}{n} \log\|C^{(n)} \| = \limsup_{n \to \infty} \frac{1}{n} \log\|E_n\dots E_1 C^{(n_0)} \| = \mathsf{\Lambda}_\phi(C^{(n_0)}) \geq \lambda_{\pi(\phi)},
\end{align*}
where $\mathsf{\Lambda}_\phi(C^{(n_0)})$ denotes the maximum of the Lyapunov exponents of the column vectors of $C^{(n_0)}$. 
Then the same line of arguments as in the proof of \cite[Proposition 4.7]{IK19} shows the inequality $\cE_\phi^x \geq \log\lambda_{\pi(\phi)}$.
\end{proof}

\begin{ex}
Here, we again use the notations in \cref{ex:3bdries+1pct_sph}.
The sign sequences $\boldsymbol{\epsilon}$ with $\boldsymbol{\epsilon} > \boldsymbol{\epsilon}^\stab_\gamma$ are of the form 
\begin{align*}
    \boldsymbol{\epsilon}=(+, +, +, \epsilon_1, \epsilon_2, -, +, -, -, +, \epsilon_3, \epsilon_4, -, +, +, +)
\end{align*}
with $\epsilon_i \in \{+,-\}$, 
all of which belong to $\bS(\gamma)$ in this case.
The characteristic polynomials of the presentation matrices which corresponds to these signs are shown below:
\begin{itemize}
    \item For $(+,+,+,+)$ and $(-,-,-,-)$, the characteristic polynomial is
    \[(\nu - 1) (\nu^3 - 1) (\nu^2 - 3\nu + 1) (\nu^6 - \nu^3 + 1).\]
    \item For $(+,-,+,-)$ and $(-,+,-,+)$, it is
    \[(\nu - 1) (\nu^3 - 1) (\nu^2 - 3\nu + 1) (\nu^6 - 3\nu^3 + 1).\]
    \item Otherwise, it is
    \[(\nu - 1) (\nu^3 - 1)^3 (\nu^2 - 3\nu + 1).\]
\end{itemize}
Thus, one can verify that the spectral radii of them is $(3+\sqrt{5})/2$.
Note that the factor $(\nu^3-1)$ corresponds to the permutation $\sigma_{\cR_\partial}(\phi)$ in the decomposition given in \cref{thm:block decomposition}.
\end{ex}

\begin{rem}
In view of the above examples, 
the authors expect that $\lambda_{\pi(\phi)}$ is the common spectral radius of the presentation matrices $E_\gamma^{\boldsymbol{\epsilon}}$ for $\boldsymbol{\epsilon} \in \bS(\gamma)$, $\boldsymbol{\epsilon} \geq \boldsymbol{\epsilon}_\gamma^\stab$, but have not been able to give a proof. 
\end{rem}

\subsection{Discussion: an estimate of \texorpdfstring{$\cE_\phi^a$}{E a} for a weakly sign-stable mutation loop}\label{subsec:discussion_A-transf}
It would be natural to expect the following:

\begin{conj}\label{conj:A-entropy}
For any pA mutation loop $\phi$ on a marked surface $\Sigma$, we have
\begin{align*}
    \cE_\phi^a \geq \log\lambda_{\pi(\phi)}.
\end{align*}
\end{conj}

For instance, the following additional condition, which is an ``$\A$-version'' of \cref{prop:entropy_weak_SS}, gives us a sufficient condition for the conjecture:
\begin{itemize}
    \item For any $\boldsymbol{\epsilon} \in \bS(\gamma)$ satisfying $\boldsymbol{\epsilon} \geq \boldsymbol{\epsilon}^\stab_{\gamma, \Omega}$, the
     presentation matrices $\check{E}_{\gamma}^{\boldsymbol{\epsilon}}$ have the stretch factor $\lambda_{\pi(\phi)}$ as positive eigenvalues. Moreover, there exists a corresponding eigendirection $\bR \boldsymbol{a}_{\phi,\Omega} \subset \bR^I$ which does not depend on $\boldsymbol{\epsilon}$.
\end{itemize}

\begin{lem}
Let $\phi=[\gamma]_\bs$ be a pA mutation loop on a marked surface $\Sigma$. If moreover it satisfies the condition above, then \cref{conj:A-entropy} holds true.
\end{lem}

\begin{proof}
The argument is similar to the proof of \cref{prop:entropy_weak_SS}. The growth of the $G$-matrices $G^{(n)}:=G^{(\tri,\ell)}_{(\tri_n,\ell_n)}$, which satisfies
\begin{align*}
    G^{(n+n_0)} = \check{E}_n \dots \check{E}_1 G^{(n_0)}
\end{align*}
for all $n \geq 0$ by the tropical duality \eqref{eq:tropical_duality_exchange}, is estimated from below as
\begin{align*}
    \limsup_{n \to \infty} \frac{1}{n} \log\|G^{(n)} \| = \limsup_{n \to \infty} \frac{1}{n} \log\|\check{E}_n\dots \check{E}_1 C^{(n_0)} \| = \mathsf{\Lambda}_\phi(G^{(n_0)}) \geq \lambda_{\pi(\phi)}.
\end{align*}
Then the inequality $\cE^a_\phi \geq \log\lambda_{\pi(\phi)}$ again follows from the same line of arguments as in the proof of \cite[Proposition 4.7]{IK19}.
\end{proof}

\begin{rem}
We have not been able to confirm the additional condition above. 
The difficulty arises from the fact that the matrix $\check{E}_\gamma^{\boldsymbol{\epsilon}}$ cannot be interpreted as the presentation matrix of a tangent map of $\phi^a$ unless $\cC_\gamma^{\boldsymbol{\epsilon}} \cap \cU_\Sigma^\uf(\bR^\trop) \neq \{0\}$. 
\end{rem}

\section{The cluster reduction and sign stability}\label{sec:cluster_reduction}
Here we reformulate the \emph{cluster reduction} procedure introduced in \cite{Ish19} in a way suited for our setting. 
As a slight generalization of \cite[Proposition 6.1]{IK20}, we show that a cluster Dehn twist admits a sign-stable representation path on a certain invariant set. 

A unification of the cluster reduction and the classical reduction procedure along a multicurve will be discussed in \cref{sec:C-reduction}. 
\subsection{Cluster reduction and sign stability}\label{subsec:cluster_reduction} 
Let $\bs:\bE_I \ni (t,\sigma) \mapsto (N^{(t,\sigma)},B^{(t,\sigma)})$ be any seed pattern. 
Fixing a vertex $v_0 \in \Exch_\bs$, let $\Delta_\bs^{\mathrm{FZ}}$ denote the Fomin--Zelevinsky cluster complex (see \emph{e.g.}, \cite{FST}). It is the simplicial complex with vertices parametrized by the cluster variables $A_i^{(v;v_0)}:=\mu_\gamma^* A_i^{(v)} \in \cO(\A_{(v_0)})$ for $i \in I_\uf$, $v \in \Exch_\bs$ and $\gamma:v_0 \to v$, and there is an edge between $A_i^{(v;v_0)}$ and $A_j^{(v';v_0)}$ if and only if $v=v'$. 
Its dual graph is the exchange graph $\Exch_\bs$. 

Fix a simplex $S$ of $\Delta_\bs^{\mathrm{FZ}}$ and consider its star $\Delta_\bs^{\mathrm{FZ}}(S)$, which is a subcomplex obtained as the union of the simplices that contain $S$. Let $\cTr \subset \Exch_\bs$ be the corresponding subgraph. See \cref{fig:cluster_complex}.
Then $\cTr$ is the exchange graph of the seed pattern $\bs|_\cTr$ obtained by \lq\lq freezing" the cluster variables that span the simplex $S$. In our setting, this seed pattern is described as follows.

\begin{figure}[h]
\begin{tikzpicture}
\foreach \i in {0,1,2,3,4}
{
\draw (72*\i:3) -- (72*\i+72:3);
\draw (0,0) -- (72*\i:3);
\draw[red,thick] (72*\i+36:1.5) -- (72*\i+108:1.5);
\draw[dashed] (72*\i+36:1.5) -- (72*\i+36:3.5);
}
\filldraw[myblue] (0,0) circle(2pt);
\draw[myblue] (0.2,0) node[above]{$S=A_3^{(v_0)}$};
\draw(3,0) node[right]{$\Delta_\bs^{\mathrm{FZ}}(S)$};
\draw (3,2.5) node{$\Exch_\bs$};
\draw[red] (-1.5,0.3) node[above]{$\cTr$};
\draw (216:3) node[below]{$A_1^{(v_0)}$};
\draw (-72:3) node[below]{$A_2^{(v_0)}$};
\draw[myblue](252:2.5) node[left]{$3$};
\draw[red] (-72:1.4) node[below]{$1$};
\draw[red] (-144:1.5) node[left]{$2$};
\end{tikzpicture}
    \caption{A portion of the Fomin--Zelevinsky cluster complex $\Delta_\bs^{\mathrm{FZ}}$ and its dual graph $\Exch_\bs$. Here we pick a local labeling of some vertices and edges (\emph{i.e.}, their lifts to $\bExch_\bs$).}
    \label{fig:cluster_complex}
\end{figure}
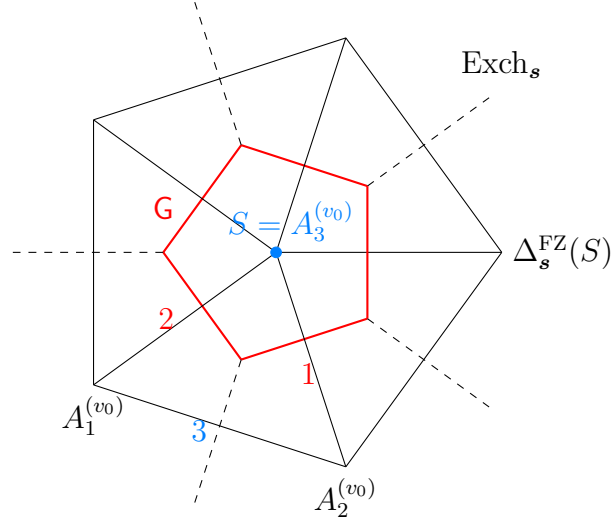

Fix a maximal simplex $S_0$ of $\Delta_\bs^{\mathrm{FZ}}$ containing $S$, which corresponds to a vertex $v_0$ of $\cTr \subset \Exch_\bs$. Choose a lift $(t_0,\sigma_0) \in \bE_I$ such that $v_0=\pi_{\mathrm{Ex}}([t_0,\sigma_0]_{\triv,\bs})$, where we denote the relation $\sim_{\triv}$ for the seed pattern $\bs$ by $\sim_{\triv,\bs}$. Then the face $S \subset S_0$ corresponds to a subset $K \subset I_\uf$. Let $J_\uf:=I_\uf \setminus K$, $J_\f:=K \cup I_\f$ and $J:=J_\uf \sqcup J_\f$. 
Notice that $\bE_J$ is not a subgraph of $\bE_I$, since $\fS_J \nsubseteq \fS_I$, Nevertheless, the initial seed $\bs^{(t_0,\sigma_0)}=(N^{(t_0,\sigma_0)},B^{(t_0,\sigma_0)})$ can be regarded as a seed for the index data $J$. 
Then the seed pattern $\bs|_\cTr$ is the unique seed pattern on $\bE_J$ with the same initial seed $\bs^{(t_0,\sigma_0)}$ at a vertex $(t'_0,\sigma'_0) \in \bE_J$. 
We have a unique isomorphism $\cTr \cong \Exch_{\bs|_\cTr}$ extending $v_0 \mapsto \pi_{\mathrm{Ex}}([t'_0,\sigma'_0]_{\triv,\bs|_\cTr})$. 
Under this isomorphism, let $\bs_\Exch|_\cTr: \widetilde{\cTr}:=\bExch_{\bs|_\cTr} \ni v \mapsto (N^{(v)}, B^{(v)})$ denote the induced pattern on the exchange graph. 

Let $N_\uf^{(v)}|_J:=\bigoplus_{j \in J_\uf} \bZ e_j^{(v)} \subset N^{(v)}$, and $M^{(v)}_\uf|_J$ its dual lattice. 
For each $v \in \cTr$, let $\X^\uf_{(v)}|_J:= T_{M_\uf^{(v)}|_J}$ be the torus associated with $M_\uf^{(v)}|_J$. 
The natural projection $M_\uf^{(v)}\to M_\uf^{(v)}|_J$ induces a projection $\X^\uf_{(v)} \to \X^\uf_{(v)}|_J$. Note that for any horizontal edge $v \overbar{k} v'$ in $\cTr$ with $k \in J_\uf$, the expression $\mu_k^*X_j^{(v')}$ for $j \in J_\uf$ does not contain $X_i^{(v)}$ for $i \in I\setminus J_\uf$. Hence these projections commute with the cluster transformations, and combine to give a positive rational surjective map 
\begin{align}\label{eq:reduction_morphism}
    \pi_\cTr: \X^\uf_\bs \to \X^\uf_{\bs|_\cTr},
\end{align}
which is defined on the union of the tori $\X^\uf_{(v)}$ associated with $v \in \cTr \subset \bExch_\bs$. 

\bigskip
\paragraph{\textbf{The cluster reduction homomorphism}}
Let $\bE_I^J \subset \bE_I$ denote the subgraph obtained by removing all the edges except for the horizontal edges labeled by $J$ and the vertical edges labeled by transpositions in the subgroup $\fS_I^J \subset \fS_I$ that preserves $J$. Note that the permutations in the subgroup $\fS_I^J$ are of the form
\begin{align*}
    \sigma= \begin{pmatrix}
    \ast & 0 & 0 \\
    0 & \ast & 0 \\
    0 & 0 & \ast
    \end{pmatrix}
\end{align*}
with respect to the decomposition $I=J_\uf \sqcup K \sqcup I_f$. Then we also have an inclusion $\bE_I^J \subset \bE_J$, and the restrictions of the seed pattens $\bs$ and $\bs|_\cTr$ coincide on $\bE_I^J$. 
We say that a mutation loop $\phi \in \Gamma_\bs$ is \emph{$\cTr$-reducible} if it admits a representation path contained in the subgraph $\bE_I^J$. A mutation loop $\phi \in \Gamma_\bs$ is said to be \emph{cluster-reducible} if it is $\cTr$-reducible for some $\cTr$. 
The $\cTr$-reducible mutation loops form a subgroup $\Gamma_\bs^\cTr \subset \Gamma_\bs$. 


\begin{lem}
We have an injective group homomorphism $\pi_\cTr^\grp:\Gamma_\bs^\cTr \hookrightarrow \Gamma_{\bs|_\cTr}$.
\end{lem}
We call $\pi_\cTr^\grp$ the \emph{cluster reduction homomorphism} for $\cTr$.

\begin{proof}
Note that for two vertices $v,v' \in \bE_I^J$, they are $\bs|_\cTr$-equivalent if and only if they are $\bs$-equivalent. Two edge paths contained in the subgraph $\cTr$ are $\bs|_\cTr$-equivalent whenever they are $\bs$-equivalent, since the path $\delta$ in \eqref{eq:equivalence of paths} can be chosen inside $\cTr$, and the commutative diagram \eqref{eq:equivalence of paths} can be projected via $\pi_\cTr$. 
The converse is also true by the synchronicity phenomenon \cite{Nak20}, which states that a periodicity of unfrozen variables implies that of frozen variables.
Hence we get a well-defined, injective group homomorphism $\pi_\cTr^\grp:\Gamma_\bs^\cTr \to \Gamma_{\bs|_\cTr}$. 
\end{proof}

\begin{rem}
The cluster reduction homomorphism is not surjective in general. 
For example, consider the seed pattern $\bs$ of finite mutation type $X_7$, whose initial quiver at $v_0$ is shown in \cref{fig:type_X7}. Here we set $I_\uf:=\{0,1,2,3,4\}$ and $I_\f:=\{5,6\}$. After the cluster reduction for the simplex $S:=\{A_3^{(v_0)},A_4^{(v_0)}\}$ with $J_\uf:=\{0,1,2\}$ and $J_\f:=\{3,4,5,6\}$, the path 
\begin{align*}
    \gamma: v_0 \xrightarrow{0,1,2} v_1 \xrightarrow{\sigma} v_2
\end{align*}
with $\sigma:=(0\ 1\ 2)(3\ 4\ 5\ 6)$ 
represents a mutation loop in $\Gamma_{\bs|_\cTr}$. However this does not come from $\Gamma_\bs^\cTr$, since $\sigma \notin \fS_I$ and hence the path $\gamma$ is not contained in the original labeled exchange graph $\bExch_\bs$. 

\end{rem}

\begin{figure}[ht]
\begin{tikzpicture}[>=latex]
\draw(0,0) circle(2pt) node[above=0.5em]{$0$};
\foreach \i in {30,150}
{
\draw(\i+30:2) circle(2pt);
\draw(\i-30:2) circle(2pt);
\draw[->>, thick, shorten >=4pt,shorten <=4pt]
(\i+30:2) -- (\i-30:2);
\qarrow{\i-30:2}{0,0}
\qarrow{0,0}{\i+30:2}
}
\foreach \i in {270}
{
{\color{myblue}
\draw(\i+30:2) circle(2pt);
\draw(\i-30:2) circle(2pt);
\draw[->>, thick, shorten >=4pt,shorten <=4pt]
(\i+30:2) -- (\i-30:2);
\qarrow{\i-30:2}{0,0}
\qarrow{0,0}{\i+30:2}
}
}
\draw(-2,0) node[left]{$1$};
\draw(120:2) node[left]{$2$};
\draw(60:2) node[right]{$3$};
\draw(2,0) node[right]{$4$};
{\color{myblue}
\draw(-60:2) node[below]{$5$};
\draw(-120:2) node[below]{$6$};
}
\end{tikzpicture}
    \caption{The initial quiver of type $X_7$, where the vertices $5$ and $6$ are frozen.}
    \label{fig:type_X7}
\end{figure}
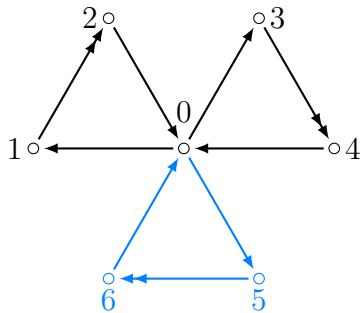

\begin{ex}
For the seed pattern $\bs_\Sigma$ associated with a marked surface $\Sigma$, recall the identification $\Exch_\Sigma \cong \Tri(\Sigma)$. The Fomin--Zelevinsky cluster complex $\Delta_\bs^\mathrm{FZ}$ is identified with (a connected component of) the \emph{tagged arc complex} (\cite[Theorem 7.11]{FST}). 

Fix an ideal arc $\alpha \in \Exch_\Sigma$ and set $S:=\{\alpha\}$. 
Then the vertices of the dual graph $\cTr$ of $\Delta_\bs^{\mathrm{FZ}}(S)$ are the ideal triangulations containing $\alpha$. Therefore, the seed pattern $\bs_\Sigma|_\cTr$ is obtained by prohibiting the flips along $\alpha$ from these triangulations. 
In effect, this amounts to view $\alpha$ as a boundary arc, and thus we have an identification $\X_{\bs_\Sigma|_\cTr}^\uf(\bR^\trop) \cong \X_{\Sigma(\alpha)}^\uf(\bR^\trop)$, where 
$\Sigma(\alpha)$ is the marked surface obtained by cutting $\Sigma$ along the arc $\alpha$. If we use the arc $\alpha'$ obtained by the flip along $\alpha$ inside some ideal triangulation, we get another seed pattern and another marked surface $\Sigma(\alpha')$. See \cref{fig:reduction} for an example. 
A mapping class on $\Sigma$ is $\cTr$-reducible if and only if it fixes the arc $\alpha$.

Since the cutting interpretation makes a duplication of the arc $\alpha$ into the boundary arcs $\beta^{(1)}$ and $\beta^{(2)}$, it is valid only for the cluster varieties $\X_\Sigma^\uf$, which do not care about frozen variables. The reader should not mix it up with the \emph{amalgamation} construction \cite{FG06}, which is rather defined for the cluster varieties $\X_\Sigma$ with frozen variables and goes in the opposite direction.
\end{ex}

\begin{figure}
\begin{tikzpicture}

\foreach \i in {0,1,2,3,4,5}
{
\node[fill, circle, inner sep=1pt] at (60*\i:2) {};
\draw[blue] (60*\i:2) -- (60*\i+60:2);
}
\draw[blue] (0:2) -- (120:2);
\draw[red] (0:2) -- node[midway,above]{$\alpha$} (180:2);
\draw[blue] (0:2) -- (240:2);
\node at (2,-2) {$\Sigma$};
\draw[-implies, double distance=2pt](3,0) to node[midway,above]{$f_\alpha$} (4,0);

\begin{scope}[xshift=7cm]
\foreach \i in {0,1,2,3,4,5}
{
\node[fill, circle, inner sep=1pt] at (60*\i:2) {};
\draw[blue] (60*\i:2) -- (60*\i+60:2);
}
\draw[blue] (0:2) -- (120:2);
\draw[red] (120:2) -- node[midway,right]{$\alpha'$} (240:2);
\draw[blue] (0:2) -- (240:2);
\end{scope}

\draw [ultra thick,-{Classical TikZ Rightarrow[length=4pt]},decorate,decoration={snake,amplitude=2pt,pre length=2pt,post length=3pt}] (0,-2.5) -- (0,-3.5);

\begin{scope}[yshift=-6cm]
\foreach \i in {0,1,2,3}
\node[fill, circle, inner sep=1pt] at (60*\i:2) {};
\foreach \i in {0,1,2}
\draw[blue] (60*\i:2) -- (60*\i+60:2);

\draw[blue] (0:2) -- (120:2);
\draw[red] (0:2) -- node[midway,above]{$\beta^{(1)}$} (180:2);
\end{scope}

\begin{scope}[yshift=-6.5cm]
\foreach \i in {0,3,4,5}
\node[fill, circle, inner sep=1pt] at (60*\i:2) {};
\draw[blue] (180:2) -- (240:2);
\draw[blue] (240:2) -- (300:2);
\draw[blue] (300:2) -- (0:2);
\draw[blue] (0:2) -- (240:2);
\draw[red] (0:2) -- node[midway,below]{$\beta^{(2)}$} (180:2);
\node at (2,-2) {$\Sigma(\alpha)$};
\end{scope}

\draw [ultra thick,-{Classical TikZ Rightarrow[length=4pt]},decorate,decoration={snake,amplitude=2pt,pre length=2pt,post length=3pt}] (7,-2.5) -- (7,-3.5);

\begin{scope}[xshift=7cm, yshift=-6cm]
\foreach \i in {2,3,4}
\node[fill, circle, inner sep=1pt] at (60*\i:2) {};
\draw[blue] (120:2) -- (180:2);
\draw[blue] (180:2) -- (240:2);
\draw[red] (240:2) --node[midway,left]{${\beta'}^{(1)}$} (120:2);
\end{scope}

\begin{scope}[xshift=7.5cm, yshift=-6cm]
\foreach \i in {0,1,2,4,5}
\node[fill, circle, inner sep=1pt] at (60*\i:2) {};
\foreach \i in {0,1,4,5}
\draw[blue] (60*\i:2) -- (60*\i+60:2);
\draw[red] (120:2) --node[midway,right]{${\beta'}^{(2)}$} (240:2);
\draw[blue] (0:2) -- (120:2);
\draw[blue] (0:2) -- (240:2);
\node at (2,-2.5) {$\Sigma(\alpha')$};
\end{scope}

\end{tikzpicture}
    \caption{Two examples of the cluster reduction.}
    \label{fig:reduction}
\end{figure}
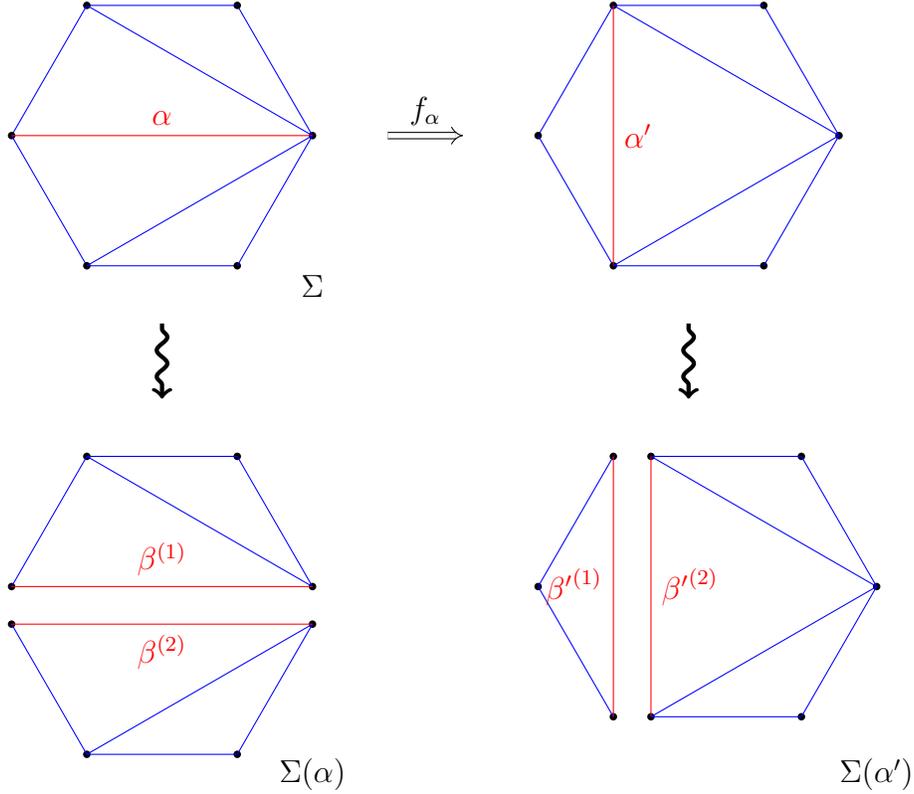

Slightly abusing the notation, let us denote a path $\gamma$ in $\Exch_\bs$ contained in the subgraph $\cTr$ by $\pi_\cTr(\gamma)$ when we regard it as a representation path of a mutation loop in $\Gamma_{\bs|_\cTr}$.  
Let us consider the tropicalization $\pi_\cTr^\trop: \X^\uf_\bs(\bR^\trop) \to \X^\uf_{\bs|_\cTr}(\bR^\trop)$ of the map \eqref{eq:reduction_morphism}, which is a well-defined surjective map. When no confusion can occur, we simply denote $\pi_\cTr^\grp$ and $\pi_\cTr^\trop$ by the same symbol $\pi_\cTr$. 

\begin{prop}\label{prop:SS_cluster_reduction}
Let $\gamma: v \to v'$ be a path contained in $\cTr$ which represents a mutation loop in $\Gamma_\bs^\cTr \subset \Gamma_\bs$. Then the path $\pi_\cTr(\gamma)$ is sign-stable on an $\bR_{>0}$-invariant set $\Omega \subset \X^\uf_{\bs|_\cTr}(\bR^\trop)$ if and only if $\gamma$ is sign-stable on $\widetilde{\Omega}:=\pi_\cTr^{-1}(\Omega\setminus \{0\}) \subset \X^\uf_\bs(\bR^\trop)$. Moreover we have 
\begin{align*}
    \rho(E_{\phi,\widetilde{\Omega}}^{(v_0)})=\rho(E_{\pi_{\cTr}(\phi),\Omega}^{(v_0)}),
\end{align*}
where $\phi:=[\gamma]_\bs \in \Gamma_\bs^{\cTr}$.
\end{prop}

\begin{proof}
Since the path $\gamma$ is contained in $\cTr$, its sign $\boldsymbol{\epsilon}_\gamma(w)$ at a point $w \in \X^\uf_\bs(\bR^\trop)$ only concerns with the coordinates $x_j^{(v)}$ for $j \in J_\uf$. In other words, we have $\boldsymbol{\epsilon}_\gamma(w)=\boldsymbol{\epsilon}_{\pi_{\cTr}(\gamma)}(\pi_\cTr(w))$ for all $w \in \X^\uf_\bs(\bR^\trop)$ and hence the first statement holds. For the second statement, note that for a horizontal edge $v \overbar{k} v'$ in $\cTr$, the matrix $E=E_{k,\epsilon}^{(v)}$ is block-decomposed into the form
\begin{align*}
    E = 
    \begin{pmatrix} 
    E|_J & 0 \\
    \ast & 1
    \end{pmatrix},
\end{align*}
where $E|_J$ is the $J_\uf \times J_\uf$-submatrix of $E$.
Since the presentation matrix $E_{\phi,\widetilde{\Omega}}^{(v_0)}$ is a product of such matrices and permutation matrices which preserve the block decomposition, it has the same block-decomposed form. Then $E_{\pi_{\cTr}(\phi),\Omega}^{(v_0)}$ is its $J_\uf\times J_\uf$-submatrix, and thus their spectral radii are the same.
\end{proof}

We mention to a relation between the cluster reduction and the \emph{special completion} introduced by Fock--Goncharov \cite{FG16}.

\begin{prop}\label{prop:cTr<->cone} 
There is a one-to-one correspondence between the set of $|J_\uf|$-dimensional simplices in  $\Delta_\bs^\mathrm{FZ}$ and the set of codimension $|J_\uf|$ cones in the fan $\fF_\bs^+$. In other words, the cluster varieties $\X_{\bs|_\cTr}^\uf$ are exactly the codimension $|J_\uf|$ strata of the special completion $\widehat{\X}_\bs^\uf$.
\end{prop}

\begin{proof}
For a $|J_\uf|$-dimensional simplex $S$ in $\Delta_\bs^{\mathrm{FZ}}$ and the corresponding subgraph $\cTr \subset \Exch_\bs$, take a vertex $v \in \cTr$ and consider the cone 
\begin{align*}
    \cC^{+,J_\uf}_{(v)}:=\{ w \in \X_{(v)}^\uf(\bR^\trop) \mid x_i^{(v)}(w) \geq 0\mbox{ for all $i \in I_\uf$},~x_j^{(v)}(w) = 0 \mbox{ for all $j \in J_\uf$}\},
\end{align*}
where we choose a labeling (\emph{i.e.}, a lift of $v$ to $\bExch_{\bs|_\cTr}$) and $J_\uf \subset I_\uf$ is the index set corresponding to the simplex $S$ in this labeling. 
Then the corresponding subset in $\X_\bs^\uf(\bR^\trop)$ is a codimension $|J_\uf|$ cone in the fan $\fF_\bs^+$, which depends only on the simplex $S$, since $x_i^{(v')}(w) = x_i^{(v)}(w)$ for all $i \in I_\uf$, $w \in \cC^{+,J}_{(v)}$ and any edge $v \overbar{k} v'$ labeled by $k \in J_\uf$. Thus we denote this cone by $\cC_S^+ \subset \X_\bs^\uf(\bR^\trop)$. 

Since each codimension $|J_\uf|$ cone in the fan $\fF_\bs^+$ is of the form $\cC=\cC_{(v)}^{+,J_\uf}$ for some $v \in \Exch_\bs$, 
the correspondence $S \mapsto \cC_S^+$ is surjective. The injectivity follows from \cite[Theorem 6.2]{CL20}\footnote{This theorem shows that the \lq\lq equivalence property" (T2) in \cite[Lemma 2.11]{Ish19} holds for any skew-symmetrizable seed pattern, which was one of the basic assumptions in the Nielsen--Thurston theory.}, since generators of the cone $\cC_{(v)}^{+,J_\uf} \subset \X_{(v_0)}^\uf(\bR^\trop)$ in the chart associated with a vertex $v_0 \in \Exch_\bs$ correspond to the cluster variables $A_j^{(v;v_0)}$ for $j \in J_\uf$ as their $g$-vectors, and the set $S=\{A_j^{(v;v_0)}\}_{j \in J_\uf}$ defines a simplex in $\Delta_\bs^{\mathrm{FZ}}$.
\end{proof}

\subsection{Sign stability of cluster Dehn twists}
Here we slightly generalize \cite[Proposition 6.1]{IK20} on the sign stability of Dehn twists to that on \emph{cluster Dehn twists}, which is introduced in \cite{Ish19} as a clsuter algebraic analogue of Dehn twists (and half-twists) using the cluster reduction.

For any seed pattern $\bs$, a mutation loop $\phi \in \Gamma_{\bs}$ of infinite order is called a \emph{cluster Dehn twist} if it admits a representation path of the form $\gamma: v \overbar{k_1} v' \overbar{(k_1\ k_2)} v''$ for some $k_1\neq k_2 \in \{1,\dots,N_\uf\}$. Equivalently, a cluster Dehn twist is a mutation loop $\phi$ of infinite order which is $\cTr$-reducible for the dual graph $\cTr$ of the simplex $S=\{A_j^{(v;v_0)}\}_{j \in I \setminus \{k_1,k_2\}}$ for some $v \in \bExch_\bs$ and $k_1 \neq k_2$. 

We call a representation path $\gamma$ as above a \emph{shortest path} for $\phi$. 
Let $\varepsilon = \sgn(b^{(v_0)}_{k_1, k_2})$ and $\ell:=|b^{(v_0)}_{k_1, k_2}|$, and consider the $\bR_{>0}$-invariant subset
\begin{align*}
    \Omega_{(v_0)}^{k_1,k_2}:=
    \begin{cases}
        \{ w \in \X_{(v_0)}(\bR^\trop) \mid (x_{k_1}^{(v_0)}(w),x_{k_2}^{(v_0)}(w)) \neq (0,0)\}, &\hspace{-3cm} \mbox{if $\ell = 2$,} \vspace{2.5mm}\\
        \{ w \in \X_{(v_0)}(\bR^\trop) \mid \varepsilon(x_{k_1}^{(v_0)}(w),x_{k_2}^{(v_0)}(w)) \notin \bR_{\geq 0} (1,(-\ell -\sqrt{\ell^2 -4})/2)\},
        & \\
        & \hspace{-3cm} \mbox{if $\ell \geq 3$},
    \end{cases}
\end{align*}
which contains the subset $\Omega^{\mathrm{can}}_{(v_0)}$. 

\begin{prop}[Sign stability of cluster Dehn twists]\label{thm:cluster_Dehn_SS}
For any (skew-symmetric) seed pattern $\bs$ and a cluster Dehn twist $\phi \in \Gamma_\bs$, each shortest path $\gamma: v_0 \overbar{k_1} v_1 \overbar{(k_1\ k_2)} v_2$ for $\phi$ is sign-stable on the $\bR_{>0}$-invariant subset $\Omega_{(v_0)}^{k_1,k_2}$ with stable sign $(-\varepsilon)$, where $\varepsilon = \sgn(b^{(v_0)}_{k_1, k_2})$.
Its cluster stretch factor is
\begin{align*}
    \lambda_\phi=\frac{\ell+\sqrt{\ell^2-4}}{2},
\end{align*}
where $\ell:=|b_{k_1,k_2}^{(v_0)}| \geq 2$. 
If $N_\uf \geq 3$, then we necessarily have $\ell=2$ and $\lambda_\phi=1$. 
\end{prop}

\begin{proof}
When the rank $N_\uf=2$ and $I=\{k_1,k_2\}$, the mutation loop $\phi$ has infinite order if and only if $|b^{(v_0)}_{k_1,k_2}| \geq 2$. 
When $N_\uf \geq 3$, the path $\gamma$ represents a mutation loop if and only if the following conditions hold \cite[Lemma 2.32]{Ish19}:
\begin{itemize}
    \item $|b^{(v_0)}_{k_1,k_2}| = 2$;
    \item For all $j \neq k_1,k_2$, we have $b^{(v_0)}_{k_1,j} = b^{(v_0)}_{j,k_2}$ and $\sgn(b^{(v_0)}_{k_1,j}) = -\sgn(b^{(v_0)}_{k_1,k_2})$.
\end{itemize}

Recall the sign stability of the shortest mutation loop on the $\ell$-Kronecker quiver.  
Namely, let $\bs_\ell$ be a seed pattern such that
\begin{align*}
    B^{(\oline{v_0})} = \begin{pmatrix}0 & -\ell \\ \ell & 0\end{pmatrix}
\end{align*}
for $\oline{v_0} \in \bExch_{\bs_\ell}$. 
Consider the path
\begin{align*}
\oline{\gamma}: \oline{v_0} \overbar{1} \oline{v_1} \overbar{(1\ 2)} \oline{v_2},
\end{align*}
which represents a mutation loop $\oline{\phi}$.
Then in \cite[Example 5.3]{IK19}, we have seen
that the path $\oline{\gamma}$ is sign-stable on
\begin{align*}
    \oline{\Omega}:=\begin{cases}
        \cX_{(\oline{v_0})}(\bR^\trop) & \mbox{if $\ell=2$}, \\
        \cX_{(\oline{v_0})}(\bR^\trop) \setminus \bR_{>0} \cdot (1, (-\ell -\sqrt{\ell^2 -4})/2) & \mbox{if $\ell \geq 3$}. 
    \end{cases}
\end{align*}
In both cases, the cluster stretch factor of  $\oline{\phi}$ is given by $\lambda_\phi= ({\ell+\sqrt{\ell^2-4}})/{2}$.

Take the simplex $S:=\{A_j^{(v_0;v_0)}\}_{j \in I \setminus \{k_1,k_2\}}$ in $\Delta_\bs^\mathrm{FZ}$. 
Then the seed pattern $\bs|_\cTr$ coincides with $\bs_{\ell}$, where $\ell = |b_{k_1, k_2}^{(v_0)}|$.
We consider the case $\bs|_\cTr = \bs_{\ell}$, equivalently $b_{k_1, k_2}^{(v_0)} <0$.
The other case is proven in the same way.
It is clear that $\pi_\cTr^{-1}(\Omega \setminus \{0\}) = \Omega^{k_1, k_2}_{(v_0)}$.
Thus the assertion follows from \cref{prop:SS_cluster_reduction}.
\end{proof}

\section{Generalized reduction and the hereditariness}
\label{sec:C-reduction}
Recall the cluster reduction (\cref{subsec:cluster_reduction}) and the reduction along a multicurve (\cref{sec:reduction}), both parametrized by a suitable rational polyhedral cone $\cC \in \X_\bs^\uf(\bR^\trop)$. We propose here the following unified and generalized construction. We say that an edge $v \overbar{k} v'$ in $\bExch_\bs$ is \emph{$\cC$-compatible} if $x_k^{(v)}(w) = 0$ (or equivalently, $x_k^{(v')}(w) = 0$) holds for all $w \in \cC$. We also say that an edge of $\mathrm{Exch}_\bs$ is \emph{$\cC$-compatible} if the associated coordinate satisfies the same condition. 

\begin{dfn}[$\cC$-reduction]\label{d:C-reduction}
Let $\bs$ be a seed pattern and consider a rational polyhedral cone $\cC \subset \cX^\uf_\bs(\bR^\trop)$.
Let $\Gamma_\bs^\cC \subset \Gamma_\bs$ denote the subgroup consisting of mutation loops that preserve $\cC$ setwisely. We call $\cC$ a \emph{reduction cone} if it satisfies the following axioms:
\begin{enumerate}
    \item[(i)] Let $E(\cC) \subset \mathrm{Exch}_\bs$ be the $\Gamma_\bs^\cC$-invariant subgraph consisting of the $\cC$-compatible edges. 
    Then there exists
    a $\Gamma_\bs^\cC$-invariant subgraph $\cTr^\bot$ of $\mathrm{Exch}_\bs$ with the corresponding edge contraction $c_{\cTr^\bot}:\mathrm{Exch}_\bs \to \mathrm{Exch}_\bs/\cTr^\bot$\footnote{By convention, an edge contraction produces no multiple edges.}
    such that $c_{\cTr^\bot}(E(\cC))$ is identified with a connected subgraph of the exchange graph $\mathrm{Exch}_{\bs_\cC}$ of a seed pattern $\bs_\cC$.\footnote{
    Of course, this condition is not sufficient to characterize the seed pattern $\bs_\cC$.
     } 
    \item[(ii)] There exists a group homomorphism $\pi_\cC=\pi_\cC^{\mathrm{grp}}:\Gamma_\bs^\cC \to \Gamma_{\bs_\cC}$ such that 
    each mutation loop $\phi \in \Gamma_\bs^\cC$ (resp. $\pi_\cC(\phi)$) admits a representation path $\gamma$ (resp. $\oline{\gamma}$) such that the horizontal edges of $\oline{\gamma}$ are precisely the $\cC$-compatible edges of $\gamma$. 
    \item[(iii)] There exists a continuous map $\pi_\cC=\pi_\cC^\trop: \X_{\bs}^\uf(\bR^\trop) \to \X_{\bs_\cC}^\uf(\bR^\trop)$, which is $\Gamma_\bs^\cC$-equivariant and PL on an open dense subset and satisfies $\pi_\cC(\cC)=\{0\}$.
    \item[(iv)] If $\phi \in \Gamma_\bs^\cC$ admits a representation path $\gamma$ as in (ii) which is weakly sign-stable on a subset $\Omega \subset \X_\bs^\uf(\bR^\trop)$ and $\cC$-hereditary (\cref{d:hereditary}), then $\oline{\gamma}$ is sign-stable on $\pi_\cC(\Omega) \subset \X_{\bs_\cC}^\uf(\bR^\trop)$.
    \item[(v)] For any representation path $\gamma$ of a mutation loop $\phi \in \Gamma_\bs^\cC$ which is weakly sign-stable on a subset $\Omega \subset \X_\bs^\uf(\bR^\trop)$ and $\cC$-hereditary, we have $\lambda_{\gamma,\Omega}= \lambda_{\pi_\cC(\phi),\pi_\cC(\Omega)}$.
    Here 
    \begin{align*}
        \lambda_{\gamma, \Omega} := \max_{\bS(\gamma) \ni \boldsymbol{\epsilon} \geq \boldsymbol{\epsilon}^\stab_\gamma} \rho(E_\gamma^\epsilon).
    \end{align*}
\end{enumerate}
Then we call the tuple $(\cC,\cTr^\bot,\pi_\cC^{\mathrm{grp}},\pi_\cC^\trop)$ a \emph{reduction data}. 
\end{dfn}

As a trivial example, let us consider the cone $\cC:=\{0\}$. Then any edge in $\mathrm{Exch}_\bs$ is $\cC$-compatible. We can choose $\cTr^\bot:=\emptyset$, so that $\bs_\cC=\bs$. Together with the identity maps $\pi_\cC^\mathrm{grp}$ and $\pi_\cC^\trop$, the tuple $(\cC,\cTr^\bot,\pi_\cC^{\mathrm{grp}},\pi_\cC^\trop)$ forms a reduction data. The condition (iv) holds since a weakly sign-stable and $\cC$-hereditary path is necessarily sign-stable. 

For the other extremity, consider the cone $\cC:=\X_\bs^\uf(\bR^\trop)$. Then it is not a reduction cone. Indeed, in this case $E(\cC)=\emptyset$ and this forces $\bs_\cC$ to be the trivial seed pattern with $|I|=0$. In particular the maps $\pi_\cC^\mathrm{grp}$ and $\pi_\cC^\trop$ must be trivial maps. Hence $\lambda_{\pi_cC(\phi),\pi_\cC(\Omega)}=\lambda_{\mathrm{id},\{0\}}=-\infty$\footnote{The unique linear map on $\{0\}$ has no eigenvalues, so has empty spectrum.} (in other words, it loses the information on the algebraic entropy) and the condition (v) fails to hold. Informally speaking, when the cone $\cC$ is \lq\lq too large" so that it possesses non-trivial dynamical complexity, then the reduction procedure may reduces the entropy so that the condition (v) fails to hold.

\begin{lem}\label{lem:reduction_connectivity}
Let $\bs$ a seed pattern, and $\cC \subset \X_\bs^\uf(\bR^\trop)$ a rational polyhedral cone. If we set $\cTr^\bot:=\Exch_\bs \setminus E(\cC)$, then $c_{\cTr^\bot}(E(\cC))$ is a connected graph. 
\end{lem}

\begin{proof}
If $E(\cC)$ is connected, then we have nothing to prove. Otherwise, let $G_1,G_2 \subset E(\cC)$ be two connected components and take a shortest edge path $\gamma$ connecting $G_1$ and $G_2$. Since each edge in $\gamma$ belongs to the graph $\cTr^\bot$, the path $\gamma$ is contracted to a vertex via $c_{\cTr^\bot}$. Thus $c_{\cTr^\bot}(G_1 \sqcup G_2)$ is connected. 
\end{proof}

Here are non-trivial examples:

\begin{lem}\label{lem:C-red_arcs}
For any seed pattern, take a simplex $S$ in $\Delta_\bs^{\mathrm{FZ}}$ with the dual graph $\cTr \subset \Exch_\bs$, and consider the cone $\cC:=\cC_S^+ \subset \X_\bs^\uf(\bR^\trop)$. Then we have $\Gamma_\bs^\cC=\Gamma_\bs^\cTr$ and $\cC$ is a reduction cone such that $\bs_\cC=\bs|_\cTr$, $\pi_\cC^{\mathrm{grp}}=\pi_{\cTr}^\grp$ and $\pi_\cC^\trop=\pi_\cTr^\trop$.
\end{lem}

\begin{proof}
It is clear that $\Gamma_\bs^\cTr \subset \Gamma_\bs^\cC$. Conversely, let $\phi \in \Gamma_\bs^\cC$. Then it fixes the cluster variables $A_j^{(v;v_0)}$ for $j \in J_\uf$, and \cite[Theorem 6.2]{CL20} implies that $\phi$ admits a representation path contained in $\cTr$ (that is, a path with the only indices in $J_\uf$). Thus $\Gamma_\bs^\cC=\Gamma_\bs^\cTr$. 

It also follows from \cite[Theorem 6.2]{CL20} that the subgraph $E(\cC)$ of $\cC$-compatible edges coincides with the subgraph $\cTr$. 
Choosing $\cTr^\bot:=\emptyset$, we get $c_{\cTr^\bot}(E(\cC))=E(\cC) = \mathrm{Exch}_{\bs|_\cTr}$ as explained in \cref{subsec:cluster_reduction}.
The maps $\pi_\cC^{\mathrm{grp}}=\pi_{\cTr}^\grp$ and $\pi_\cC^\trop=\pi_\cTr^\trop$ clearly satisfies the conditions (ii) and (iii) with $\oline{\gamma}:=\pi_\cTr(\gamma)$.
The condition (iv) holds since a weakly sign-stable path contained in $\cTr$ which is $\cC$-hereditary is necessarily sign-stable. 
Then the condition (v) follows from \cref{prop:SS_cluster_reduction}. 
\end{proof}

\begin{lem}\label{lem:C-red_curves}
For the seed pattern $\bs=\bs_\Sigma$ associated with a marked surface $\Sigma$, take the cone $\cC:=\cC(\cR)$ associated with a curve system $\cR$. Then we have $\Gamma_\Sigma^\cC=\Gamma_{\Sigma,\cR}$ and $\cC$ is a reduction cone such that $\bs_\cC=\bs_{\Sigma_\cR}$, $\pi_\cC^\mathrm{grp}=\pi_\cR^\grp$ and $\pi_\cC^\trop=\pi_\cR^\trop$.
\end{lem}

\begin{proof}
The equality $\Gamma_\Sigma^\cC=\Gamma_{\Sigma,\cR}$ follows immediately from the definitions. 
Choose $\cTr^\bot:=\mathrm{Exch}_\Sigma \setminus E(\cC)$, then \cref{lem:collapsed_tri} tells us that $c_{\cTr^\bot}(E(\cC))$ is a subgraph of $\mathrm{Exch}_{\Sigma_\cR}$, which is connected by \cref{lem:reduction_connectivity}. The maps $\pi_\cC^{\mathrm{grp}}:=\pi_{\cR}^\grp$ and $\pi_\cC^\trop:=\pi_\cR^\trop$ clearly satisfies the conditions (ii) and (iii) with the reduced path $\oline{\gamma}$ described in \cref{subsec:reduction_path}.
The condition (iv) holds since a weakly sign-stable path $\gamma:(\tri,\ell) \to (\tri',\ell')$ contained in $c_{\cTr^\bot}(E(\cC))$ and $\cC$-hereditary necessarily produces a sign-stable path $\oline{\gamma}:(\oline{\tri},\lambda) \to (\oline{\tri'},\lambda')$ after the reduction.
Since the cluster stretch factor $\lambda_{\pi_\cR(\phi),\pi_\cR(\Omega)}$ does not depend on the choice of a sign-stable path by \cite[Remark 5.13]{IK20}, by \cref{cor:reduction_spec} we get
\begin{align*}
    \lambda_{\pi_\cR(\phi),\pi_\cR(\Omega)} =  \rho(E_{\pi_\cR(\phi),\pi_\cC(\Omega)}^{(\oline{\tri},\lambda)}) 
    = \max_{\bS(\gamma) \ni \boldsymbol{\epsilon} \geq \boldsymbol{\epsilon}^\stab_\gamma} \rho(E_{\gamma}^{\boldsymbol{\epsilon}}) =\lambda_{\gamma,\Omega}
\end{align*}
for a point $\hL \in \Omega$ with $\boldsymbol{\epsilon}_\gamma(\hL)=\boldsymbol{\epsilon}_{\gamma,\Omega}^\stab$. This proves the condition (v).
\end{proof}
Now we can consider a \lq\lq mixed" reduction involving both ideal arcs and simple closed curves. Let $\Sigma$ be a marked surface, and consider a collection $\cR=\{\alpha_1,\dots,\alpha_{K_1},C_1,\dots,C_{K_2}\}$ of mutually disjoint curves, where $\alpha_i$ are ideal arcs and $C_j$ are simple closed curves. Let $\Sigma_\cR$ be the marked surface obtained by cutting $\Sigma$ along the curves in $\cR$, each ideal arc (resp. simple closed curve) producing a pair of boundary intervals (resp. punctures). See \cref{fig:general_reduction}. Take an ideal triangulation $\tri$ containing $\alpha_i$ for $i=1,\dots,K_1$. For $i=1,\dots,K_1$, consider the basic lamination $l_i^\tri$ defined by the condition $x_j^\tri(l_i^\tri)=\delta_{ij}$ for $j=1,\dots,N$. They do not depend on the choice of $\tri$. 

\begin{figure}
    \centering
    \begin{tikzpicture}[scale=.65]
    \draw(0,0) circle[x radius=5cm, y radius=3cm];
    \begin{scope}[xshift=-1.5cm,yshift=-0.5cm]
        \draw(1,0) arc[start angle=0, end angle=-180, radius=1cm];
    \draw(0.707,-0.707) to[out=135, in=45] (-0.707,-0.707);
    \end{scope}
    \draw[fill](-2.5,1) circle(3.5pt);
    \draw[fill](-0.5,1) circle(3.5pt);
    {\color{red}
    \draw[shorten >=1.7pt, shorten <=1.7pt](-2.5,1) -- (-0.5,1);
    \draw(-1.5,-1.5) arc[start angle=90, end angle=270, y radius=0.68cm, x radius=0.5cm];
    \draw[dashed] (-1.5,-1.5) arc[start angle=90, end angle=-90, y radius=0.68cm, x radius=0.5cm];}
    \begin{scope}[xshift=2cm, yshift=0.5cm]
        \draw [fill=gray!20] (0,0) circle(0.8cm);
        \foreach \i in {0,90,180,270}
    	    \fill(\i:0.8) circle(4pt) coordinate(A\i); 
    \end{scope}
    
    \draw [ultra thick,-{Classical TikZ Rightarrow[length=4pt]},decorate,decoration={snake,amplitude=2pt,pre length=2pt,post length=3pt}](6,0) -- (8.5,0);
    
    \draw[fill](11.85,-2) circle(3.5pt);
    \draw[fill](14.5,-2) circle(3.5pt);
    \draw [fill=gray!20] (13,1) ellipse (1 and 0.75);
    \draw[fill](12,1) circle(3.5pt);
    \draw[fill](14,1) circle(3.5pt);
    \begin{scope}[xshift=2cm, yshift=0.5cm, shift={(14.5,0)}]
        \draw [fill=gray!20] (0,0) circle(0.8cm);
        \foreach \i in {0,90,180,270}
    	    \fill(\i:0.8) circle(4pt) node (A\i) {}; 
    \end{scope}
    \draw (11.85,-2) .. controls (10.45,-1.2) and (12,-0.5) .. (13.25,-0.5) .. controls (14.5,-0.5) and (16,-1.3) .. (14.5,-2);
    \draw (12,-2) .. controls (8.8,-2.65) and (7.5,3) .. (14,3) .. controls (21,3) and (20,-3) .. (14.5,-2);
    \end{tikzpicture}
    \caption{The reduction along a collection $\cR$ including both an ideal arc and a simple closed curve.}
    \label{fig:general_reduction}
\end{figure}
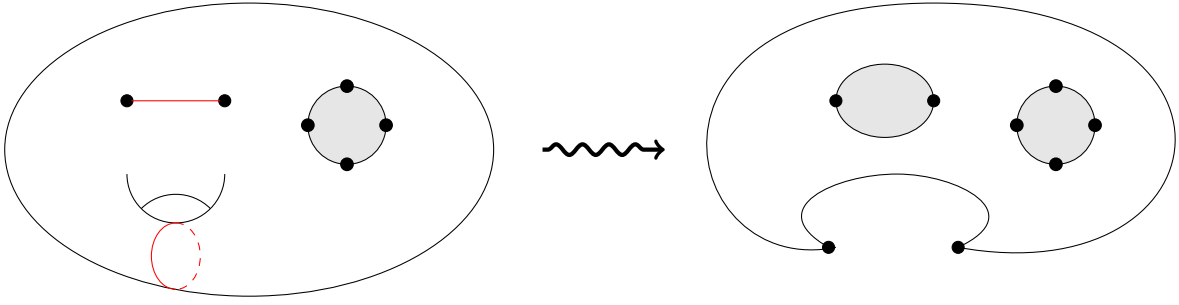

\begin{prop}\label{prop:mixed reduction}
Let $\cR$ be a collection of curves on a marked surface $\Sigma$ as above. Take the cone $\cC=\cC(\cR)$ consisting of the real $\X$-laminations $\hL=(L,\sigma_L=(\sigma_p)_{p\in P})$ of the form 
\begin{align*}
    L=\bigsqcup_{i=1}^{K_1}w_i l_i^\tri \sqcup \bigsqcup_{j=1}^{K_2}w_j C_j
\end{align*}
with $w_i,w_j \geq 0$ and $\sigma_p=+$ for all $p \in P$. Then $\Gamma_\Sigma^\cC$ consists of mutation loops that preserve $\cR$ setwisely, and $\cC$ is a reduction cone such that $\bs_\cC=\bs_{\Sigma_\cR}$.
\end{prop}

\begin{proof}
The first assertion is clear from the definitions.  Let $\cC_2 \subset \cC$ be the subcone spanned by the simple closed curves in $\cR$, and take $\cTr^\bot:=\mathrm{Exch}_\Sigma \setminus E(\cC_2)$. The edges in $E(\cC_2)$ are $\cR_2$-compatible edges, where $\cR_2:=\{C_1,\dots,C_{K_2}\}$. Then the contracted graph $c_{\cTr^\bot}(\mathrm{Exch}_\Sigma)$ is identified with a connected subgraph of $\mathrm{Exch}_{\Sigma_{\cR_2}}$ by the proof of \cref{lem:C-red_curves}. Then the image $c_{\cTr^\bot}(E(\cC))$ of the subgraph $E(\cC) \subset E(\cC_2)$ consists of $\oline{\cR}_1$-compatible edges, where $\oline{\cR}_1:=\{\oline{\alpha}_1,\dots,\oline{\alpha}_{K_1}\}$ and $\oline{\alpha}_i$ is the image of $\alpha_i$ in $\Sigma_{\cR_2}$. Hence from the proof of \cref{lem:C-red_arcs}, we see that $c_{\cTr^\bot}(E(\cC))$ is isomorphic to the exchange graph $\mathrm{Exch}_{\Sigma_\cR}$ of the marked surface $\Sigma_\cR$ obtained by cutting $\Sigma_{\cR_2}$ along the curves in $\oline{\cR}_1$. We define the other data by compositions
\begin{align*}
    \pi_\cC^{\mathrm{grp}}: \Gamma_\Sigma^\cC \subset \Gamma_\Sigma^{\cC(\cR_2)} \xrightarrow{\pi_{\cC(\cR_2)}^{\mathrm{grp}}} \Gamma_{\Sigma_{\cR_2}} \xrightarrow{\pi_{\cC(\oline{\cR}_1)}^{\mathrm{grp}}} \Gamma_{\Sigma_\cR}
\end{align*}
and
\begin{align*}
    \pi_\cC^\trop: \X_\Sigma^\uf(\bR^\trop) \xrightarrow{\pi_{\cC(\cR_2)}^\trop} \X_{\Sigma_{\cR_2}}^\uf(\bR^\trop) \xrightarrow{\pi_{\cC(\oline{\cR}_1)}^\trop} \X_{\Sigma_\cR}^\uf(\bR^\trop).
\end{align*}
Then the conditions (ii)--(v) can be checked by a concatenation of arguments in \cref{lem:C-red_arcs,lem:C-red_curves}.
\end{proof}
\appendix

\section{Enhanced train tracks}\label{sec:train track}

A train track is a powerful tool to study laminations and foliations and their behavior.
Here we give a crush course in train tracks and generalize it for $\cX$-laminations.
This section is mainly written with reference to \cite{Hat}.
We refer the reader to \cite{PH,Mos} for a more detailed presentation.

\subsection{Definition}\label{subsec:train track_lamination}
Let $\Sigma$ be a marked surface.
A finite trivalent graph $\tau \subset \Sigma^\circ$ embedded in $\Sigma^\circ := \Sigma \setminus \bigsqcup_{p \in P} D_p \sqcup \bigsqcup_{m \in M_\partial} D_m$, where $D_m$ denotes a small half disk around $m \in M_\partial$, is a \emph{train track} if 
\begin{itemize}
    \item $\tau$ is transverse to $\partial \Sigma^\circ$,
    \item each vertex $v$ looks like the configuration in \cref{fig:switch_cond}, and
    \item each complementary region $C$ satisfies 
    \[ \chi(C) - \frac{1}{2} \#(\mbox{cusps of $C$}) - \frac{1}{4} \#(\mbox{corners of $C$}) < 0, \]
    where $\chi(C)$ denotes the Euler characteristic of $C$.
\end{itemize}

A \emph{measure} on a train track $\tau$ is a function 
\[ \nu: \{ \mbox{edges of $\tau$} \} \too \bR_{\geq 0} \]
satisfying the \emph{switch condition} $\nu(e_0) = \nu(e_1) + \nu(e_2)$ for each vertex $v$ (also called a \lq switch\rq) as shown in \cref{fig:switch_cond}.
We denote by $V(\tau)$ the set of measures on a train track $\tau$. 
A train track is said to be \emph{recurrent} if it admits a measure $\nu$ such that $\nu(e) >0$ for any edge $e \subset \tau$.

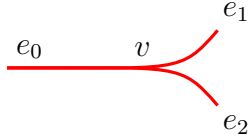
\begin{figure}[h]
    \centering
    \begin{tikzpicture}[auto]
    \draw[very thick, red] (-1.3,0) .. controls (1,0) and (1,-0.08) .. (1.5,0.5);
    \draw[very thick, red] (-1.3,0) .. controls (1,0) and (1,0.08) .. (1.5,-0.5);
    \node at (0.5,0.25) {$v$};
    \node at (1.75,0.75) {$e_1$};
    \node at (1.75,-0.75) {$e_2$};
    \node at (-1,0.25) {$e_0$};
    \end{tikzpicture}
    \caption{Switch condition}
    \label{fig:switch_cond}
\end{figure}

An enhanced measure $\widehat{\nu} = (\nu, \sigma)$ on a train track $\tau$ is a measure $\nu\in V(\tau)$ equipped with a signature $\sigma_p \in \{+, -, 0\}$ for each puncture $p$ such that
\[\begin{cases}
    \sigma_p = 0 & \!\begin{array}{l}
        \mbox{if $\tau$ has no edges transverse to $D_p$ or}\\
        \mbox{the measure of an edge transverse to $D_p$ is zero,}
    \end{array}\\
    \sigma_p \in \{+, -\} & \mbox{ otherwise}.
\end{cases}\] 
We denote by $\widehat{V}(\tau)$ the set of enhanced measures on a train track $\tau$.
We can regard it as a cone in the vector space $\bR^{\{\mbox{\scriptsize edges of $\tau$}\}}$ by the injective map
\begin{align*}
    w_\tau: \widehat{V}(\tau) \to \bR^{\{\mbox{\scriptsize edges of $\tau$}\}};\ \widehat{\nu} \mapsto w(\widehat{\nu}),
\end{align*}
where
\[w_\tau(\widehat{\nu})(e) :=
\begin{cases}
    \sigma_p \cdot \nu(e) & \mbox{if $e$ is transverse to $D_p$ for some $p \in P$},\\
    \nu(e) & \mbox{otherwise}
\end{cases}\]
for $\widehat{\nu} = (\nu, \sigma) \in \widehat{V}(\tau)$.
We say that a recurrent train track $\tau$ is \emph{complete} if the dimension of $\mathrm{Im}(w_\tau)$ is $6g-6+3h$,
\footnote{The definition of completeness in \cite{PH} is different from our definition but they are equivalent \cite[Proof of Lemma 3.1.2]{PH}.
} which is the maximal value of the dimensions among the train tracks on $\Sigma$.

An enhanced measure $\widehat{\nu}=(\nu, \sigma_\nu) \in \widehat{
V}(\tau)$ on $\tau$ is said to be rational if $\nu(e) \in \bQ$ for any edge $e$ of $\tau$. Given a rational enhanced measure, we construct a rational $\cX$-lamination $\hL_\nu$ as follows:
\begin{enumerate}
    \item Take an integer $c$ such that $c \nu(e) \in \bZ$ for any $e$.
    \item Take $c\nu(e)$ arcs parallel to $e$ and concatenate them around each vertex of $\tau$. Such an operation is well-defined, thanks to the switch conditions.
    \item Let $L'_\nu$ denote the resulting collection of curves on $\Sigma$. Then we set
    \[ \hL_\nu := (c^{-1} L'_\nu, \sigma_\nu) \in \cX^\uf_\Sigma(\bQ^\trop). \]
\end{enumerate}

\begin{lem}
The map $\widehat{V}(\tau)_\bQ \to \cX^\uf_\Sigma(\bQ^\trop);\ \widehat{\nu} \mapsto \hL_\nu$ is injective.
Here $\widehat{V}(\tau)_\bQ \subset \widehat{V}(\tau)$ denotes the set of rational enhanced measures on $\tau$.
\end{lem}


Extending this procedure, we can construct an (enhanced) measured geodesic lamination $(G_{\widehat{\nu}}, \mu_\nu)$ from an arbitrary enhanced measure $\widehat{\nu} = (\nu, \sigma_\nu) \in \widehat{V}(\tau)$.
From now on, we assume that $\Sigma$ is a punctured surface and fix a complete hyperbolic structure $F$ on $\Sigma^\circ$ with geodesic boundary. 
\begin{enumerate}
    \item Take a \emph{fibered neighborhood} $N_\tau$ of $\tau$ equipped with a retraction $r: N_\tau \searrow \tau$, as shown in \cref{fig:fibered nbd}.
    \item The fibers of $N_\tau$ gives rise to a measured foliation $(\cF_\nu, \mu_\nu)$ on $N_\tau$ 
    by assigning the transverse measure $\mu_\nu$ so that $\mu_\nu(r^{-1}(z)) = \nu(e)$ for each edge $e \subset \tau$ and $z \in \interior e$.
    \item \lq\lq Slit" $N_\tau$ along its singular leaves (that is, the leaves starting at the cusps of $N_\tau$). 
    Let $L_\nu$ be the collection of remaining non-singular curves.  
    \item 
    Straighten each compact leaf of $L_\tau$ into a complete geodesic, and each ideal leaf into a bi-infinite geodesic which spirals in the direction determined by $\sigma_\nu$.
    Let $G_{\widehat{\nu}}$ be the resulting geodesic lamination, which is equipped with the transverse measure $\mu_\nu$ naturally induced by that on $\cF_\nu$. 
\end{enumerate}
We refer the reader to \cite[Construction 1.7.7]{PH} for a detail on each step of the above construction. 

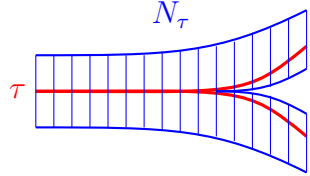
\begin{figure}[h]
    \centering
    \begin{tikzpicture}[scale=1.2]
    \draw[very thick, red] (-1.5,0) .. controls (0.8,0) and (0.8,-0.08) .. (1.5,0.5) node (v1) {};
    \draw[very thick, red] (-1.5,0) .. controls (0.8,0) and (0.8,0.08) .. (1.5,-0.5);
    \draw [blue, thick](-1.5,0.4) .. controls (0,0.4) and (0.5,0.4) .. (1.5,0.9);
    \draw [blue, thick](-1.5,-0.4) .. controls (0,-0.4) and (0.5,-0.4) .. (1.5,-0.9);
    \draw [blue, thick](0.5,0) .. controls (0.95,0) and (1.15,0.05) .. (1.5,0.25);
    \draw [blue, thick](0.5,0) .. controls (0.95,0) and (1.15,-0.05) .. (1.5,-0.25);
    \draw [blue](-1.5,0.4) -- (-1.5,-0.4);
    \draw [blue](-1.3,0.4) -- (-1.3,-0.4);
    \draw [blue](-1.1,0.4) -- (-1.1,-0.4);
    \draw [blue](-0.9,0.4) -- (-0.9,-0.4);
    \draw [blue](0.5,0.5) -- (0.5,-0.5);
    \draw [blue] (0.7,0.55) -- (0.7,0);
    \draw [blue] (0.7,0) -- (0.7,-0.55);
    \draw [blue](1.5,0.25) -- (1.5,0.9);
    \draw [blue](1.5,-0.25) -- (1.5,-0.9);
    \draw [blue](-0.7,0.4) -- (-0.7,-0.4);
    \draw [blue](-0.5,-0.4) -- (-0.5,0.4);
    \draw [blue](-0.3,0.4) -- (-0.3,-0.4);
    \draw [blue](-0.1,-0.44) -- (-0.1,0.44);
    \draw [blue](0.1,0.45) -- (0.1,-0.45);
    \draw [blue](0.3,0.48) -- (0.3,-0.48);
    \draw [blue](0.9,0.65) -- (0.9,0.01);
    \draw [blue](0.9,-0.03) -- (0.9,-0.65);
    \draw [blue](1.1,0.7) -- (1.1,0.05);
    \draw [blue](1.1,-0.05) -- (1.1,-0.7) -- cycle;
    \draw [blue](1.3,0.8) -- (1.3,0.15);
    \draw [blue](1.3,-0.15) -- (1.3,-0.8);
    \node [blue] at (0,0.85) {$N_\tau$};
    \node [red] at (-1.7,0) {$\tau$};
    \end{tikzpicture}
    \caption{A fibered neighborhood of a train track around a vertex}
    \label{fig:fibered nbd}
\end{figure}

\begin{prop}\label{lem:measures_subset_L^x}
The map $\psi_\tau: \widehat{V}(\tau) \to \eML(F)$; $\widehat{\nu} \mapsto (G_{\widehat{\nu}}, \mu_\nu)$ is injective.
\end{prop}
We are going to proof this proposition. We use \emph{enhanced train track measure spectrum}
\begin{align*}
I^\tau_* : \widehat{V}(\tau) \to \bR^P \times \bR_{\geq 0}^{\cS(\Sigma)},\ \widehat{\nu} \mapsto I^\tau_{\widehat{\nu}}
\end{align*}
defined by 
\begin{align*}
I^\tau_{\widehat{\nu}}(\alpha) := \min_{c \in \gamma,\, c \pitchfork \tau} \sum_{z \in c \cap \tau} \nu(e_z)
\end{align*}
for $\alpha \subset \cS(\Sigma)$, where $e_z$ denote an edge of $\tau$ containing $z$, and
\begin{align*}
    I^\tau_{\widehat{\nu}}(p) := \sigma_\nu(p) \cdot \nu(e_p),
\end{align*}
where $e_p$ is an edge of $\tau$ transverse to $\partial_p$.

\begin{lem}\label{lem:I^tau_inj}
The map $I^\tau_*$ is a PL injection.
\end{lem}
\begin{proof}
The piecewise linearity follows from \cite[Proposition 2.1]{Hat}.
For a sign $\sigma=(\sigma_p)_{p \in P} \in \{+,-\}^P$, let us consider the subset
\begin{align*}
    \widehat{V}^{(\sigma)}(\tau):= \{(\nu,\sigma_\nu) \in \widehat{V}(\tau) \mid \sigma_p \sigma_\nu(p) \geq 0 \mbox{ for $p \in P$}\}.
\end{align*}
Then the injectivity of the restriction of the map $I^\tau_*$ to each $\widehat{V}^{(\sigma)}(\tau)$ follows from \cite[Lemma 2.2]{Hat}. 
For two signs $\sigma_1, \sigma_2 \in \{+,-\}^P$, observe the followings:
\begin{itemize}
    \item $I^\tau_*(\mathring{V}^{(\sigma_1)}(\tau)) \cap I^\tau_*(\mathring{V}^{(\sigma_2)}(\tau)) = \emptyset$ if $\sigma_1 \neq \sigma_2$. Here for $i=1,2$, $\mathring{V}^{(\sigma_i)}(\tau)$ denote the set of enhanced measures $(\nu,\sigma_\nu) \in \widehat{V}^{(\sigma_i)}(\tau)$ satisfying $(\sigma_\nu)_p \neq 0$ for all $p \in P$ such that some of edges of $\tau$ incident to $\partial_p$.
    \item For the subtrack $\tau'$ of $\tau$ obtained by eliminating the edges of $\tau$ transverse to $\partial_p$ for each puncture $p$ such that $(\sigma_1)_p \neq (\sigma_2)_p$, we have
    \[\widehat{V}^{(\sigma_1)}(\tau) \cap \widehat{V}^{(\sigma_2)}(\tau) = \{ \widehat{\nu} \mid (\sigma_\nu)_p = 0 \mbox{ if } (\sigma_1)_p \neq (\sigma_2)_p\} = \widehat{V}^{(\sigma_1)}(\tau').\]
\end{itemize}
Then we get the desired assertion.
\end{proof}

\begin{proof}[Proof of \cref{lem:measures_subset_L^x}]
Comparing the definitions, one can verify that the following diagram commutes:
\[
\begin{tikzcd}
\widehat{V}(\tau) \ar[r, "\psi_\tau"] \ar[rd, "I^\tau_*"'] & \eML(F) \ar[d, "I^{\mathrm{ML}}_*"]\\
& \bR^P \times \bR_{\geq 0}^{\cS(\Sigma)}
\end{tikzcd}
\]
Then
\cref{lem:I^tau_inj} implies that the horizontal map is injective.
\end{proof}


We say that $(G, \mu) \in \eML(F)$ is \emph{carried by} a train track $\tau \subset \Sigma$ (or $\tau$ \emph{carries} $(G, \mu)$) if $(G, \mu) \in \psi_\tau(\widehat{V}(\tau))$.
Using the identification $\eML(F) \cong \cX_\Sigma(\bR^\trop)$, we also say that the point $\hL \in \cX_\Sigma(\bR^\trop)$ corresponding to such $(G,\mu)$ is carried by $\tau$.

Let $\widehat{V}^\circ(\tau)$ denotes the subset of $\widehat{V}(\tau)$ consists of strictly positive (enhanced) measures on $\tau$.

\begin{thm}\label{thm:tt_atlas}
The set $\{ (\widehat{V}^\circ(\tau),\ \psi_\tau) \mid \mbox{$\tau \subset \Sigma$ is a complete train track} \}$ gives a PL atlas of $\eML(F)$.
\end{thm}
\begin{proof}
We give a topology $\widehat{V}(\tau)$ by $w: \widehat{V}(\tau) \to \bR^{\scriptsize \{\mbox{edges of $\tau$}\}}$, then $\interior \widehat{V}(\tau) = \widehat{V}^\circ(\tau)$.
Thus by the invariance of domain, we obtain the openness of $\psi_\tau(\widehat{V}^\circ(\tau))$ in the same manner as the proof of \cite[Lemma 3.1.2]{PH}.
\end{proof}

\subsection{Dehn--Thurston coordinates}
We recall here the PL coordinate system on the space of measured laminations/foliations associated with a pants decomposition, called the \emph{Dehn--Thurston coordinates}. Roughly speaking, these coordinates are the tropical analogues of the classically well-known Fenchel--Nielsen coordinates on the \Teich\ space.

\paragraph{\textbf{The case where $\Sigma$ is a punctured surface}}
Let $\Sigma$ be a punctured surface, and fix a complete hyperbolic structure $F$. 
Take a pants decomposition $\cR = \{ C_i \}_{i=1}^{3g-3+h}$ of the compact surface $\Sigma^\circ$. We first recall the \emph{standard train tracks}, following \cite{PH}. 

Fix a tubular neighborhood $A_i$ of each curve $C_i \in \cR$. 
Then each component $P_j$ of $\Sigma^\circ \setminus \bigcup_i \interior A_i$ is a pair of pants embedded in $\Sigma$:
\begin{align*}
\Sigma^\circ \setminus \bigcup_{i=1}^{3g-3+h} \interior A_i = \bigsqcup_{j=1}^{2g-2+h} P_j.
\end{align*}
As a model surfaces, fix an annulus $A$ with a distinguished point on each boundary component and a pair of pants $P$ with a distinguished point on each boundary component. Fix orientation-preserving homeomorphisms
\begin{align*}
    \xi_i &: A \to A_i,\\
    \eta_j &: P \to P_j
\end{align*}
so that $\xi_i(p) = \eta_j(q)$ for each distinguished point $p \in \partial A$ and $q \in \partial P$ whenever $\xi_i(p), \eta_j(q) \in A_i \cap P_j \neq \emptyset$.
We refer to the distinguished points on $A$ and $P$, or their images under the maps $\xi_i$ and $\eta_j$ as the \emph{stops}.
A train track $\tau \subset \Sigma^\circ$ is \emph{standard} with respect to $\cR$ if
\begin{enumerate}
    \item $\xi_i^{-1}(\tau \cap A_i)$ (resp. $\eta^{-1}_j(\tau \cap P_j)$) is either empty or the subtrack of one of the train tracks in \cref{fig:connector_std} (resp. \cref{fig:pants_std}) for each $i = 1, \dots, 3g-3+h$ (resp. $j=1, \dots, 2g-2+h$), and
    \item it is recurrent
    \footnote{The train tracks satisfying (1) but are not recurrent are shown in \cite[Figure 2.6.3]{PH}.}.
\end{enumerate}

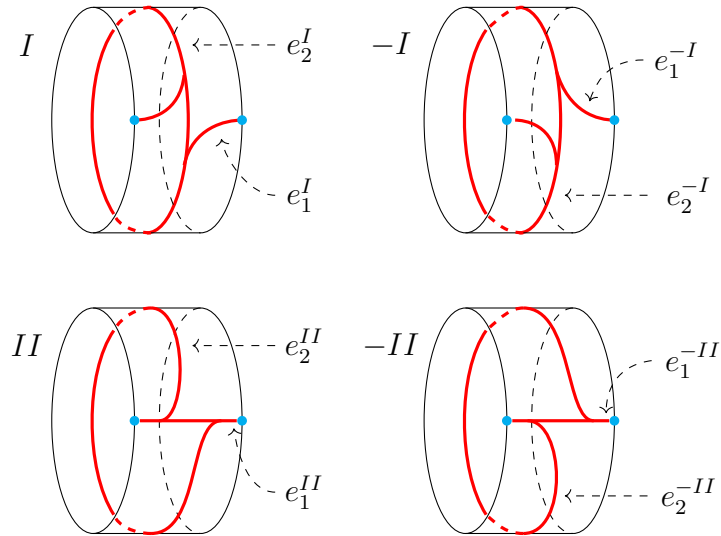
\begin{figure}[h]
    \centering
    \begin{tikzpicture}[xscale=1.1]
    \draw (3.6,-1) coordinate (v2) arc [start angle=-90, end angle=90, x radius=.5cm, y radius=1.5cm];
    \draw [red, very thick] (2.95,-1) arc [start angle=-90, end angle=90, x radius=.5cm, y radius=1.5cm];
    \draw [dashed] (3.6,2) coordinate (v1) arc [start angle=90, end angle=270, x radius=.5cm, y radius=1.5cm];
    \draw  (2.3,0.5) ellipse (0.5 and 1.5);
    \draw (v1) -- (2.3,2);
    \draw (v2) -- (2.3,-1);
    \draw [red, very thick] (2.55,1.7) .. controls (2.2,1.15) and (2.2,-0.15) .. (2.55,-0.75);
    \draw [red, very thick, dashed] (3,2) .. controls (2.85,2) and (2.65,1.9) .. (2.55,1.7);
    \draw [red, very thick, dashed] (3,-1) .. controls (2.8,-1) and (2.6,-0.9) .. (2.55,-0.75);
    \draw [red, very thick](3.4,1.1) .. controls (3.35,0.6) and (3,0.5) .. (2.8,0.5);
    \draw [red, very thick](3.4,-0.1) .. controls (3.5,0.4) and (3.9,0.5) .. (4.1,0.5);
    \node [fill, circle, inner sep = 1.3, cyan] at (4.1,0.5) {};
    \node [fill, circle, inner sep = 1.3, cyan] at (2.8,0.5) {};
    
    \draw (8.1,-1) coordinate (v2) arc [start angle=-90, end angle=90, x radius=.5cm, y radius=1.5cm];
    \draw [red, very thick] (7.45,-1) arc [start angle=-90, end angle=90, x radius=.5cm, y radius=1.5cm];
    \draw [dashed] (8.1,2) coordinate (v1) arc [start angle=90, end angle=270, x radius=.5cm, y radius=1.5cm];
    \draw  (6.8,0.5) ellipse (0.5 and 1.5);
    \draw (v1) -- (6.8,2);
    \draw (v2) -- (6.8,-1);
    \draw [red, very thick] (7.05,1.7) .. controls (6.7,1.15) and (6.7,-0.15) .. (7.05,-0.75);
    \draw [red, very thick, dashed] (7.5,2) .. controls (7.35,2) and (7.25,1.9) .. (7.05,1.7);
    \draw [red, very thick, dashed] (7.5,-1) .. controls (7.3,-1) and (7.2,-0.9) .. (7.05,-0.75);
    \draw [red, very thick](7.9,-0.1) .. controls (7.9,0.45) and (7.5,0.5) .. (7.4,0.5);
    \draw [red, very thick](7.9,1.15) .. controls (8.05,0.6) and (8.4,0.5) .. (8.6,0.5);
    \node [fill, circle, inner sep = 1.3, cyan] at (8.6,0.5) {};
    \node [fill, circle, inner sep = 1.3, cyan] at (7.3,0.5) {};

    \draw (3.6,-5) coordinate (v2) arc [start angle=-90, end angle=90, x radius=.5cm, y radius=1.5cm];
    \draw [dashed] (3.6,-2) coordinate (v1) arc [start angle=90, end angle=270, x radius=.5cm, y radius=1.5cm];
    \draw  (2.3,-3.5) ellipse (0.5 and 1.5);
    \draw (v1) -- (2.3,-2);
    \draw (v2) -- (2.3,-5);
    \draw [red, very thick] (2.55,-2.3) .. controls (2.2,-2.85) and (2.2,-4.15) .. (2.55,-4.75);
    \draw [red, very thick, dashed] (3,-2) .. controls (2.85,-2) and (2.65,-2.1) .. (2.55,-2.3);
    \draw [red, very thick, dashed] (3,-5) .. controls (2.8,-5) and (2.6,-4.9) .. (2.55,-4.75);
    \node [fill, circle, inner sep = 1.3, cyan] (v6) at (4.1,-3.5) {};
    \node [fill, circle, inner sep = 1.3, cyan] (v5) at (2.8,-3.5) {};
    
    \draw (8.1,-5) coordinate (v2) arc [start angle=-90, end angle=90, x radius=.5cm, y radius=1.5cm];
    \draw [dashed] (8.1,-2) coordinate (v1) arc [start angle=90, end angle=270, x radius=.5cm, y radius=1.5cm];
    \draw  (6.8,-3.5) ellipse (0.5 and 1.5);
    \draw (v1) -- (6.8,-2);
    \draw (v2) -- (6.8,-5);
    \draw [red, very thick] (7.05,-2.3) .. controls (6.7,-2.85) and (6.7,-4.15) .. (7.05,-4.75);
    \draw [red, very thick, dashed] (7.5,-2) .. controls (7.35,-2) and (7.25,-2.1) .. (7.05,-2.3);
    \draw [red, very thick, dashed] (7.5,-5) .. controls (7.3,-5) and (7.2,-4.9) .. (7.05,-4.75);
    \node [fill, circle, inner sep = 1.3, cyan] (v4) at (8.6,-3.5) {};
    \node [fill, circle, inner sep = 1.3, cyan] (v3) at (7.3,-3.5) {};
    
    \draw [red, very thick](v3) -- (v4);
    \draw [red, very thick](v5) -- (v6);
    \draw [red, very thick](3,-2) .. controls (3.45,-2) and (3.45,-3.5) .. (3.1,-3.5);
    \draw [red, very thick](3,-5) .. controls (3.6,-5) and (3.45,-3.5) .. (3.85,-3.5);
    \draw [red, very thick](7.5,-2) .. controls (8.1,-2) and (8,-3.5) .. (8.35,-3.5);
    \draw [red, very thick](7.55,-3.5) .. controls (8,-3.5) and (8.05,-5) .. (7.5,-5);
    
    \draw [->, dashed] (4.5,-0.5) .. controls (4.2,-0.5) and (3.9,-0.05) .. (3.85,0.25);
    \node at (4.8,-0.5) {$e^I_1$};
    \draw [->, dashed] (4.5,-4.45) .. controls (4.2,-4.45) and (4.05,-4) .. (4,-3.6);
    \node at (4.85,-4.45) {$e^{II}_1$};
    \draw [->, dashed](8.95,1.3) .. controls (8.6,1.3) and (8.4,1.1) .. (8.3,0.75);
    \draw [->, dashed](9.05,-2.7) .. controls (8.7,-2.75) and (8.5,-3) .. (8.5,-3.4);
    \node at (9.35,1.3) {$e_1^{-I}$};
    \node at (9.55,-2.7) {$e_1^{-II}$};
    \draw [->, dashed](4.5,1.5) -- (3.5,1.5);
    \draw [->, dashed](4.5,-2.5) -- (3.5,-2.5);
    \draw [->, dashed](9,-0.5) -- (8,-0.5);
    \draw [->, dashed](9,-4.5) -- (8,-4.5);
    \node at (4.8,1.5) {$e_2^{I}$};
    \node at (4.85,-2.5) {$e_2^{II}$};
    \node at (9.5,-0.5) {$e_2^{-I}$};
    \node at (9.5,-4.5) {$e_2^{-II}$};
    \node at (1.5,1.5) {$I$};
    \node at (1.5,-2.5) {$II$};
    \node at (5.9,1.5) {$-I$};
    \node at (5.9,-2.5) {$-II$};
    \end{tikzpicture}
    \caption{Standard tracks in an annulus}
    \label{fig:connector_std}
\end{figure}

We denote by $ST_\cR$ the set of standard train tracks with respect to a pants decomposition $\cR$.

We call the train tracks shown in \cref{fig:connector_std} the standard track of type $I$, $-I$, $II$ and $-II$, respectively.

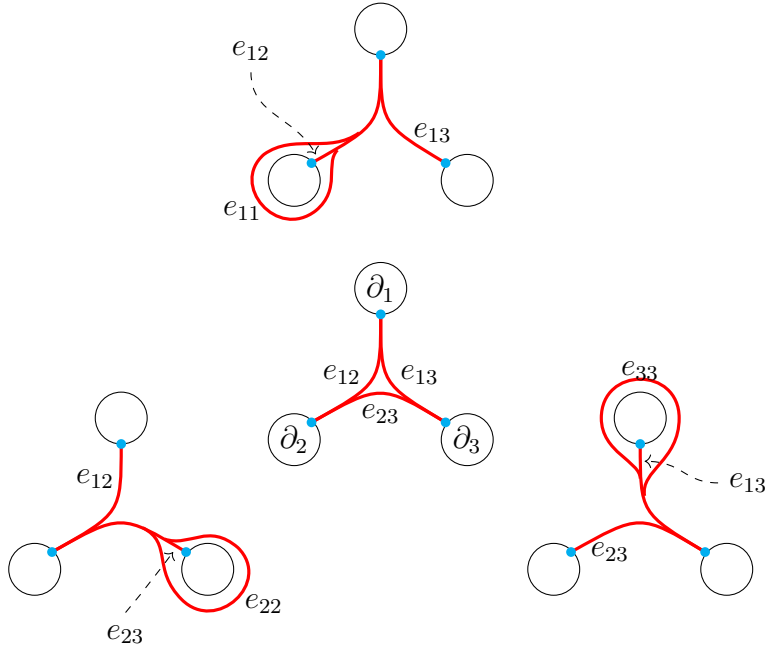
\begin{figure}[h]
    \centering
    \begin{tikzpicture}[scale=1.15]
    \draw  (0,2.5) ellipse (0.3 and 0.3);
    \draw  (-1,0.75) ellipse (0.3 and 0.3);
    \draw  (1,0.75) ellipse (0.3 and 0.3);
    \draw [red, very thick] (0,2.2) .. controls (0,1.4) and (0,1.4) .. (-0.8,0.95);
    \draw [red, very thick](-0.8,0.95) .. controls (0,1.4) and (0,1.4) .. (0.75,0.95);
    \draw [red, very thick](0.75,0.95) .. controls (0,1.4) and (0,1.4) .. (0,2.2);
    \node [fill, circle, inner sep=1.3, cyan] at (0,2.2) {};
    \node [fill, circle, inner sep=1.3, cyan] at (-0.8,0.95) {};
    \node [fill, circle, inner sep=1.3, cyan] at (0.75,0.95) {};
    
    \draw  (3,1) ellipse (0.3 and 0.3);
    \draw  (2,-0.75) ellipse (0.3 and 0.3);
    \draw  (4,-0.75) ellipse (0.3 and 0.3);
    \draw [red, very thick](2.2,-0.55) .. controls (3,-0.1) and (3,-0.1) .. (3.75,-0.55);
    \draw [red, very thick](3.75,-0.55) .. controls (3,-0.1) and (3,-0.1) .. (3,0.7);
    \node [fill, circle, inner sep=1.3, cyan] at (3,0.7) {};
    \node [fill, circle, inner sep=1.3, cyan] at (2.2,-0.55) {};
    \node [fill, circle, inner sep=1.3, cyan] at (3.75,-0.55) {};
    
    \draw  (-3,1) ellipse (0.3 and 0.3);
    \draw  (-4,-0.75) ellipse (0.3 and 0.3);
    \draw  (-2,-0.75) ellipse (0.3 and 0.3);
    \draw [red, very thick] (-3,0.7) .. controls (-3,-0.1) and (-3,-0.1) .. (-3.8,-0.55);
    \draw [red, very thick](-3.8,-0.55) .. controls (-3,-0.1) and (-3,-0.1) .. (-2.25,-0.55);
    \draw [red, very thick](-2.75,-0.27) .. controls (-2.4,-0.45) and (-2.65,-0.75) .. (-2.3,-1.1) .. controls (-1.9,-1.5) and (-1.25,-0.9) .. (-1.65,-0.5) .. controls (-2,-0.2) and (-2.2,-0.5) .. (-2.5,-0.4);
    \node [fill, circle, inner sep=1.3, cyan] at (-3,0.7) {};
    \node [fill, circle, inner sep=1.3, cyan] at (-3.8,-0.55) {};
    \node [fill, circle, inner sep=1.3, cyan] at (-2.25,-0.55) {};
    
    \draw  (0,5.5) ellipse (0.3 and 0.3);
    \draw  (-1,3.75) ellipse (0.3 and 0.3);
    \draw  (1,3.75) ellipse (0.3 and 0.3);
    \draw [red, very thick] (0,5.2) .. controls (0,4.4) and (0,4.4) .. (-0.8,3.95);
    \draw [red, very thick](0.75,3.95) .. controls (0,4.4) and (0,4.4) .. (0,5.2);
    \draw [red, very thick](-0.25,4.3) .. controls (-0.6,4.05) and (-1,4.3) .. (-1.35,4.05) .. controls (-1.75,3.7) and (-1.2,3.05) .. (-0.75,3.4) .. controls (-0.45,3.65) and (-0.65,3.95) .. (-0.5,4.1);
    \node [fill, circle, inner sep=1.3, cyan] at (0,5.2) {};
    \node [fill, circle, inner sep=1.3, cyan] at (-0.8,3.95) {};
    \node [fill, circle, inner sep=1.3, cyan] at (0.75,3.95) {};
    \draw [red, very thick](3.05,0.1) .. controls (3,0.4) and (3.45,0.5) .. (3.45,1) .. controls (3.45,1.6) and (2.55,1.6) .. (2.55,1) .. controls (2.55,0.6) and (3,0.5) .. (3,0.3);
    \draw [->, dashed](-1.5,5) .. controls (-1.5,4.5) and (-1,4.5) .. (-0.75,4.05);
    \node at (0,2.5) {$\partial_1$};
    \node at (-1,0.75) {$\partial_2$};
    \node at (1,0.75) {$\partial_3$};
    \node at (-1.5,5.25) {$e_{12}$};
    \node at (-1.6,3.4) {$e_{11}$};
    \node at (0.6,4.3) {$e_{13}$};
    \node at (-0.45,1.5) {$e_{12}$};
    \node at (0,1.05) {$e_{23}$};
    \node at (0.45,1.5) {$e_{13}$};
    \node at (-1.35,-1.15) {$e_{22}$};
    \node at (-2.95,-1.5) {$e_{23}$};
    \node at (-3.3,0.3) {$e_{12}$};
    \draw [->, dashed](-2.95,-1.25) .. controls (-2.7,-0.95) and (-2.55,-0.75) .. (-2.4,-0.55);
    \node at (3,1.55) {$e_{33}$};
    \node at (2.65,-0.55) {$e_{23}$};
    \node at (4.25,0.25) {$e_{13}$};
    \draw [->, dashed](3.9,0.25) .. controls (3.5,0.25) and (3.5,0.5) .. (3.05,0.5);
    \end{tikzpicture}
    \caption{Standard tracks in a pair of pants}
    \label{fig:pants_std}
\end{figure}

\begin{thm}\label{thm:tt_cone_decomp}
Fix a pants deomposition $\cR$ of $\Sigma^\circ$.
Then the collection
\begin{align*}
\fF_\cR = \{ \psi_\tau(\widehat{V}(\tau)) \mid \tau \in ST_\cR \}
\end{align*} 
of cones gives a complete fan in $\eML(F)$.
That is, 
\begin{itemize}
    \item $\widehat{V}(\tau)$ is a polyhedral cone,
    \item $\bigcup_{\tau \in ST_\cR} \psi_\tau(\widehat{V}(\tau)) = \eML(F)$,
    \item for each face $F$ of $\psi_\tau(\widehat{V}(\tau)) \in \fF_\cR$, there is a standard track $\tau_F \subset \tau$ such that $\psi_{\tau_F}(\widehat{V}(\tau_F)) = F$,
    \item if $\psi_{\tau_1}(\widehat{V}(\tau_1)) \cap \psi_{\tau_2}(\widehat{V}(\tau_2)) \neq 0$ for $\tau_1 \neq \tau_2 \in ST_\cR$, then it is a common face of the cones $\psi_{\tau_1}(\widehat{V}(\tau_1))$ and $\psi_{\tau_2}(\widehat{V}(\tau_2))$.
\end{itemize}
\end{thm}
\begin{proof}
For a fixed sign $\sigma \in \{+,-\}^P$, $\{\psi_\tau(\widehat{V}^{(\sigma)}(\tau))\}$ gives a complete fan in $\eML^{(\sigma)}(F)$ by \cite[Proposition 2.3]{Hat}. 
Then they glue together as \cref{lem:I^tau_inj}, and we get the desired fan in the entire space.
\end{proof}

Now we come to the definition of the Dehn--Thurston coordinates. 
Let $\tau$ be a standard train track with respect to some pants decomposition $\cR$, and let $\widehat{\nu} = (\nu, \sigma) \in \widehat{V}(\tau)$. Recall the branches $e_k^X$ for $k=1,2$ and $X=I,II,-I,-II$ from \cref{fig:pants_std}. 
For each $i \in \{1, \dots, 3g-3+h\}$, define
\begin{align*}
    m_i(\widehat{\nu}) &:=
    \begin{cases}
        \nu(\xi_i(e_1^{X})) & \mbox{if $\xi^{-1}_i(A_i \cap \tau)$ is of type $X = I$ or $II$,}\\
        -\nu(\xi_i(e_1^{X})) & \mbox{if $\xi^{-1}_i(A_i \cap \tau)$ is of type $X = -I$ or $-II$,}\\
    \end{cases}\\
    t_i(\widehat{\nu}) &:= \nu(\xi^{-1}_i(e_2^{X})) \quad \mbox{for each type $X = I, -I, II$ and $-II$}.
\end{align*}
Also, for each $p \in P$, define
\begin{align*}
    m_p(\widehat{\nu}) := \sigma \cdot \nu(e_p).
\end{align*}
Here $e_p \subset \tau$ is an edge transverse to the boundary component $\partial_p$.

\begin{thm}[Dehn--Thurston coordinates]\label{thm:DT_coord}
Let $\Sigma$ be a punctured surface, and $\cR$ a pants decomposition of $\Sigma^\circ$.
Then the map
\begin{align*}
    \Pi_\cR: \eML(F) \to (\bR^2/ (\bZ/2))^{3g-3+h} \times \bR^P \cong \bR^{6g-6+3h}
\end{align*}
given by
\begin{align*}
    \Pi_\cR(G, \mu) :=  \big( (m_i(\widehat{\nu}), t_i(\widehat{\nu}))_{i=1}^{3g-3+h}, (m_p(\widehat{\nu}))_{p \in P} \big)
\end{align*}
is well-defined and it gives a global coordinate system on $\eML(F)$.
Here the action of $\bZ/2$ on $\bR^2$ is the antipodal map, and $\widehat{\nu} = \psi_\tau^{-1}(G, \mu)$ if $(G, \mu)$ is carried by a standard train track $\tau \in \fC_\cR$.
\end{thm}

\begin{proof}

For the case of compact laminations $\ML_0(F)$, the continuous inverse map $\nu_\cR^u$ of the restriction
\begin{align*}
    \Pi_\cR^u: \ML_0(F) \to (\bR^2/ (\bZ/2))^{3g-3+h}\ ,\quad 
    (G,\mu) \mapsto (m_i(\widehat{\nu}), t_i(\widehat{\nu}))_{i=1}^{3g-3+h}
\end{align*}
of $\Pi_\cR$ 
is constructed in the proof of \cite[Theorem 3.1.1]{PH} based on a fan structure of $\ML_0(F)$, which is the restriction of the fan $\fF_\cR$ given in \cref{thm:tt_cone_decomp}.
For instance, the measure $\nu= \nu^u_\cR((m_i, t_i)_i)$ on the edges which lies in a pair of pants $P_j$ is given by
\begin{align*}
    2\nu(e^j_{11}) &= [m_1 - m_2 - m_3]_+,\\
    2\nu(e^j_{12}) &= [m_1 + m_2 - m_3]_+,\\
    2\nu(e^j_{13}) &= [m_1 - m_2 + m_3]_+,\\
    2\nu(e^j_{22}) &= [-m_1 + m_2 - m_3]_+,\\
    2\nu(e^j_{23}) &= [-m_1 + m_2 + m_3]_+,\\
    2\nu(e^j_{33}) &= [-m_1 - m_2 + m_3]_+,
\end{align*}
where $e_{ab}^j := \eta_j(e_{ab})$ and
$m_a := m_i$ if $\eta_j(\partial_a) = C_i$ for $a,b \in {1,2,3}$.
(See \cref{fig:pants_std}.)

In our case, the only difference is the existence of edges which are transverse to the boundary components.
So one can naturally extend $\nu_\cR^u$ to $\nu_\cR: \bR^{6g-6+3h} \to \eML(F)$ by setting $m_a = m_p$ if $\eta_j(\partial_a) = \partial_p$.
The measures on the edges which lies in the annuli $A_i$ are given in the same manner as in the case of $\ML_0(\Sigma)$.
It gives the continuous inverse map of $\Pi_\cR$.
\end{proof}

Instead of an explanation of a general method to obtain the standard train track which carries a given measured lamination, we give an example which 
looks non-trivial so that one can guess the general construction.
\begin{ex}
We choose the surface $\Sigma$ to be a sphere with 4 punctures, and consider the lamination $l^-_\tri \in \cX_\Sigma(\bR^\trop)$ shown in the left of \cref{fig:tt_ideal_tri}, which is defined in \cref{ex:3bdries+1pct_sph}.
(See \cite[Lemma 4.4]{IK20}.)
We are going to find the standard train track $\tau$ with respect to a pants decomposition $\cR = \{C\}$ drawn in the left of \cref{fig:tt_ideal_tri} (right blue circle) which carries $l^-_\tri$, where the existence is guaranteed by \cref{thm:tt_cone_decomp}.
For this, first take a tubular neighborhood $A_C$ of $C$ and fix a point (that is, a stop) on each of its boundary components.
Deform the lamination $l^-_\tri$ so that 
\begin{itemize}
    \item transverse to $\partial A_C$,
    \item intersects to $A_C$ only \lq\lq around" the stops, and
    \item the arcs in the complement $\Sigma\setminus A_C$ winds around $C$ minimally.
\end{itemize}
The deformed lamination is shown in the middle of \cref{fig:tt_ideal_tri}.
Finally, having one of the standard tracks in mind, we appropriately pinch the lamination and obtain the desired standard train track $\tau$.
The corresponding measure of $\tau$ is given by counting the segments of $l^-_\tri$ parallel to an edge of $\tau$.
Further giving the same signs at punctures as those of the lamination $l^-_\tri$, we obtain an enhanced measure $\widehat{\nu}$ on $\tau$ such that $l^-_\tri= \psi_\tau(\widehat{\nu})$.

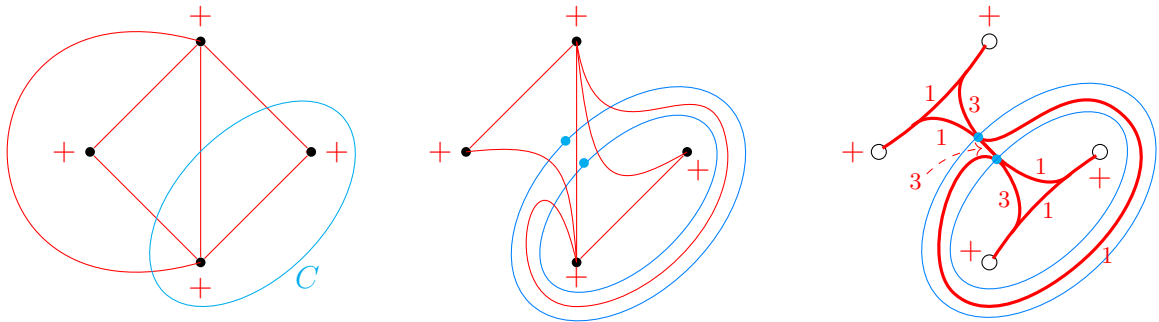
\begin{figure}[h]
    \centering
    \hspace{-1cm}
    \begin{tikzpicture}[scale=.98]
    \node [fill, circle, inner sep=1.3] (v1) at (-0.1,1.5) {};
    \node [fill, circle, inner sep=1.3] (v2) at (-1.6,0) {};
    \node [fill, circle, inner sep=1.3] (v3) at (-0.1,-1.5) {};
    \node [fill, circle, inner sep=1.3] (v4) at (1.4,0) {};
    \draw [red](v1) -- (v2) -- (v3) -- (v4) -- (v4) -- (v1) -- (v3);
    \draw [red](-0.1,1.5) .. controls (-3.6,2.5) and (-3.6,-2.5) .. (-0.1,-1.5);
    \draw [cyan, rotate=45] (-0.0712,-0.9244) ellipse (1.7 and 1);
    \node [cyan] at (1.35,-1.7) {$C$};
    
    \node [fill, circle, inner sep=1.3] (v1) at (5,1.5) {};
    \node [fill, circle, inner sep=1.3] (v2) at (3.5,0) {};
    \node [fill, circle, inner sep=1.3] (v3) at (5,-1.5) {};
    \node [fill, circle, inner sep=1.3] (v4) at (6.5,0) {};
    \draw [myblue, rotate=45] (3.5351,-4.5307) ellipse (1.5 and .8);
    \draw [myblue, rotate=45] (3.5351,-4.5307) ellipse (1.9 and 1.2);
    \draw [red](v1) -- (v3);
    \draw [red](v1) -- (v2);
    \draw [red](v3) -- (v4);
    \draw [red](6.5,0) .. controls (5.45,-0.6) and (5.2,-0.5) .. (5,1.5);
    \draw [red](5,-1.5) .. controls (4.9,-0.15) and (4.8,0.15) .. (3.5,0);
    \draw [red](5,-1.5);
    \draw [red](5,-1.5) .. controls (4.5,0.7) and (3.75,-2.1) .. (5,-2.1) .. controls (5.95,-2.1) and (7.1,-1) .. (7.05,0) .. controls (6.9,1.6) and (5.3,-0.55) .. (5,1.5);
    \node [fill, circle, inner sep=1.3, cyan] at (4.85,0.15) {};
    \node [fill, circle, inner sep=1.3, cyan] at (5.1,-0.15) {};
    
    \node [draw, circle, inner sep=2] at (10.6,1.5) {};
    \node [draw, circle, inner sep=2] at (9.1,0) {};
    \node [draw, circle, inner sep=2] at (10.6,-1.5) {};
    \node [draw, circle, inner sep=2] at (12.1,0) {};
    \draw [red, very thick](10.65,-1.45) .. controls (11,-0.9) and (11.5,-0.4) .. (12.05,-0.05);
    \draw [red, very thick](9.15,0.05) .. controls (9.7,0.4) and (10.2,0.9) .. (10.55,1.45);
    \draw [red, very thick](10.3,1.1) .. controls (9.8,0.55) and (11.2,-0.8) .. (11.7,-0.3);
    \draw [red, very thick](9.55,0.35) .. controls (10.1,0.85) and (11.4,-0.55) .. (10.9,-1.1);
    \draw [red, very thick](10.4,0.25) .. controls (10.75,-0.2) and (11.6,1.15) .. (12.35,0.55) .. controls (13,0.05) and (12.5,-0.85) .. (12,-1.35) .. controls (11.5,-1.85) and (10.6,-2.4) .. (10.1,-1.9) .. controls (9.6,-1.4) and (10.2,0.45) .. (10.8,-0.2);
    \node [fill, circle, inner sep=1.3, cyan] at (10.45,0.2) {};
    \node [fill, circle, inner sep=1.3, cyan] at (10.7,-0.1) {};
    \draw [myblue, rotate=45] (7.5052,-8.4344) ellipse (1.5 and .8);
    \draw [myblue, rotate=45] (7.5052,-8.4344) ellipse (1.9 and 1.2);
    \node [red] at (-0.1,1.85) {$+$};
    \node [red] at (-1.95,0) {$+$};
    \node [red] at (-0.1,-1.85) {$+$};
    \node [red] at (1.75,0) {$+$};
    \node [red] at (5,1.85) {$+$};
    \node [red] at (3.15,0) {$+$};
    \node [red] at (5,-1.7) {$+$};
    \node [red] at (6.65,-0.25) {$+$};
    \node [red] at (10.6,1.85) {$+$};
    \node [red] at (8.75,0) {$+$};
    \node [red] at (10.35,-1.35) {$+$};
    \node [red] at (12.1,-0.35) {$+$};
    \node [red] at (9.8,0.8) {\scriptsize 1};
    \node [red] at (9.95,0.2) {\scriptsize 1};
    \node [red] at (10.4,0.65) {\scriptsize 3};
    \node [red] at (10.8,-0.65) {\scriptsize 3};
    \node [red] at (11.4,-0.8) {\scriptsize 1};
    \node [red] at (11.3,-0.2) {\scriptsize 1};
    \node [red] at (12.2,-1.4) {\scriptsize 1};
    \draw [->, red, dashed](9.75,-0.3) .. controls (9.95,-0.1) and (10.15,0) .. (10.5,0.05);
    \node [red] at (9.6,-0.4) {\scriptsize 3};
    \end{tikzpicture}
    \caption{An example of a train track carrying both a pants decomposition and a collection of ideal arcs. 
    Left and Center: The given lamination $l^-_\tri$ and its deformation. Right: The train track which carries $l^-_\tri$ and the enhanced measure associated to $l^-_\tri$.}
    \label{fig:tt_ideal_tri}
\end{figure}

\end{ex}

\begin{rmk}\label{rem:connector_folded}
Our choice of the four standard tracks in an annulus  \cref{fig:connector_std} is according to \cite[Figure 2.6.1]{PH}. In fact, the train tracks of types $I$ and $II$ (resp. $-I$ and $-II$) can be obtained by the \emph{splittings} in two directions from a common train track $\tau^+$ (resp. $\tau^-$) shown in the left (resp. right) of \cref{fig:connector_ver_Hat}. For the combinatorial operations on train tracks including the splittings, see \cite[Section 2.1]{PH}. 
It follows that
\begin{align*}
    \widehat{V}(\tau^\epsilon) = \widehat{V}(\tau^{\epsilon I}) \cup \widehat{V}(\tau^{\epsilon II})
\end{align*}
for $\epsilon = +, -$, where $\tau^{\epsilon X}$ is the standard track of type $\epsilon X$ in an annulus for $\epsilon = +, -$, $X=I, II$. 
Note that if we use a train track $\tau \subset \Sigma^\circ$ such that $\xi_i^{-1}(\tau \cap A_i)$ is either $\tau^+$ or $\tau^-$ instead of those in \cref{fig:connector_std}, then it carries the curve $C_i$, that is, $C_i \in \widehat{V}(\tau)$ for $i = 1, \dots, 3g-3+h$. 
Sometimes, a train track such that $\xi_i^{-1}(\tau \cap A_i)$ is one of the train tracks either in \cref{fig:connector_std} or \cref{fig:connector_ver_Hat} for each $i=1,\dots,3g-3+h$ is also called a standard track.

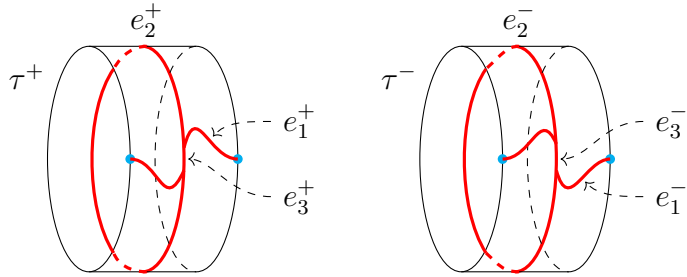
\begin{figure}[h]
    \centering
    \begin{tikzpicture}[xscale=1.1]
    \draw (3.6,-1) coordinate (v2) arc [start angle=-90, end angle=90, x radius=.5cm, y radius=1.5cm];
    \draw [red, very thick] (2.95,-1) arc [start angle=-90, end angle=90, x radius=.5cm, y radius=1.5cm];
    \draw [dashed] (3.6,2) coordinate (v1) arc [start angle=90, end angle=270, x radius=.5cm, y radius=1.5cm];
    \draw  (2.3,0.5) ellipse (0.5 and 1.5);
    \draw (v1) -- (2.3,2);
    \draw (v2) -- (2.3,-1);
    \draw [red, very thick] (2.6,1.7) .. controls (2.25,1.15) and (2.25,-0.15) .. (2.6,-0.75);
    \draw [red, very thick, dashed] (3,2) .. controls (2.85,2) and (2.7,1.9) .. (2.6,1.7);
    \draw [red, very thick, dashed] (3,-1) .. controls (2.8,-1) and (2.65,-0.9) .. (2.6,-0.75);
    
    \node [fill, circle, inner sep = 1.3, cyan] at (4.1,0.5) {};
    \node [fill, circle, inner sep = 1.3, cyan] at (2.8,0.5) {};
    
    \draw (8.1,-1) coordinate (v2) arc [start angle=-90, end angle=90, x radius=.5cm, y radius=1.5cm];
    \draw [red, very thick] (7.45,-1) arc [start angle=-90, end angle=90, x radius=.5cm, y radius=1.5cm];
    \draw [dashed] (8.1,2) coordinate (v1) arc [start angle=90, end angle=270, x radius=.5cm, y radius=1.5cm];
    \draw  (6.8,0.5) ellipse (0.5 and 1.5);
    \draw (v1) -- (6.8,2);
    \draw (v2) -- (6.8,-1);
    \draw [red, very thick] (7.1,1.7) .. controls (6.75,1.15) and (6.75,-0.15) .. (7.1,-0.75);
    \draw [red, very thick, dashed] (7.5,2) .. controls (7.35,2) and (7.3,1.9) .. (7.1,1.7);
    \draw [red, very thick, dashed] (7.5,-1) .. controls (7.3,-1) and (7.25,-0.9) .. (7.1,-0.75);
    
    \node [fill, circle, inner sep = 1.3, cyan] at (8.6,0.5) {};
    \node [fill, circle, inner sep = 1.3, cyan] at (7.3,0.5) {};
    
    \draw [very thick, red](3.45,0.35) .. controls (3.2,-0.25) and (3.15,0.5) .. (2.8,0.5);
    \draw [very thick, red](3.45,0.65) .. controls (3.6,1.3) and (3.75,0.5) .. (4.1,0.5);
    \draw [very thick, red](7.95,0.35) .. controls (8.1,-0.25) and (8.3,0.5) .. (8.6,0.5);
    \draw [very thick, red](7.95,0.7) .. controls (7.75,1.2) and (7.65,0.5) .. (7.3,0.5);
    \node at (1.55,1.6) {$\tau^+$};
    \node at (6.05,1.6) {$\tau^-$};
    \draw [dashed, ->](4.5,1) .. controls (4.2,1) and (3.95,1) .. (3.8,0.8);
    \node at (4.85,1) {$e^+_1$};
    \draw [dashed, ->](9,0) .. controls (8.75,0) and (8.45,0) .. (8.3,0.2);
    \node at (9.35,0) {$e^-_1$};
    \node at (3,2.35) {$e_2^+$};
    \node at (7.5,2.35) {$e_2^-$};
    \draw [dashed, ->](9,1) .. controls (8.5,1) and (8.35,0.5) .. (8,0.5);
    \draw [dashed, ->](4.5,0) .. controls (4,0) and (3.8,0.5) .. (3.5,0.5);
    \node at (4.85,0) {$e_3^+$};
    \node at (9.35,1) {$e_3^-$};
    \end{tikzpicture}
    \caption{The folded standard tracks $\tau^+$ and $\tau^-$ on an annulus.}
    \label{fig:connector_ver_Hat}
\end{figure}

\end{rmk}

\paragraph{\textbf{The general case}}

Now we allow the surface $\Sigma$ to have boundary components. 
We will mean by a pants decomposition of $\Sigma$ the union of a pants decomposition $\cR$ of $\bar{\Sigma}$ (being naturally regarded as a collection of curves in $\Sigma$) and the multicurve $\cR_\partial$ (see \cref{subsec:dyn_gen_X}).
Take an $\iota$-invariant pants decomposition $\widetilde{D\cR}$ of the punctured double $D^\times \Sigma$ such that $\cR \subset \widetilde{D\cR} \cap \Sigma$.
Then we have
\begin{align*}
    \widetilde{D\cR} = D\cR \sqcup \cR(\partial_1) \sqcup \cdots \sqcup \cR(\partial_b),
\end{align*}
where $D\cR$ is the double of $\cR$, and $\cR(\partial_k)$ is an $\iota$-invariant pants decomposition of the punctured sphere $D^\times \bD(\partial_k)$ for $k = 1, \dots, b$.
In order to extend \cref{thm:tt_cone_decomp} to the case of a marked surface, we need to consider $\iota$-invariant \lq\lq standard" train tracks in the punctured spheres $D^\times \bD(\partial_k)$.

Let $\bD$ be a once-punctured disk, and $D^\times \bD$ its double. Consider an $\iota$-invariant pants decomposition $\cR$ of $D^\times \bD$. 
We say that a curve $C_o \in \cR$ is \emph{outermost} if $C_o$ is one of the boundary components of a pair of pants whose other boundary components are given by the two copies of the puncture of $\bD$.
Such an outermost curve uniquely exists for any $\iota$-invariant pants decomposition of the double of a punctured disk. For example, see \cref{fig:std_5_sp_pts}. 
Take tubular neighborhoods $A_i$ of $C_i \in \cR$ and denote by $P_j$ the connected components of $D^\times \bD \setminus \bigcup_i A_i$. 
Fix orientation-preserving homeomorphisms $\xi_i: A \to A_i$ and $\eta_j: P \to P_j$ so that their stops are $\iota$-invariant.
The \emph{$\iota$-invariant standard tracks} in an annulus, a pair of pants without outermost curve, and a pair of pants with an outermost curve are shown in \cref{fig:std_ann_iota_inv}, \cref{fig:std_pants_iota_inv} and \cref{fig:std_pants_iota_inv_out}, respectively.

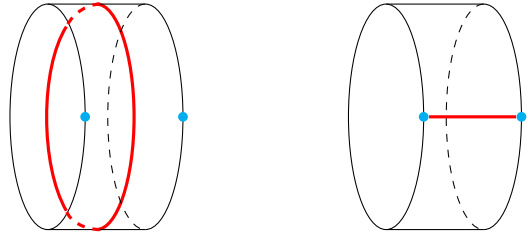
\begin{figure}[h]
    \centering
    \begin{tikzpicture}
    \draw (3.6,-1) coordinate (v2) arc [start angle=-90, end angle=90, x radius=.5cm, y radius=1.5cm];
    \draw [red, very thick] (2.95,-1) arc [start angle=-90, end angle=90, x radius=.5cm, y radius=1.5cm];
    \draw [dashed] (3.6,2) coordinate (v1) arc [start angle=90, end angle=270, x radius=.5cm, y radius=1.5cm];
    \draw  (2.3,0.5) ellipse (0.5 and 1.5);
    \draw (v1) -- (2.3,2);
    \draw (v2) -- (2.3,-1);
    \draw [red, very thick] (2.55,1.7) .. controls (2.2,1.15) and (2.2,-0.15) .. (2.55,-0.75);
    \draw [red, very thick, dashed] (3,2) .. controls (2.85,2) and (2.65,1.9) .. (2.55,1.7);
    \draw [red, very thick, dashed] (3,-1) .. controls (2.8,-1) and (2.6,-0.9) .. (2.55,-0.75);
    
    \node [fill, circle, inner sep = 1.3, cyan] at (4.1,0.5) {};
    \node [fill, circle, inner sep = 1.3, cyan] at (2.8,0.5) {};
    
    \draw (8.1,-1) coordinate (v2) arc [start angle=-90, end angle=90, x radius=.5cm, y radius=1.5cm];
    \draw [dashed] (8.1,2) coordinate (v1) arc [start angle=90, end angle=270, x radius=.5cm, y radius=1.5cm];
    \draw  (6.8,0.5) ellipse (0.5 and 1.5);
    \draw (v1) -- (6.8,2);
    \draw (v2) -- (6.8,-1);
    \node [fill, circle, inner sep = 1.3, cyan] (v4) at (8.6,0.5) {};
    \node [fill, circle, inner sep = 1.3, cyan] (v3) at (7.3,0.5) {};
\draw [very thick, red](v3) -- (v4);	
\end{tikzpicture}
    \caption{$\iota$-invariant standard tracks in an annulus}
    \label{fig:std_ann_iota_inv}
\end{figure}

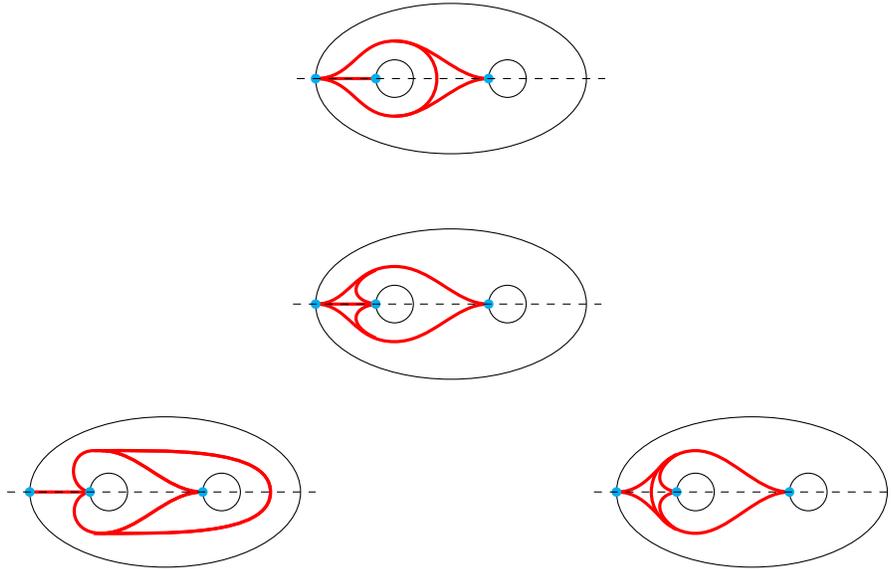
\begin{figure}[h]
    \centering\begin{tikzpicture}[rotate=90]
    \draw  (0,-1.05) ellipse (1 and 1.8);
    \draw  (0,-0.3) ellipse (0.25 and 0.25);
    \draw  (0,-1.8) ellipse (0.25 and 0.25);
    \draw [red, very thick](0,0.75) -- (0,-0.05);
    \draw [red, very thick](0,0.75) .. controls (0,0.25) and (-0.5,0.2) .. (-0.5,-0.3) .. controls (-0.5,-0.8) and (0,-1.1) .. (0,-1.55);
    \draw [red, very thick](0,0.75) .. controls (0,0.25) and (0.5,0.2) .. (0.5,-0.3) .. controls (0.5,-0.8) and (0,-1.1) .. (0,-1.55);
    \draw [red, very thick](-0.5,-0.3) .. controls (-0.5,-1.05) and (0.5,-1.05) .. (0.5,-0.3);
    \node [fill, circle, inner sep=1.3pt, cyan] at (0,-1.55) {};
    \node [fill, circle, inner sep=1.3pt, cyan] at (0,0.75) {};
    \node [fill, circle, inner sep=1.3pt, cyan] at (0,-0.05) {};
    \draw [dashed](0,1) -- (0,-3.1);

    \draw  (-5.5,-5.05) ellipse (1 and 1.8);
    \draw  (-5.5,-4.3) ellipse (0.25 and 0.25);
    \draw  (-5.5,-5.8) ellipse (0.25 and 0.25);
    \draw [red, very thick](-5.5,-3.25) .. controls (-5.5,-3.75) and (-6.05,-3.8) .. (-6.05,-4.3) .. controls (-6.05,-4.75) and (-5.5,-5.1) .. (-5.5,-5.55);
    \draw [red, very thick](-5.5,-3.25) .. controls (-5.5,-3.75) and (-4.95,-3.8) .. (-4.95,-4.3) .. controls (-4.95,-4.75) and (-5.5,-5.1) .. (-5.5,-5.55);
    \draw [red, very thick](-5.5,-4.05) .. controls (-5.55,-3.75) and (-5.8,-3.75) .. (-6,-4.05);
    \draw [red, very thick](-5.5,-4.05) .. controls (-5.45,-3.75) and (-5.2,-3.75) .. (-5,-4.05);
    \draw [red, very thick](-6.05,3.6) .. controls (-6.05,3.1) and (-5.5,2.75) .. (-5.5,2.25);
    \draw [red, very thick](-4.95,3.6) .. controls (-4.95,3.1) and (-5.5,2.75) .. (-5.5,2.25);
    \draw [red, very thick](-6.05,3.7) .. controls (-6.05,2.75) and (-6.1,1.35) .. (-5.5,1.35) .. controls (-4.9,1.35) and (-4.95,2.75) .. (-4.95,3.7);
    
    \draw  (-5.5,2.75) ellipse (1 and 1.8);
    \draw  (-5.5,3.5) ellipse (0.25 and 0.25);
    \draw  (-5.5,2) ellipse (0.25 and 0.25);
    \draw [red, very thick](-5.5,4.55) -- (-5.5,3.75);
    \draw [very thick, red](-5.5,3.75) .. controls (-5.45,4.05) and (-4.95,4.05) .. (-4.95,3.7);
    \draw [red, very thick](-5.5,3.75) .. controls (-5.55,4.05) and (-6.05,4.05) .. (-6.05,3.6);

    \draw  (-3,-1.05) ellipse (1 and 1.8);
    \draw  (-3,-0.3) ellipse (0.25 and 0.25);
    \draw  (-3,-1.8) ellipse (0.25 and 0.25);
    \draw [red, very thick](-3,0.75) -- (-3,-0.05);
    \draw [red, very thick](-3,0.75) .. controls (-3,0.3) and (-3.5,0.25) .. (-3.5,-0.3) .. controls (-3.5,-0.8) and (-3,-1.1) .. (-3,-1.55);
    \draw [red, very thick](-3,0.75) .. controls (-3,0.3) and (-2.5,0.25) .. (-2.5,-0.3) .. controls (-2.5,-0.8) and (-3,-1.1) .. (-3,-1.55);
    \draw [red, very thick](-3,-0.05) .. controls (-3.05,0.3) and (-3.3,0.3) .. (-3.45,-0.05);
    \draw [red, very thick](-3,-0.05) .. controls (-2.95,0.3) and (-2.7,0.3) .. (-2.55,-0.05);
    \node [fill, circle, inner sep=1.3pt, cyan] at (-3,0.75) {};
    \node [fill, circle, inner sep=1.3pt, cyan] at (-3,-0.05) {};
    \node [fill, circle, inner sep=1.3pt, cyan] at (-3,-1.55) {};
    \draw [dashed](-3,1.05) -- (-3,-3.05);
    
    \draw [red, very thick](-6.05,3.6) .. controls (-6.05,3.1) and (-5.5,2.75) .. (-5.5,2.25);
    \draw [red, very thick](-4.95,3.6) .. controls (-4.95,3.1) and (-5.5,2.75) .. (-5.5,2.25);
    \draw [red, very thick](-6.05,3.7) .. controls (-6.05,2.75) and (-6.1,1.35) .. (-5.5,1.35) .. controls (-4.9,1.35) and (-4.95,2.75) .. (-4.95,3.7);
    \draw [red, very thick](-6,-4.05) .. controls (-5.8,-3.6) and (-5.2,-3.6) .. (-5,-4.05);

    \node [fill, circle, inner sep=1.3pt, cyan] at (-5.5,-3.25) {};
    \node [fill, circle, inner sep=1.3pt, cyan] at (-5.5,-4.05) {};
    \node [fill, circle, inner sep=1.3pt, cyan] at (-5.5,-5.55) {};
    \node [fill, circle, inner sep=1.3pt, cyan] at (-5.5,4.55) {};
    \node [fill, circle, inner sep=1.3pt, cyan] at (-5.5,3.75) {};
    \node [fill, circle, inner sep=1.3pt, cyan] at (-5.5,2.25) {};

    \draw [dashed](-5.5,4.85) -- (-5.5,0.75);
    \draw [dashed](-5.5,-2.95) -- (-5.5,-7.05);
\end{tikzpicture}
    \caption{$\iota$-invariant standard tracks in a pair of pants without outermost curve.}
    \label{fig:std_pants_iota_inv}
\end{figure}

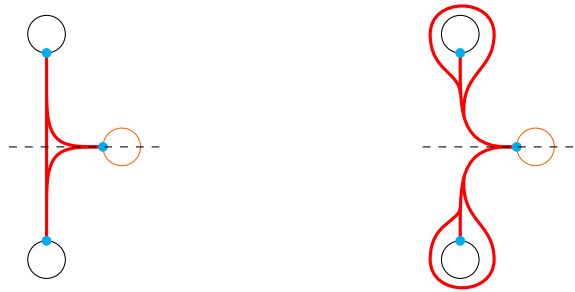
\begin{figure}[h]
    \centering
    \begin{tikzpicture}[rotate=90]
    \draw  (0,0.5) ellipse (0.25 and 0.25);
    \draw [myorange] (1.5,-0.5) ellipse (0.25 and 0.25);
    \draw  (3,0.5) ellipse (0.25 and 0.25);
    \draw  (0,-5) ellipse (0.25 and 0.25);
    \draw [myorange] (1.5,-6) ellipse (0.25 and 0.25);
    \draw  (3,-5) ellipse (0.25 and 0.25);
    \draw [very thick, red](0.8,0.5) .. controls (1.5,0.5) and (1.5,0.3) .. (1.5,-0.25);
    \draw [very thick, red](2.2,0.5) .. controls (1.5,0.5) and (1.5,0.3) .. (1.5,-0.25);
    \draw [very thick, red](0.25,-5) .. controls (1,-5) and (1.55,-4.95) .. (1.5,-5.75);
    \draw [very thick, red](1.5,-5.75) .. controls (1.45,-4.95) and (2,-5) .. (2.75,-5);
    \draw [very thick, red](2.75,0.5) -- (0.25,0.5);
    \node [fill, circle, cyan, inner sep=1.3] at (0.25,0.5) {};
    \node [fill, circle, cyan, inner sep=1.3] at (1.5,-0.25) {};
    \node [fill, circle, cyan, inner sep=1.3] at (2.75,0.5) {};
    \node [fill, circle, cyan, inner sep=1.3] at (0.25,-5) {};
    \node [fill, circle, cyan, inner sep=1.3] at (1.5,-5.75) {};
    \node [fill, circle, cyan, inner sep=1.3] at (2.75,-5) {};
    \draw [very thick, red](1.1,-5.05) .. controls (0.5,-5) and (0.5,-5.45) .. (0,-5.45) .. controls (-0.5,-5.45) and (-0.5,-4.6) .. (0,-4.6) .. controls (0.5,-4.6) and (0.5,-5) .. (0.7,-5);
    \draw [very thick, red](1.9,-5.05) .. controls (2.5,-5) and (2.5,-5.45) .. (3,-5.45) .. controls (3.5,-5.45) and (3.5,-4.6) .. (3,-4.6) .. controls (2.5,-4.6) and (2.5,-5) .. (2.2,-5);
    \draw [dashed](1.5,1) -- (1.5,-1);
    \draw [dashed](1.5,-4.5) -- (1.5,-6.5);
    \end{tikzpicture}
    \caption{$\iota$-invariant standard tracks in a pair of pants with a outermost curve (shown in orange).}
    \label{fig:std_pants_iota_inv_out}
\end{figure}

For an $\iota$-invariant train track $D\tau$ in $(D^\times \bD)^\circ$, the restriction $\tau := \bD^\circ \cap D\tau$ is a train track on $\bD^\circ$ if $D\tau$ has no edges transverse to the boundary components arising from the special points.
We call an $\iota$-invariant train track of this type a \emph{doubled} train track.
A doubled train track $D\tau$ in $(D^\times \bD)^\circ$ is \emph{standard} if it satisfies (the $\iota$-invariant version of) the same conditions as for the usual standard train tracks.





\begin{ex}
Let $\bD$ be the once-punctured disk with five special points on its boundary.
An $\iota$-invariant pants decomposition of $D^\times \bD$ (blue and orange) and an $\iota$-invariant standard train track $D\tau$ (red) are shown in the left of \cref{fig:std_5_sp_pts}. 
The restriction $\tau := D\tau \cap \bD(\partial_k)^\circ$ is shown in the right.
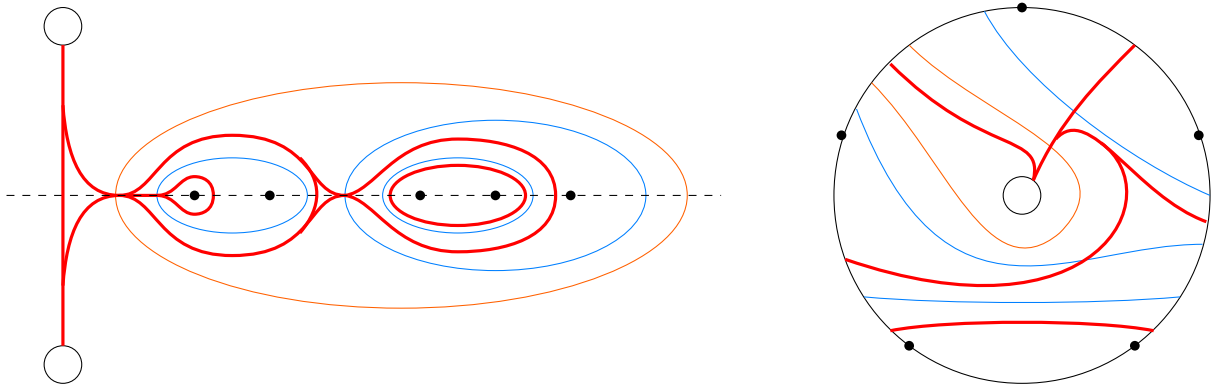
\begin{figure}[h]
    \centering
    \begin{tikzpicture}
    \begin{scope}[rotate=90, xshift=1cm]
    \draw  (-2.25,5.25) ellipse (0.25 and 0.25);
    \draw  (2.25,5.25) ellipse (0.25 and 0.25);
    \draw [myblue] (0,3) ellipse (0.5 and 1);
    \draw [myblue] (0,-0.5) ellipse (1 and 2);
    \draw [myblue] (0,0) node (v2) {} ellipse (0.5 and 1);
    \node [fill, circle, inner sep=1.3] at (0,3.5) {};
    \node [fill, circle, inner sep=1.3] at (0,2.5) {};
    \node [fill, circle, inner sep=1.3] at (0,0.5) {};
    \node [fill, circle, inner sep=1.3] at (0,-0.5) {};
    \node [fill, circle, inner sep=1.3] at (0,-1.5) {};
    \draw [myorange] (0,0.75) ellipse (1.5 and 3.8);
    \draw [very thick, red](-1.2,5.25) .. controls (-0.5,5.2) and (0,5) .. (0,4.5);
    \draw [very thick, red](1.2,5.25) .. controls (0.5,5.2) and (0,5) .. (0,4.5);
    \draw [very thick, red](0,4.5) .. controls (0,4) and (-0.8,4) .. (-0.8,3) .. controls (-0.8,1.9) and (0,2) .. (0,1.5);
    \draw [very thick, red](0,4.5) .. controls (0,4) and (0.8,4) .. (0.8,3) .. controls (0.8,1.9) and (0,2) .. (0,1.5);
    \draw [very thick, red](0,4.5) -- (0,4);
    \draw [very thick, red](0,4) .. controls (0,3.75) and (-0.25,3.7) .. (-0.25,3.5) .. controls (-0.25,3.4) and (-0.2,3.25) .. (0,3.25) .. controls (0.2,3.25) and (0.25,3.4) .. (0.25,3.5) .. controls (0.25,3.7) and (0,3.75) .. (0,4);
    \draw [very thick, red](-0.5,2.1) .. controls (-0.15,1.8) and (0.2,1.8) .. (0.5,2.1);
    \draw [very thick, red](0,1.5) .. controls (-0.05,1.1) and (-0.75,0.85) .. (-0.75,0) .. controls (-0.75,-0.65) and (-0.55,-1.3) .. (0,-1.3) .. controls (0.55,-1.3) and (0.75,-0.65) .. (0.75,0) .. controls (0.75,0.85) and (0.05,1.1) .. (0,1.5);
    \draw [very thick, red] (v2) ellipse (0.4 and 0.9);
    \draw [dashed](0,6) -- (0,-3.5);
    \draw [very thick, red](-2,5.25) -- (2,5.25);
    \end{scope}
    
    \draw (7.5,1) node (v1) {} ellipse (2.5 and 2.5);
    \node [fill, circle, inner sep=1.3] at (7.5,3.5) {};
    \node [fill, circle, inner sep=1.3] at (5.1,1.8) {};
    \node [fill, circle, inner sep=1.3] at (6,-1) {};
    \node [fill, circle, inner sep=1.3] at (9,-1) {};
    \node [fill, circle, inner sep=1.3] at (9.85,1.8) {};
    \draw (v1) ellipse (0.25 and 0.25);
    \draw [myorange](6,3) .. controls (7,2) and (9,1.5) .. (8,0.5) .. controls (7.05,-0.3) and (6.5,1.5) .. (5.5,2.5);
    \draw [myblue](7,3.45) .. controls (7.5,2.5) and (8.95,1.45) .. (10,1);
    \draw [myblue](5.3,2.15) .. controls (6.5,-1) and (8.15,0.35) .. (9.9,0.35);
    \draw [myblue](5.4,-0.35) .. controls (6.5,-0.45) and (8.5,-0.45) .. (9.6,-0.35);
    \draw [very thick, red](5.75,2.75) .. controls (7,1.5) and (7.8,1.85) .. (7.65,1.2);
    \draw [very thick, red](7.65,1.2) .. controls (8,2) and (8.5,2.5) .. (9,3);
    \draw [very thick, red](7.95,1.75) .. controls (8.4,2.25) and (9,1) .. (9.95,0.65);
    \draw [very thick, red](8.55,1.7) .. controls (9.5,1) and (8.5,-1) .. (5.15,0.15);
    \draw [very thick, red](5.75,-0.8) .. controls (6.5,-0.65) and (8.5,-0.65) .. (9.25,-0.8);
    \end{tikzpicture}
    \caption{Left: an $\iota$-invariant pants decomposition and a standard train track in the double $D^\times \bD$ of the once-punctured disk $\bD$ with five special points. Right: their restrictions to the disk $\bD$.}
    \label{fig:std_5_sp_pts}
\end{figure}
\end{ex}

Now let us return to the marked surface $\Sigma$.
For $k=1, \dots, b$, let $A_{\partial_k}$ be a tubular neighborhood of a curve in $\cR_\partial$ which is parallel to the $k$-th boundary component of $\Sigma$, and fix an orientation-preserving homeomorphism $\xi_{\partial_k}: A \to A_{\partial_k}$.
An $\iota$-invariant train track $D\tau$ in $(D^\times \Sigma)^\circ$ is \emph{standard} if the restrictions $D\tau \cap (\bar{\Sigma})^\circ$, $\xi^{-1}_{\partial_k} (D \tau \cap A_{\partial_k})$ and $D\tau \cap (D^\times \bD(\partial_k))^\circ$ are standard for all $k=1, \dots, k$.
Let $\widehat{V}(\tau)$ denote the set of $\iota$-invariant enhanced measures on an $\iota$-invariant standard train track $D\tau$ on $(D^\times \Sigma)^\circ$, and set
\begin{align*}
ST_{\widetilde{D\cR}} := \{ \tau \mid \tau = D\tau \cap \Sigma^\circ,\ \mbox{$D\tau$ is an $\iota$-invariant standard train track in $(D^\times \Sigma)^\circ$} \}.    
\end{align*}
Let $\psi_\tau$ denote the restriction of the map $\psi_{D\tau}$ to $\widehat{V}(\tau)$.
The following lemma can be shown in the same manner as \cref{thm:tt_cone_decomp}:

\begin{thm}\label{thm:tt_cone_decomp_gen}
Fix a complete hyperbolic structure $D^\times F$ of $D^\times \Sigma$.
Then the collection $\fF_{\widetilde{D\cR}} = \{\psi_\tau(\widehat{V}(\tau)) \mid \tau \in ST_{\widetilde{D\cR}}\}$ of cones gives a complete fan in $\eML(F)$.
\end{thm}

\section{The action of pseudo-Anosov mapping classes on the \texorpdfstring{$\cA$}{A}-laminations}\label{sec:A-dynamics}

A sign stability is the statement for a dynamics of an action of a mutation loop on a tropical cluster $\cX$-variety.
Although, apart from sign stability, we can consider an action on a tropical cluster $\cA$-variety.

Let $\Sigma$ be a marked surface.
Even if $\Sigma$ has not boundary components, the action of a pA mapping class on $\bS \cA_\Sigma(\bR^\trop)$ has not a NS dynamics since the action on the $\cA$-laminations which consists of only peripheral curves is finite order.
Although, the dynamics like as \cref{thm:pA_NS_general} holds for $\cA$-laminations:

\begin{prop}\label{prop:NS_A}
Let $\phi \in \Gamma_\Sigma$ be a pA mutation loop.
Then we have
\[
\lim_{n \to \infty} \phi^{\pm n}([L]) = [L_\phi^\pm] \in \bS \cU^\uf_\Sigma(\bR^\trop)
\]
for any $[L] \in \bS\cA_\Sigma(\bR^\trop) \setminus (p \circ \pi)^{-1}([L_{\pi(\phi)}^\mp])$.
Here, we consider $\cU^\uf_\Sigma(\bR^\trop)$ as the subspace of $\cA_\Sigma(\bR^\trop)$ without peripheral curves, $L^\pm_{\pi(\phi)}$ are the attracting/repelling points of $\pi(\phi)$ and $L_\phi^\pm \in \cU_\Sigma(\bR^\trop)$ are defined in \cref{thm:pA_NS_general}.
\end{prop}

One can prove it in the same manner as \cref{thm:pA_NS_general}.

This proposition is one of our reasons why to use $\cX$-variety for definition of sign stability.

\end{document}